\documentclass[12pt]{amsart}

\usepackage[latin1]{inputenc}
\usepackage{amsmath, amsfonts, amssymb, amsthm, graphicx, color}
\usepackage{epsfig}
\usepackage{fancyhdr}
\usepackage{framed}
\usepackage{floatrow}
\usepackage{mdframed}
\usepackage{url}
\usepackage{mathabx}
\usepackage{cfr-lm}

\theoremstyle{plain}
\newtheorem{lem}{Lemma}
\newtheorem{thm}{Theorem}

\newtheorem{Def}{Definition}
\newtheorem{cor}{Corollary}
\newtheorem{con}{Construction}

\theoremstyle{definition}
\newtheorem{ex}{Example}
\newtheorem{rem}{Remark}

\def \v {\mathrm{v}}
\def \h {\mathrm{h}}
\def \cS {\mathcal{S}}
\def \cG {\mathcal{G}}
\def \cL {\mathcal{L}}
\def \cT {\mathcal{T}}
\def \cR {\mathcal{R}}
\def \cE {\mathcal{E}}
\def \cB {\mathcal{B}}
\def \cV {\mathcal{V}}
\def \cM {\mathcal{M}}
\def \cC {\mathcal{C}}

\def \E {\mathcal{E}^{\rm ex}}
\def \L {\mathcal{L}_{\rm plane}}
\def \mL {\widetilde{\mathcal{L}}_{\rm plane}}
\def \S {\mathcal{S}_t}
\def \mS {\widetilde{\mathcal{S}}_t}
\def \D {\mathcal{D}_t}

\def\be{\begin{equation}}
\def\ee{\end{equation}}

\author{Yevgeniy Kovchegov}
\address{Department of Mathematics, Oregon State University, Corvallis, OR, USA}
\email{kovchegy@math.oregonstate.edu}

\author{Ilya Zaliapin}
\address{Department of Mathematics and Statistics, University of Nevada, Reno, NV, USA}
\email{zal@unr.edu}

\title[Dynamical pruning of rooted trees]{Dynamical pruning of rooted trees with applications to 1D ballistic annihilation}

\begin{document}

\maketitle	

\begin{abstract}
We introduce {\it generalized dynamical pruning} on rooted binary trees with edge lengths.
The pruning removes parts of a tree $T$, starting from the leaves, according to a pruning function defined on subtrees within $T$. 
The generalized pruning encompasses a number of discrete and continuous pruning operations, 
including the tree erasure and Horton pruning.
The main result is invariance of a finite critical binary Galton-Watson tree with exponential edge lengths with respect 
to the generalized dynamical pruning for an arbitrary admissible pruning function.
The second part of the paper examines the continuum 1-D ballistic annihilation model 
$A+A \rightarrow \zeroslash$ for a constant particle density and
initial velocity that alternates between the values of $\pm 1$. 
The model evolution is equivalent to a
generalized dynamical pruning of the shock tree that represents dynamics of sinks 
(points of particle annihilation), with the pruning function equal to the total tree length. 
The shock tree is isometric to the level set tree of the model potential (integral of velocity).
This equivalence allows us to construct a complete probabilistic description of the 
annihilation dynamics for the initial velocity that alternates between the values of $\pm$1 
at the epochs of a stationary Poisson process.
Finally, we discuss several real tree representations of the ballistic annihilation model,
closely connected to the shock wave tree. 
\end{abstract}

\tableofcontents

\section{Introduction}
\label{sec:intro}
Pruning of tree graphs is a natural operation that induces 
a contracting map \cite{KH97} on a suitable space of trees, with
the empty tree $\phi$ as the fixed point.
Examples of prunings studied in probability literature
include erasure from leaves at unit speed \cite{Neveu86,Evans2005,Winkel2012},
cutting the leaves \cite{BWW00,KZ17ahp,KZ18},
and eliminating nodes/edges at random \cite{AP98,AJH12}.
We consider here erasure of a tree from the leaves at a 
non-constant tree-dependent rate. 
Specifically, we introduce {\it generalized dynamical pruning} $\S(\varphi,T)$
of a rooted tree $T$ that eliminates all subtrees $\Delta_{x,T}$ 
(defined as the points descendant to point $x$ in $T$) for which
the value of a function $\varphi(\Delta_{x,T})$ is below $t$
(see Section~\ref{sec:pruning} for a formal definition).
The generalized dynamical pruning encompasses a number of
discrete and continuous pruning operations, depending on a choice of function 
$\varphi$.
For instance, the tree erasure from leaves at unit speed
\cite{Neveu86,Evans2005} corresponds to the pruning function $\varphi(T)$ equal to
the height of $T$; and the Horton pruning \cite{BWW00,KZ18}
corresponds to $\varphi(T)$ equal to the 
Horton-Strahler order of $T$.
For most selections of $\varphi(T)$, the map induced by the generalized dynamical pruning 
does not have a semigroup property, which distinguishes it from 
the operations studied in the literature.
Our main result (Section~\ref{sec:PI}, Theorem~\ref{main}) establishes invariance of the space
of finite critical binary Galton-Watson trees with i.i.d. exponential edge lengths
with respect to the generalized dynamical pruning, independently of
(an admissible) pruning function.
The invariance includes scaling of the edge lengths by the scaling constant 
equal to the survival 
probability ${\sf P}(\S(\varphi,T)\not{=}\phi)$.
The explicit form of the survival probability is established 
in Theorem~\ref{pdelta} for pruning by tree height (erasure from leaves at unit speed),
by Horton order, and by tree length.
The generalized prune invariance unifies several known invariance results (e.g., \cite{Neveu86,BWW00})
and suggests a framework for studying diverse problem-specific pruning operations.

\begin{figure}[t] 
\centering\includegraphics[width=0.9\textwidth]{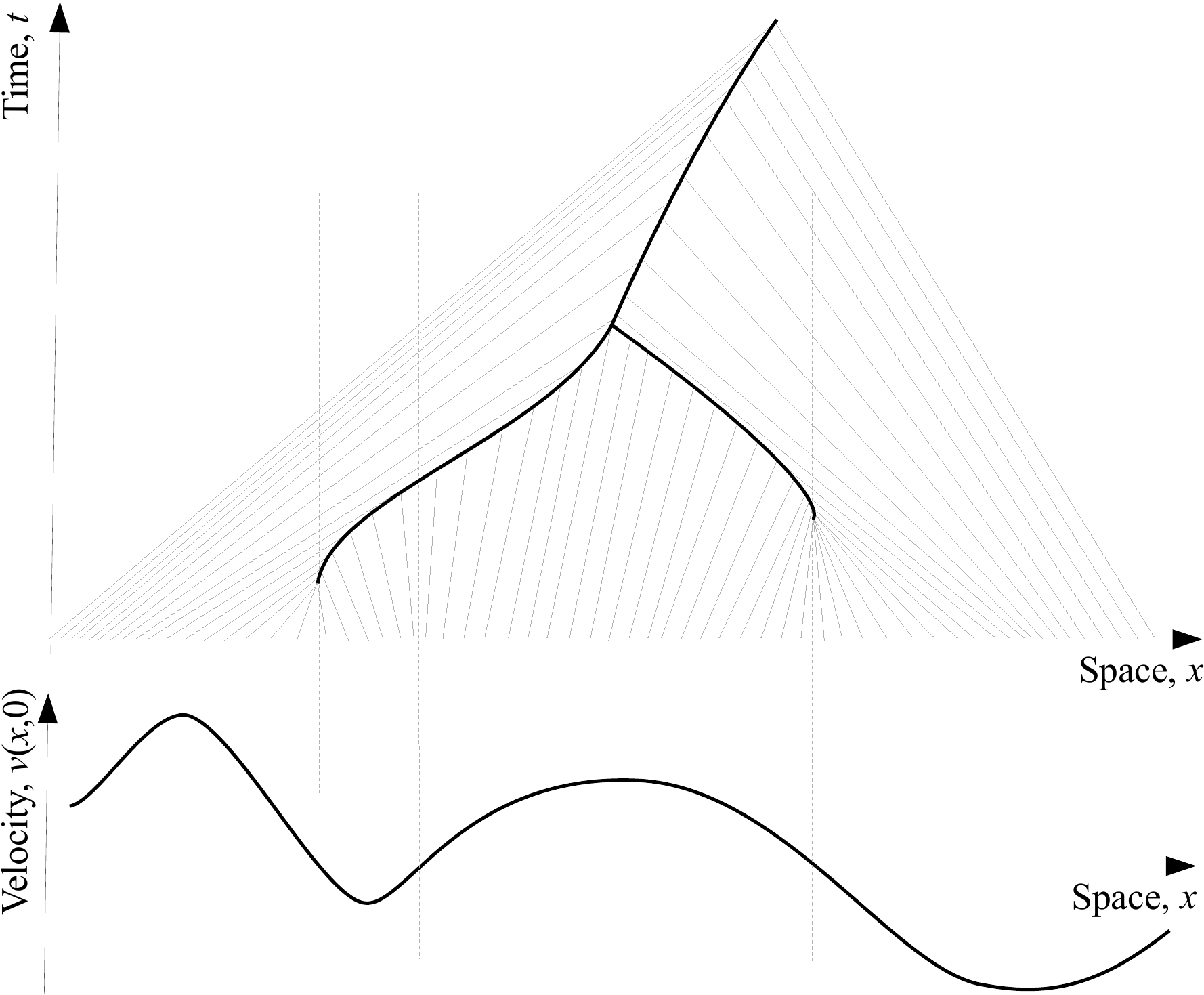}
\caption[Ballistic annihilation model: an illustration]
{Ballistic annihilation model: an illustration.
A particle with Lagrangian coordinate $x$ moves with velocity $v(x,0)$
until it collides with another particle and annihilates.
(Bottom panel): Initial velocity $v(x,0)$.
(Top panel): The space-time portrait of the system. 
The trajectories of selected particles are depicted by gray thin lines.
The shock wave that describes the motion and coalescence of sinks is shown 
by solid black line.
The sink trajectory forms an inverted Y-shaped tree. 
Vertical dashed lines show the time instants when the initial velocity changes sign. 
}
\label{fig:bam}
\end{figure} 

\subsection{Ballistic annihilation model}
As a notable application, we consider the 1-D ballistic annihilation model, 
traditionally denoted $A+A \rightarrow \zeroslash$.
This model describes the dynamics of particles on a real line:
a particle with Lagrangian coordinate $x$ moves with the velocity $v(x,0)$
until it collides with another particle, at which moment both particles annihilate,
hence the model notation.
The annihilation dynamics appears in chemical kinetics and bimolecular reactions; 
see \cite{EF85, BNRL93,Belitsky1995,Piasecki95,Droz95,BNRK96,Ermakov1998,Blythe2000,KRBN2010,Sidoravicius2017}.
The annihilation dynamics produces {\it sinks} ({\it shocks}) that correspond to the collisions 
of individual particles with consequent annihilation.
The moving {\it shock waves} represent the {\it sinks} that aggregate the annihilated particles and hence accumulate the mass of the media. 
Dynamics of these sinks resembles a coalescent process that generates a tree structure for their trajectories. 
The dynamics of a ballistic annihilation model with two coalescing sinks is illustrated in Fig.~\ref{fig:bam}.

\begin{figure}[t] 
\centering\includegraphics[width=0.9\textwidth]{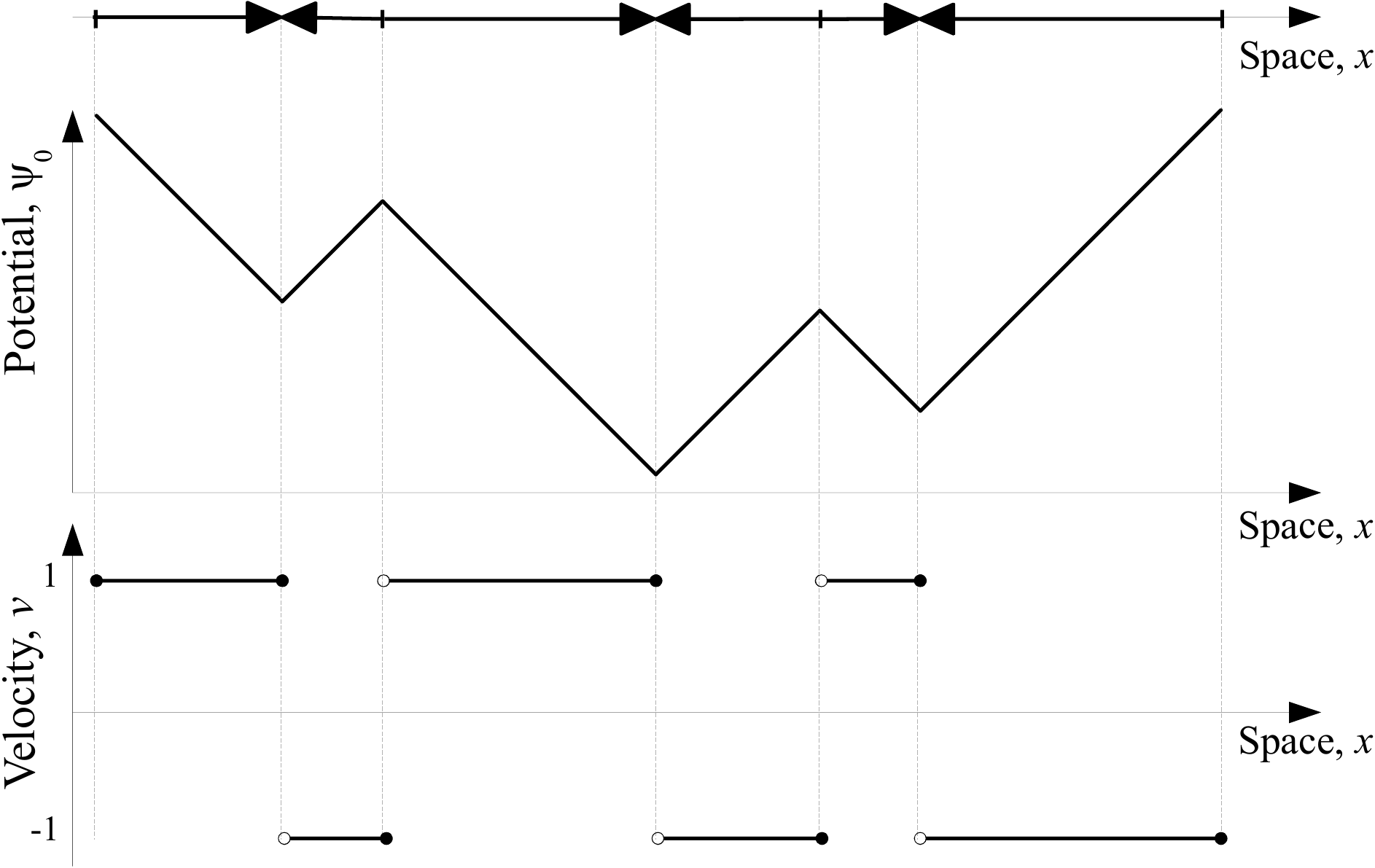}
\caption[Piece-wise linear unit slope potential]
{Piece-wise linear unit slope potential: an illustration.
(Top): Arrows indicate alternating directions of particle movement 
on an interval in $\mathbb{R}$.
(Middle): Potential $\Psi_0(x)$ is a piece-wise linear unit slope function.
(Bottom): Particle velocity alternates between values $\pm1$ within consecutive intervals.}
\label{fig:vel}
\end{figure} 

\subsection{Ballistic annihilation with two valued initial velocity}

\begin{figure}[t] 
\centering\includegraphics[width=0.9\textwidth]{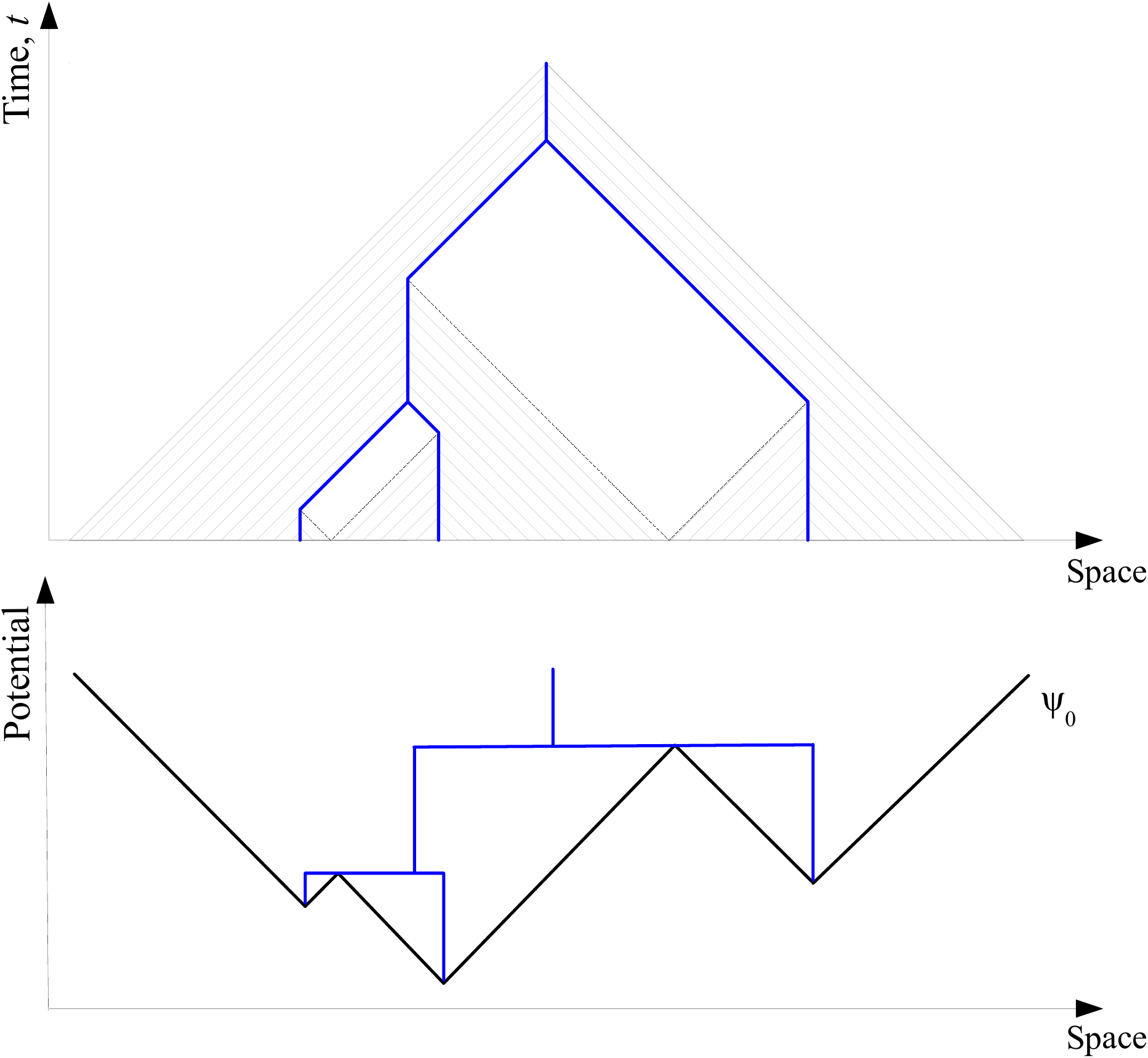}
\caption[Shock tree in a model with a unit slope potential: an illustration]
{Shock wave tree (sink tree) in a model with a unit slope potential: an illustration. 
(Top panel): Space-time dynamics of the system. Trajectories of particles
are illustrated by gray lines.
The sink trajectory (shock wave tree) is shown by blue line. 
Notice the appearance of empty regions (zero particle density) in the space-time domain.
(Bottom panel): Initial unit slope potential $\Psi_0(x)$ with three local 
minima (black line) and a graphical representation of the shock wave tree 
(blue line) in the phase space $(x,\psi(x,t))$.
}
\label{fig:shock_tree}
\end{figure} 

We consider here a model on a finite interval $[a,b]$ with a constant initial particle density 
$g(x)=g_0$ and an initial velocity field $v(x,0)$ that alternates between the values 
$\pm 1$, as illustrated in Fig.~\ref{fig:vel}. 
Equivalently, we work with potential velocity field
$v(x,t) = -\partial_x\psi(x,t)$ where the initial potential 
$\Psi_0(x) = \psi(x,0)$ is a piece-wise linear continuous function with slopes $\pm 1$.
We furthermore assume that $\Psi_0(x)$ is a negative excursion on $[a,b]$. 
This choice corresponds to a particularly tractable structure
of the shock wave tree, which is completely described in this work.
The dynamics of a system with a simple unit slope potential
is illustrated in Fig.~\ref{fig:shock_tree}.
Prior to collision, the particles move at unit speed either to the left or to the right,
so their trajectories in the $(x,t)$ space are given by lines with 
slope $\pm 1$ (Fig.~\ref{fig:shock_tree}, top panel, gray lines).
The sinks appear at $t=0$ at the local minima
of the potential $\psi(x,0)$. 
These minima correspond to the points 
whose right neighborhood moves to the left and left neighborhood
moves to the right with unit speed, hence immediately creating a sink.
The sinks move and merge to create a shock wave tree, shown
in blue in Fig.~\ref{fig:shock_tree}. 
Importantly, for our particular choice of initial potential, 
the combinatorial structure and planar embedding of the shock tree coincides 
with that of the level set tree $T=\textsc{level}(\psi(x,0))$ of the initial potential 
(Section~\ref{sec:solution}, Theorem~\ref{thm:SWT}).
The bottom panel of Fig.~\ref{fig:shock_tree} illustrates a particularly
useful embedding of the shock wave tree into the phase space $(x,\psi(x,t))$
of the system; this embedding is discussed in detail in Sect.~\ref{sec:solution}.

\begin{figure}[t] 
\centering\includegraphics[width=0.9\textwidth]{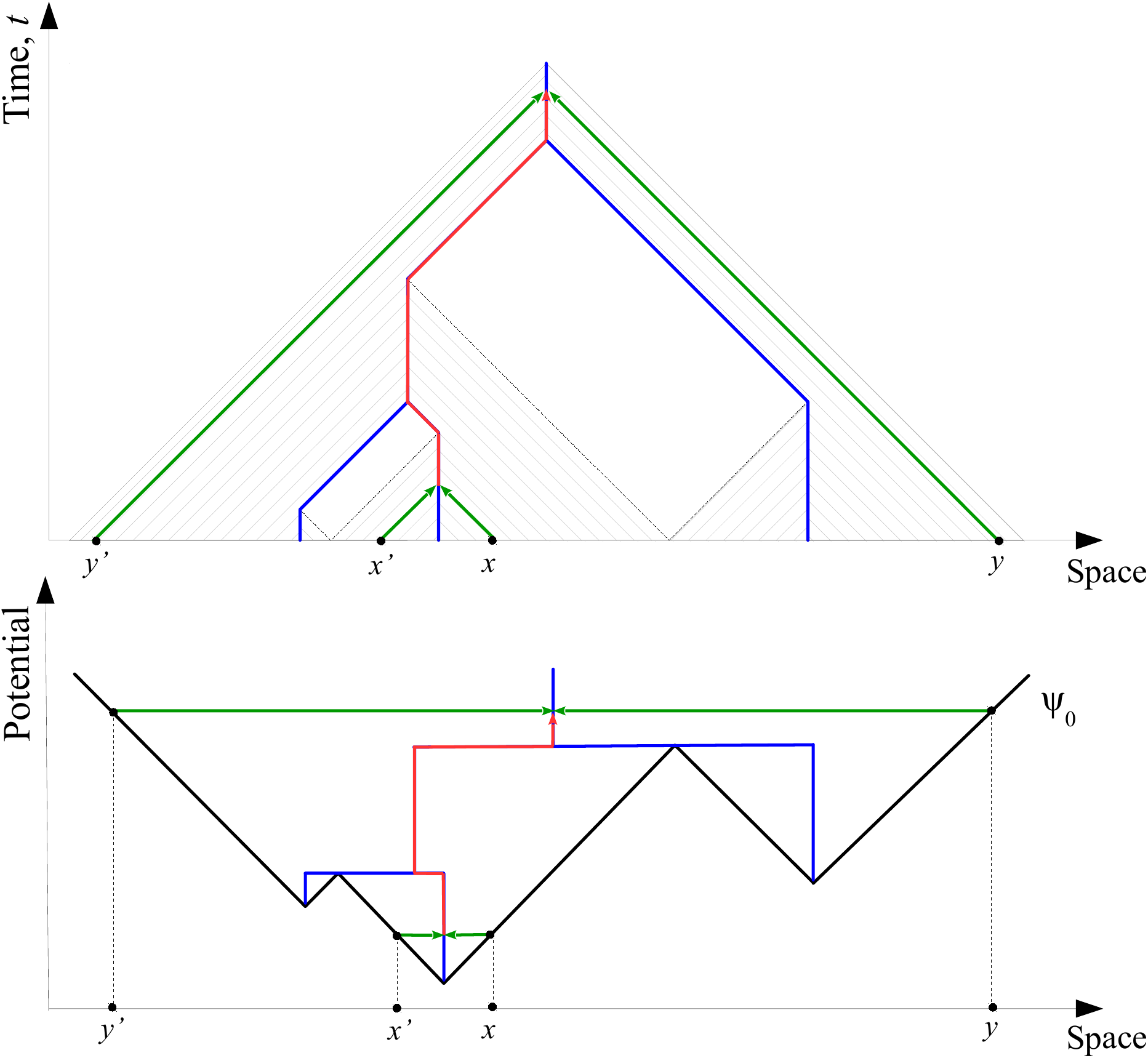}
\caption[Real tree representation of a model with a unit slope potential: an illustration]
{$\mathbb{R}$-tree representation of a ballistic annihilation model with a unit slope 
potential: an illustration. 
Figure illustrates dynamics of four points, $x,x',y,$ and $y'$, marked in the horizontal space axis.
The pairs of points $\{x,x'\}$ and $\{y,y'\}$ collide and annihilate with each 
other.
Green arrows correspond to ballistic runs of points $x,x',y,y'$,
and hence to leaves of tree $\mathbb{T}(\Psi_0)$.
Red line corresponds to the trajectory of points $x,x'$ after their 
collision, within a sink.
The rest of notations are the same as in Fig.~\ref{fig:shock_tree}.
}
\label{fig:Rtree}
\end{figure} 

\subsection{Ballistic annihilation as dynamical pruning}
The main application result of our work (Section~\ref{sec:Bdyn}, Theorem~\ref{thm:pruning})
states that the ballistic annihilation dynamics in case of a unit slope 
potential is equivalent to the
generalized dynamical pruning of the shock wave tree with pruning 
function $\varphi(\tau)$ equal to the total length of $\tau$.
The pruned tree in this construction describes the potential
restricted to the domain of particles that did not annihilate until instant $t$.
To retain information about sinks and empty intervals, 
we equip a tree with {\it massive points}, placed at the tree {\it cuts} --
the boundary of the pruned tree parts (Section~\ref{sec:Bdyn}, Definition~\ref{def:cuts}).
A complete description of ballistic annihilation dynamics is then given in terms of
mass-equipped trees, which involves a suitably modified 
definition of pruning (Section~\ref{sec:Bdyn}).
In particular, we establish a one-to-one correspondence between
pruned mass-equipped trees and time-advanced potentials $\psi(x,t)$ 
with massive sinks (Section~\ref{sec:Bdyn}, Constructions~\ref{con1},
\ref{con2}).
Theorem~\ref{thm:annihilation} describes the ballistic annihilation dynamics 
for the initial velocity field
that alternates between $\pm1$ at epochs of a stationary Poisson point process 
on $\mathbb{R}$.
The respective potential corresponds to the Harris path of a critical binary 
Galton-Watson tree with i.i.d. exponential edge lengths. 
This equivalence allows one to use a suit of results available for the 
exponential Galton-Watson tree to study the ballistic annihilation;
in particular, this
connects the ballistic annihilation dynamics with the invariance results 
of Theorems~\ref{main}, \ref{pdelta}.
We use this connection to derive the time-dependent mass distribution of a 
random sink in an infinite potential (Section~\ref{sec:rand_mass}, Theorem~\ref{thm:mass}).

\subsection{Real tree representation of ballistic annihilation}
The applied part of this work examines the shock wave tree of ballistic annihilation,
which is a finite tree with edge lengths considered as a metric space. 
Section~\ref{sec:Rtree} discussed a natural extension of this construction
to {\it real trees} (or $\mathbb{R}$-trees) that are tightly connected to the shock wave tree
and possess key information about the annihilation dynamics.

Recall that an $\mathbb{R}$-tree is a generalization of 
the concept of a finite tree with edge lengths to infinite spaces \cite{Evans2005};
see Sect.~\ref{sec:Rsetup} for a formal setup.
We construct (Sect.~\ref{sec:Rfull}) an $\mathbb{R}$-tree $\mathbb{T}=\mathbb{T}(\Psi_0)$ that describes the 
entire model dynamics as coalescence of particles
and sinks; this tree is sketched by gray lines in the top panel of Fig.~\ref{fig:shock_tree}. 
Specifically, the tree consists of points $(x,t)$  such that there exist either 
a particle or a sink with coordinate $x$ at time $t$.
There is one-to-one correspondence between the initial particles $(x,0)$ 
and leaf vertices of $\mathbb{T}$.
Each leaf edge of $\mathbb{T}$ corresponds (one-to-one) to
the free (ballistic) run of a corresponding particle before annihilating in a sink.
Four of such free runs are depicted by green arrows in Fig.~\ref{fig:Rtree}.
The shock wave tree (movement and coalescence of sinks) corresponds to the non-leaf
part of the tree $\mathbb{T}$; it is shown by blue lines in Figs.~\ref{fig:shock_tree},
\ref{fig:Rtree}.
We adopt a convention that the motion of a particle consists of
two parts: an initial ballistic run at unit speed, and subsequent motion within 
a respective sink.
For example, the within-sink motion of particles $x$ and $x'$ is shown by red line in Fig.~\ref{fig:Rtree}.
This interpretation extends motion of all particles to the same 
time interval $[0, t_{\rm max}]$, with $t_{\rm max}$ being the 
time of appearance of the final sink that accumulates the total mass
on the initial interval.
This final sink serves as the tree root.
Section~\ref{sec:Rfull} introduces a proper metric on this space 
so that the model is represented by a time oriented rooted $\mathbb{R}$-tree.
In particular, the metric induced by this tree on 
the initial particles $(x,0)$ becomes an {\it ultrametric},
with the distance between any two particles equal to
the time until their collision (as particles or as respective sinks).

Section~\ref{sec:Rdomain} discusses two metric space representations
of the system's domain $[a,b]$, 
one is an $\mathbb{R}$-tree and the other is not, that
describe the ballistic annihilation dynamics and are readily constructed
from the initial potential $\Psi_0(x)$. 
One of these spaces, which is an $\mathbb{R}$-tree, 
establishes an equivalence between the pairs of points that collide
with each other, like the pairs $(x,x')$ and $(y,y')$ in Fig.~\ref{fig:Rtree}.
This tree is isometric to the level set tree $\textsc{level}(\Psi_0)$ of the 
initial potential that is used in this work to describe the shock wave tree
(Theorem~\ref{thm:SWT}); it is known in the literature 
as a {\it tree in continuous path} \cite[Definition 7.6]{Pitman},\cite[Example 3.14]{Evans2005}.

The tree metrics and prunings considered in this work are connected to the
dynamics of ballistic annihilation with particular initial conditions. 
In Sect.~\ref{sec:Rprune} we briefly discuss a natural way of introducing 
alternative prunings on $\mathbb{R}$-trees and show that a typical pruning does not have 
the semigroup property.

\medskip

The rest of the paper is organized as follows.
The generalized dynamical pruning is introduced in Section~\ref{pruning}. 
Section~\ref{sec:GWT} collects necessary results on level set trees
and proves the invariance theorems for critical binary Galton-Watson trees with i.i.d. exponential edge lengths.
The shock wave tree for the dynamics of 1-D ballistic annihilation $A+A \rightarrow \zeroslash$ with piece-wise
unit slope potential is analyzed in Section~\ref{sec:annihilation}.
Section~\ref{sec:Rtree} discusses a real tree representation of ballistic annihilation.
Sections~\ref{sec:BDEE},\ref{sec:rand_mass} examine a unit slope 
potential with exponential segments 
durations (Poisson epoch velocity alterations), for a finite and infinite 
domain, respectively.
Section~\ref{sec:discussion} concludes.

\begin{figure}[t] 
\centering\includegraphics[width=0.9\textwidth]{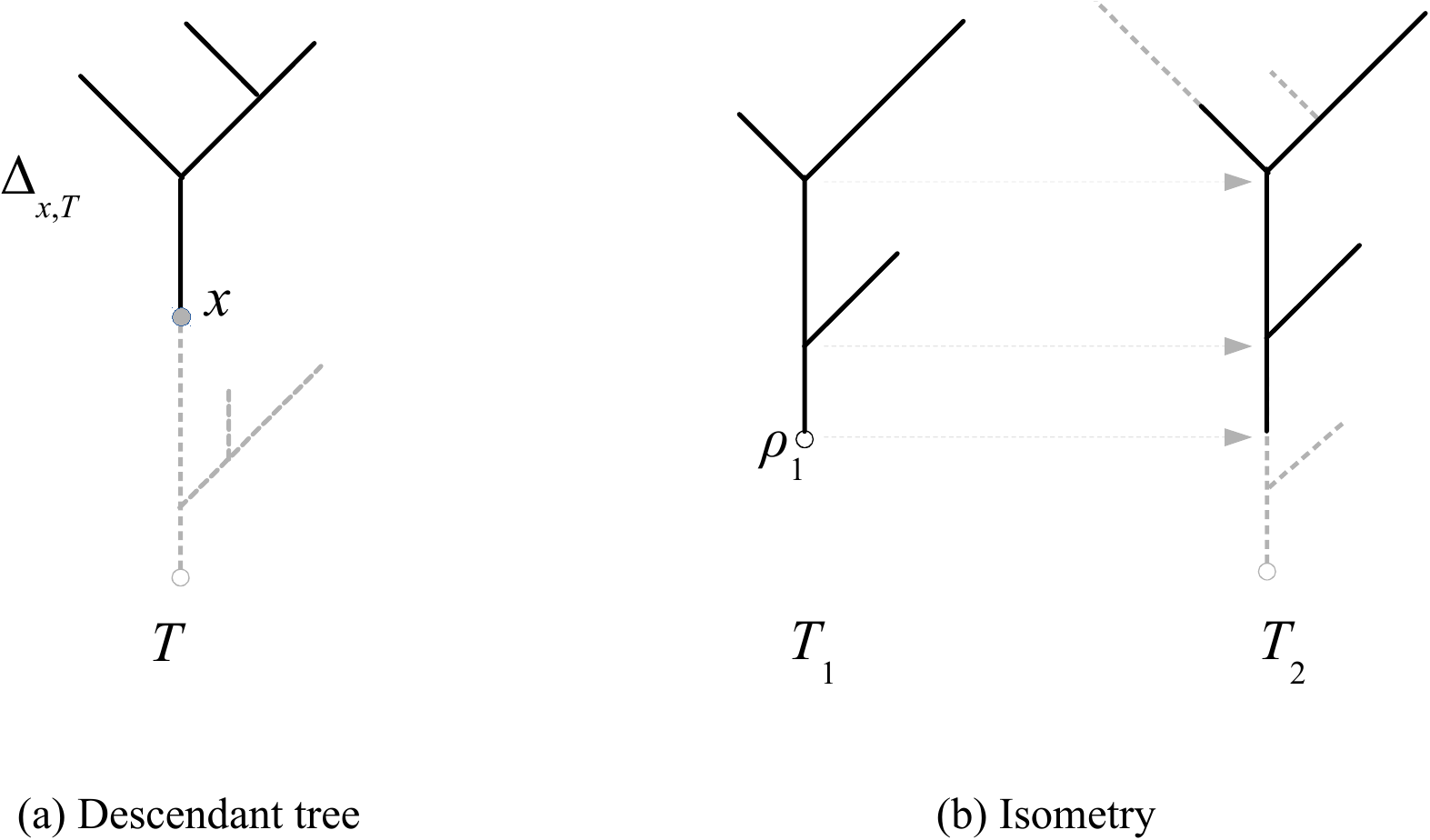}
\caption[Descendant subtree and isometry: an illustration.]
{\small Descendant subtree and tree isometry: an illustration.
(a) Subtree $\Delta_{x,T}$ (solid black lines) descendant to a point $x$ (gray circle)
in a tree $T$ (union of dashed gray and soling black lines).
(b) Isometry of trees. Tree $T_1$ (left) is mapped to tree $T_2$ (right).
The image of $T_1$ within $T_2$ is shown by black lines, the rest of $T_2$
is shown by dashed gray lines. 
Here, tree $T_1$ is less than tree $T_2$, $T_1 \preceq T_2.$}
\label{fig:isometry}
\end{figure}

\section{Generalized dynamical pruning}\label{pruning} 
\subsection{Trees}
Consider a space $\L$ of finite unlabeled rooted reduced binary trees with edge 
lengths and planar embedding. 
The space includes the {\it empty tree} $\phi$ comprised of a
root vertex and no edges.
A binary tree is called {\it rooted} if one of its vertices of degree 1 or 2 is selected 
as the tree root $\rho$.
The existence of a root vertex imposes the parent-offspring relation between each pair of connected vertices in a tree $T\in\L$: the one closest to the root is called {\it parent}, 
and the other -- {\it offspring}.
The tree root is the only vertex that does not have a parent.
Formally, a binary tree $T=\rho\cup\{v_i,e_i\}_{1\le i \le \#T}$ is comprised of the root $\rho$
and a collection of non-root vertices $v_i$, each of which is connected to its 
unique parent $v_{{\sf parent}(i)}$
by the parental edge $e_i$, $1\le i \le \#T$.
Here $\#T$ denotes the number of non-root vertices, equal to the number of edges, in a tree $T$.
Unless indicated otherwise, the vertices are indexed in order of the depth-first search, starting from the root.
A tree is called {\it reduced} if it has no vertices of degree 2, with the 
root as the only possible exception.
The operation of {\it series reduction} removes each degree-two non-root vertex 
by merging its adjacent edges into one and adding the respective lengths.  
Planar embedding is equivalent to introducing a relative orientation (right/left)
for every pair of siblings.

A non-empty rooted tree is called {\it planted} if its root has degree 1;
in this case the only edge connected to the root is called {\it stem}.
Otherwise the root has degree 2 and a tree is called {\it stemless}.
We denote by $\L^{|}$ and $\L^{\vee}$ the subspaces of planted and stemless
trees, respectively.
Hence $\L=\L^{|}\cup\L^{\vee}$ and $\L^{|}\cap\L^{\vee}=\{\phi\}$. 
Fig.~\ref{fig:planted} shows examples of a planted and a stemless tree.
Most of discussion in this work refers to planted trees.

Let $l_T=(l_1,\dots,l_{\#T})$ be the vector of edge lengths.
The length of a tree $T$ is the sum of the lengths of its edges:
\[\textsc{length}(T) = \sum_{i=1}^{\#T} l_i.\]
A tree $T\in\L$ is naturally equipped with a length metric $d(x,y)$ for points $x,y \in T$.
The distance $d(x,y)$ equals the length of the minimal path within $T$ between $x$ and $y$.
The height of a tree $T$ is the maximal distance between the root
and a vertex:
\[\textsc{height}(T) = \max_{1\le i \le \#T} d(v_i,\rho).\]

Sometimes we focus on the combinatorial tree ${\textsc{shape}(T)}$, 
which retains the branching structure of $T$ while omitting its edge lengths and embedding.
Similarly, ${\textsc{p-shape}(T)}$ retains the branching structure of $T$ and planar embedding,
and omits the edge length information.
The space of finite unlabelled rooted reduced binary planted combinatorial (planar) 
trees is denoted by $\cT$ ($\cT_{\rm plane}$).

\begin{figure}[t] 
\centering\includegraphics[width=0.7\textwidth]{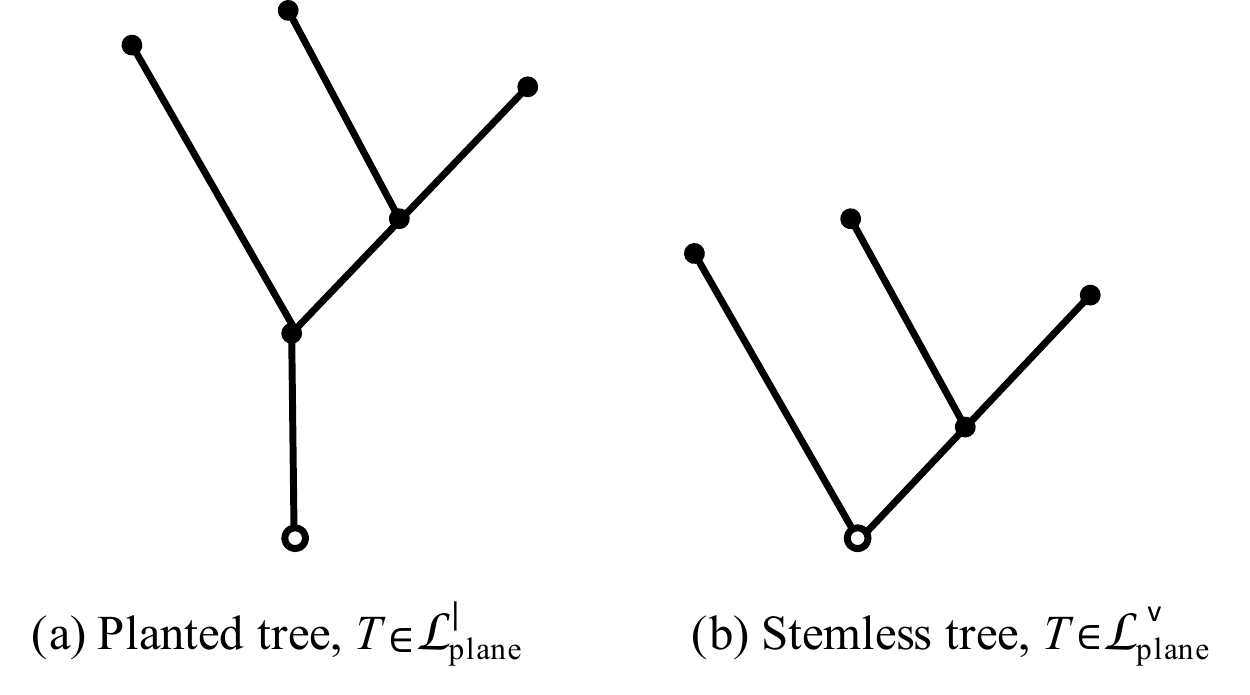}
\caption[Planted and stemless trees.]
{\small Examples of planted (a) and stemless (b) trees.}
\label{fig:planted}
\end{figure}

\subsection{Generalized dynamical pruning}
\label{sec:pruning}
Given a tree $T \in \L$ and a point $x \in T$, let $\Delta_{x,T}$ be the {\it descendant tree} 
of $x$: it is comprised of all points of $T$ descendant to $x$, including $x$; see Fig.~\ref{fig:isometry}a. 
Then $\Delta_{x,T}$ is itself a tree in $\L$ with root at $x$. 
Let $(T_1,d)$ and $(T_2,d)$ be two metric rooted trees, and let $\rho_1$ denote the root of $T_1$. 
A function $f: (T_1,d) \rightarrow (T_2,d)$ is said to be an {\it isometry} if 
${\sf Image}[f] \subseteq \Delta_{f(\rho_1),T_2}$ and for all pairs $x,y \in T_1$,
$$d\big(f(x),f(y)\big)=d(x,y).$$
The tree isometry is illustrated in Fig.~\ref{fig:isometry}b. 
We use the isometry to define a {\it partial order} in the space $\L$ as follows.  
We say that $T_1$ is {\it less than or equal to} $T_2$ and write $T_1 
\preceq T_2$ if 
and only if there is an isometry $f: (T_1,d) \rightarrow (T_2,d)$. 
The relation $\preceq$ is a partial order as it satisfies the 
reflexivity, antisymmetry, and transitivity conditions. 
Moreover, a variety of  other properties of this partial order can be 
observed, including order denseness and semi-continuity. 

We say that a function $\varphi:\L \rightarrow \mathbb{R}$ is {\it monotone non-decreasing} 
with respect to the partial order $\preceq $ if
$\varphi(T_1) \leq \varphi(T_2)$ whenever $T_1 \preceq T_2.$
Consider a monotone non-decreasing function $\varphi:\L \rightarrow \mathbb{R}^+$. 
We define the {\it generalized dynamical pruning} operator $\S(\varphi,T):\L\rightarrow\L$ induced by $\varphi$ 
at any $t\ge 0$ as
$$\S(\varphi,T):=\rho\cup\Big\{x \in T\setminus\rho ~:~\varphi\big(\Delta_{x,T}\big)\geq t \Big\}.$$
Informally, the operator $\S$ cuts all subtrees $\Delta_{x,T}$ for which the value of $\varphi$
is below threshold $t$, and always keeps the tree root.
Extending the partial order to $\L$
by assuming $\phi \preceq T$ for all $T \in \L$, we observe for any $T\in\L$ that 
$S_s(T) \preceq S_t(T)$ whenever $s \geq t$.

The dynamical pruning operator $\S$ encompasses and 
unifies a range of problems, depending on a choice of $\varphi$,
as we illustrate in the following examples.

\begin{ex}[{\bf Tree height}]
\label{ex:height}
Let the function $\varphi(T)$ equal the height of tree $T$:
\begin{equation}\label{phi_hight}
\varphi(T) = \textsc{height}(T).
\end{equation}
In this case the operator $\S$ satisfies {\bf continuous semigroup property}:
$$\cS_t\circ\cS_s=\cS_{t+s} ~\text{ for any }~t,s\ge 0.$$ 
It coincides with the continuous pruning (tree erasure) studied in Neveu \cite{Neveu86}, who
established invariance of a critical and sub-critical binary
Galton-Watson processes with i.i.d. exponential edge lengths with respect to this operation.

It is readily seen that for a coalescent process, the dynamical pruning $\S$ of the corresponding coalescent tree with $\varphi(T)$ as in (\ref{phi_hight}) 
replicates the coalescent process.
\end{ex}

\begin{ex}[{\bf Horton-Strahler order}]
\label{ex:H}
Let the function $\varphi(T)+1$ equal the Horton-Strahler order ${\sf k}(T)$ of a tree $T$:
\begin{equation}\label{phi_Horton}
\varphi(T) = {\sf k}(T)-1.
\end{equation}
The Horton-Strahler order \cite{Pec95,BWW00,KZ16} 
is closely related to the operation $\cR$ of leaf pruning
with consecutive series reduction in a planted rooted tree,
This operation is known as {\it Horton pruning}; it is illustrated in Fig.~\ref{fig:HST}.
The pruning induces a contracting map on $\L$. 
The trajectory of each tree $T$ under $\cR(\cdot)$ is uniquely
determined and finite:
\be\label{TRR0}
T\equiv\cR^0(T)\to \cR^1(T) \to\dots\to\cR^k(T)=\phi,
\ee
with the empty tree $\phi$ as the (only) fixed point \cite{KZ18}.
The Horton-Strahler order ${\sf k}(T)$ of a planted tree from $\L^{|}$ 
is the minimal number of prunings necessary to eliminate a tree $T$.
The Horton-Strahler order ${\sf k}(T)$ of an unplanted tree from $\L^{\vee}$
is the minimal number of prunings necessary to eliminate a tree $T$ plus one.
The Horton-Strahler order is also known as the {\it register number} 
\cite{FRV79}, as it equals the minimum number of memory registers 
necessary to evaluate an arithmetic expression described by a tree $T$.

With the choice \eqref{phi_Horton} the dynamical pruning operator coincides with the
Horton pruning:
$\S=\mathcal{R}^{\lfloor t \rfloor}$. 
It is readily seen that $\S$ satisfies {\bf discrete semigroup property}:
$$\cS_t\circ\cS_s=\cS_{t+s} ~\text{ for any }~t,s\in \mathbb{N}_0.$$ 

It has been shown in \cite{BWW00} that a critical binary Galton-Watson tree
is invariant with respect to Horton pruning; moreover, this is the only prune-invariant distribution 
from the Galton-Watson family. 
Study \cite{KZ18} introduced a one-parameter family of prune-invariant trees, which includes
critical binary Galton-Watson distribution as a special case.
The Horton prune invariance is also empirically found in multiple
observed and modeled systems (e.g., \cite{Pec95,ZK12,ZZF13}). 

\begin{figure}[h] 
\centering\includegraphics[width=0.9\textwidth]{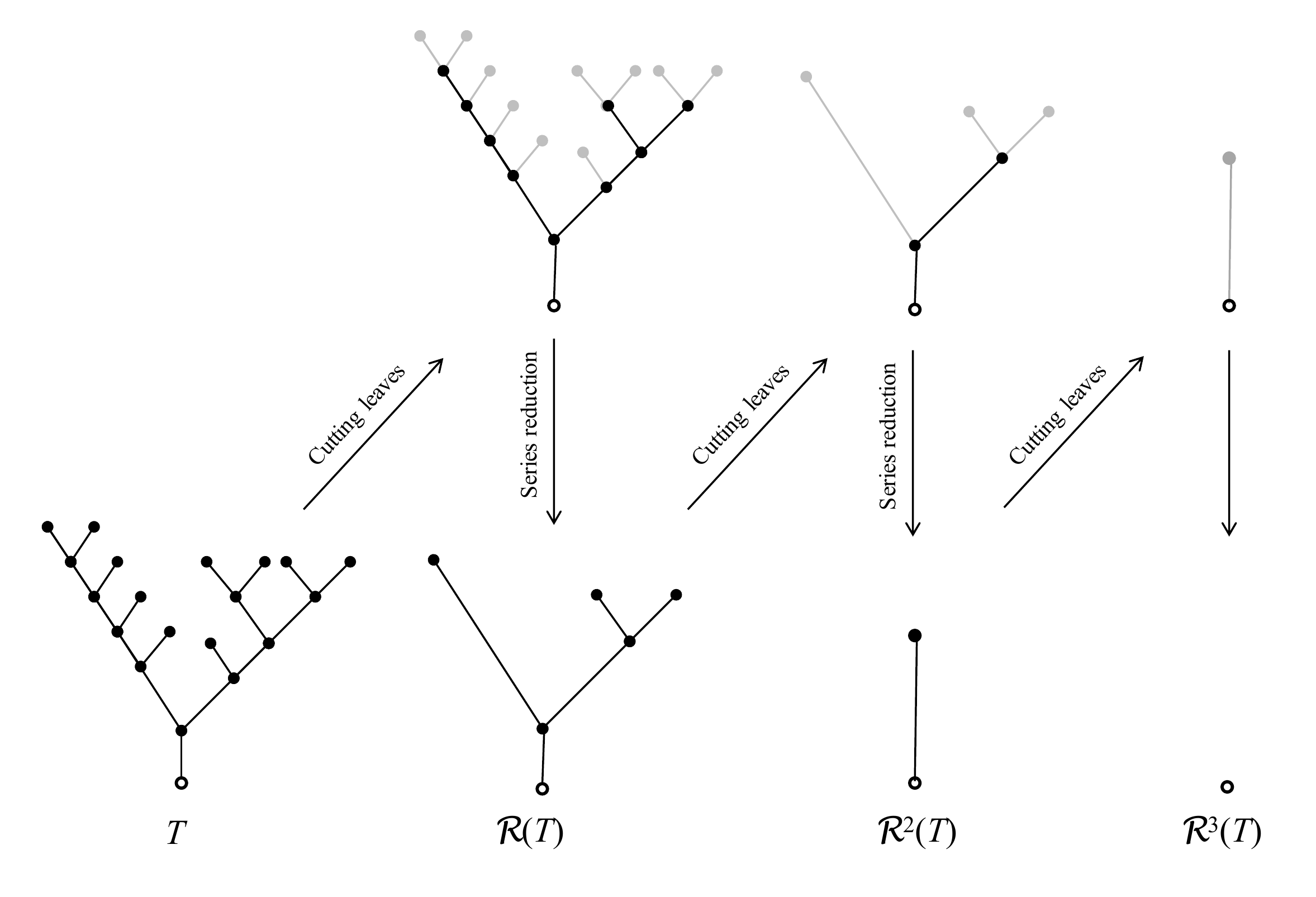}
\caption[Example of Horton-Strahler indexing]
{\small Horton pruning and Horton-Strahler ordering: an example.
The order of the tree is ${\sf k}(T)=3$, since the tree $T$ is eliminated in three prunings.
Each pruning consists of cutting leaves (top row) and consecutive series reduction (bottom row).
The pruning trajectory $T \to \cR(T) \to \cR^2(T) \to \cR^3(T)=\phi$ is shown in the bottom row of panels.
}
\label{fig:HST}
\end{figure}

A widespread empirical constraint related to the Horton-Strahler orders is so-called
{\it Horton law} -- a geometric decay of the number of branches of a given order in a
finite tree; see \cite{Kirchner93,Pec95,Tarboton96,DR99,KZ18a} and references therein.
Sufficient conditions for the Horton law in an asymptotically increasing tree 
were found in \cite{KZ16}, hence providing 
rigorous foundations for the celebrated regularity that has escaped a formal 
explanation for a long time.
A weak form of Horton law was proved for Kingman coalescent and the level set tree 
of a sequence of i.i.d. random variables \cite{KZ17ahp}. 
\end{ex}

\begin{ex}[{\bf Total tree length}]
\label{ex:L}
Let the function $\varphi(T)$ equal the total lengths of $T$:
\begin{equation}\label{phi_length}
\varphi(T) = \textsc{length}(T).
\end{equation}
The dynamical pruning by the tree length is illustrated in Fig.~\ref{fig:LP} for
a Y-shaped tree that consists of three edges.

Importantly, in this case $\S$ {\bf does not satisfy} the semigroup property.
To see this, consider an internal vertex point $x \in T$ (see Fig.~\ref{fig:LP},
where the only internal vertex is marked by a gray ball). 
Then $\Delta_{x,T}$ consists of point $x$ as its root, 
the left subtree of length $a$ and the right subtree of length $b$. 
Observe that the whole left subtree is pruned away by time $a$, and the whole right 
subtree is pruned away by time $b$.
However, since $$\varphi(\Delta_{x,T}) = \textsc{length}(\Delta_{x,T})=a+b,$$
the junction point $x$ will not be pruned until time instant $a+b$. 
Thus, $x$ will be a leaf of  $\S(\varphi,T)$ for all $t$ such that
$$\max\{a,b\} \leq t \leq a+b.$$
This situation corresponds to Stage IV in Fig.~\ref{fig:LP}, where
each of the left and right subtrees stemming from point $x$, marked by
a gray ball, consists of a single edge.

\begin{figure}[t] 
\centering\includegraphics[width=0.8\textwidth]{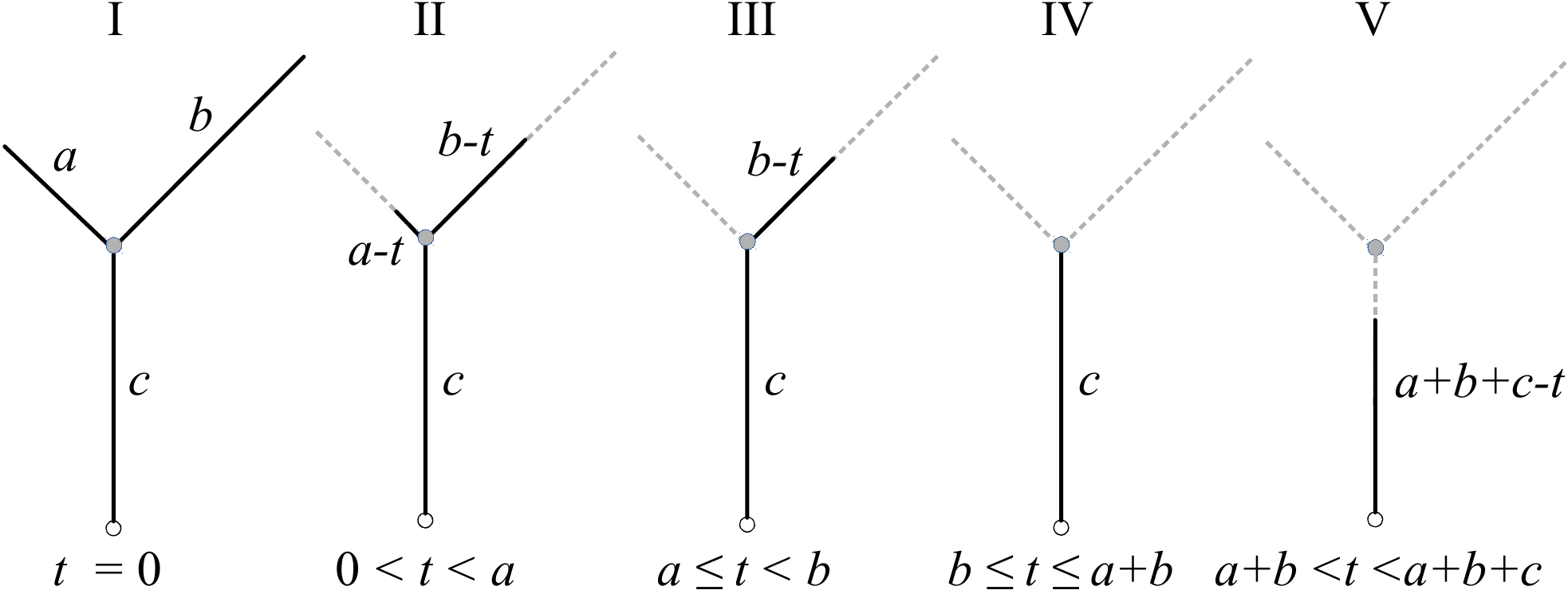}
\caption[Pruning by length: an illustration]
{\small Pruning by tree length: an illustration. 
Figure shows five generic stages in the dynamical pruning of a Y-shaped tree $T$,
with pruning function $\varphi(T) = \textsc{length}(T)$.
The pruned tree $\S$ is shown by solid black lines; the pruned parts of the initial 
tree are shown by dashed gray lines.\\
{\bf Stage I}: Initial tree $T$ consists of three edges, with lengths $a,b,c$
indicated in the panel; without loss of generality we assume $a<b$.\\
{\bf Stage II}: For any $t<a$ the pruned tree $\S$
has a Y-shaped form with leaf edges truncated by $t$. \\
{\bf Stage III}: For any $a\le t < b$ the pruned tree $\S$
consists of a single edge of length $c+b-t$.\\
{\bf Stage IV}: For any $b\le t \le a+b$ the pruned tree $\S$
consists of a single edge of length $c$. Notice that during this stage 
the tree $\S$ does not change with $t$; this loss of memory 
causes the process to violate the semigroup property.\\
{\bf Stage V}: For any $a+b<t<a+b+c$ the pruned tree $\S$
consists of a single edge of length $a+b+c-t$.}
\label{fig:LP}
\end{figure}

The semigroup property in this example can be introduced by considering 
{\bf mass-equipped trees}.
Informally, we replace each pruned subtree $\tau$ of $T$ with a point of mass equal
to the total length of $\tau$.
The information encoded in the massive points allows one to reconstruct
some of the information lost during the pruning process, and hence establish
the semigroup property.
Specifically, by time $a$, the pruned away left subtree turns into a massive point of mass $a$ 
attached to $x$ on the left side. 
Similarly, by time $b$, the pruned away right subtree turns into a massive point of mass $b$ 
attached to $x$ on the right side. 
For $\max\{a,b\} \leq t \leq a+b$, the construction keeps truck of the quantity $a+b-t$ 
associated with point $x$, and when the quantity $a+b-t$ decreases to $0$, the two massive points coalesce into one. 
If a single massive point seats at a leaf, its mass is $t$. 
If a double massive point seats at a leaf, further pruning of the leaf's parental edge is prevented 
until the two massive points coalesce.
Keeping track of all such quantities makes $\S$ satisfy the continuous semigroup property.
This construction is formally introduced in Section~\ref{sec:annihilation}.

\medskip
\noindent
Notably, in this case the pruning operator $\S$ coincides with 
the potential dynamics of continuum mechanics formulation of the 1-D ballistic annihilation model,
$A+A \rightarrow \zeroslash$, as discussed 
below in Section~\ref{sec:annihilation}. 
\end{ex}

\begin{ex}[{\bf Number of leaves}]
\label{ex:numL}
Let the function $\varphi(T)$ equal the number of leaves in a tree $T$.
This choice is closely related to the mass-conditioned dynamics of
an aggregation process. 
Specifically, consider $N$ singletons (particles with unit mass)
that appear in a system at instants $t_n\ge 0$, $1\le n\le N$. 
The existing clusters merge into consecutively larger clusters by 
pair-wise mergers. 
The cluster mass is additive: a merger of two clusters of masses
$i$ and $j$ results in a cluster of mass $i+j$.
We consider a time-oriented tree $T$ that describes this process. 
The tree $T$ has $N$ leaves and $(N-1)$ internal vertices.
Each leaf corresponds to an initial particle, each internal 
vertex corresponds to a merger of two clusters, and the edge lengths represent
times between the respective mergers. 
The action of $\S$ on such a tree coincides with a conditional aggregation process
state that only considers clusters of mass $\ge t$. 
A well-studied special case is a coalescent process with a kernel $K(i,j)$,
where all particles appear at instant $t=0$ and each pair of clusters with masses $i,j$ merges with 
intensity proportional to $K(i,j)=K(j,i)$, independently of all other pairs.
\end{ex}

\subsection{Pruning for $\mathbb{R}$-trees}
The generalized dynamical pruning introduced in Sect.~\ref{sec:pruning}
is readily applied to non-binary and {\it real trees} (see Sect.~\ref{sec:Rsetup}
for definitions), although this is not
the focus of our work. 
We notice that the total tree length (Example~\ref{ex:L}) and number of leaves 
(Example~\ref{ex:numL}) might be undefined (infinite) for an $\mathbb{R}$-tree. 
We introduce in Sect.~\ref{sec:Rprune} a {\it mass} function that can serve as 
a natural general analog of these and other finite tree functions. 
We show, in particular, that pruning my mass is equivalent to the pruning
by the total tree lengths in the particular situation of ballistic
annihilation model discussed in this paper.
Accordingly, our results are not limited to finite trees and should be straightforwardly
extended to $\mathbb{R}$-trees that appear, for instance, with other initial potentials.

\subsection{Relation to other generalizations of pruning}
After completing the initial stage of this work, we have learned that a 
pruning operation similar in spirit to our generalized dynamical pruning 
(defined in Subsection \ref{sec:pruning}) 
was considered in a preprint by Duquesne and Winkel \cite{Winkel2012} that extended a formalism by Evans \cite{Evans2005} and Evans et al.~\cite{EPW06}. 
We notice that the two definitions of pruning, ours in Subsection \ref{sec:pruning} and that in \cite{Winkel2012}, are principally different, despite their similar appearance.  
In essence, the work \cite{Winkel2012} assumes the Borel measurability with respect to the Gromov-Hausdorff metric (\cite{Winkel2012}, Section 2), which implies the semigroup property of the respective pruning (\cite{Winkel2012}, Lemma 3.11).
On the contrary, the generalized dynamical pruning of this study may have the semigroup property only 
under very particular choices of $\varphi(T)$; see Examples 1 and 2 above. 
The majority of natural choices of $\varphi(T)$, including the tree length $\varphi(T) = \textsc{length}(T)$
(Example 3) or the number of leaves in a tree (Example 4), do not have the semigroup property, and hence
are not covered by the pruning of \cite{Winkel2012}.
The main application results of this work (Section~\ref{sec:annihilation}) refer to the pruning by $\varphi(T)=\textsc{length}(T)$ 
that has no semigroup property.

Curiously, for the above two examples with no semigroup property, i.e., when $\varphi(T) = \textsc{length}(T)$ and when $\varphi(T)$ equals the number of leaves in $T$, the following discontinuity property holds with respect to the Gromov-Hausdorff metric  $d_{\sf GH}$ defined in \cite{Evans2005,EPW06,Winkel2012}. For any $\epsilon>0$ and any $M>0$, there exist trees $T$ and $T'$ in $\L$ such that 
$$|\varphi(T) -\varphi(T')|>M ~~\text{ while }~~ d_{\sf GH}(T,T') <\epsilon .$$
Indeed, if $\varphi(T) = \textsc{length}(T)$, we consider a tree $T$ with the number of leaves exceeding $M/\epsilon$, and let $T'$ be the tree obtained from $T$ by elongating each of its leaves by $\epsilon$. Similarly, if  $\varphi(T)$ is the number of leaves in $T$, we construct $T'$ from $T$ by attaching at least $M/\epsilon$ new leaves, each of length $\epsilon$.

\subsection{Invariance with respect to generalized dynamical pruning}
\label{sec:pi}
Consider a tree $T\in\L$ with edge lengths given
by a positive vector $l_T=(l_1,\dots,l_{\#T})$.
The edge length vector $l_T$ can be specified by
distribution $\chi(\cdot)$ of a point $x_T=(x_1,\dots,x_{\#T})$ on the standard simplex
\[\Delta^{\#T}=\left\{x_i:\sum_i^{\#T} x_i = 1, 0<x_i\le 1\right\},\] 
and conditional distribution $F(\cdot|x_T)$ of the tree length
$\textsc{length}(T)$, so that
\[l_T = x_T\times\textsc{length}(T).\]
Accordingly, a tree $T$ can be completely specified by its planar shape,
a vector of proportional edge lengths, and the total tree length:
\[T=\left\{\textsc{p-shape}(T),x_T,\textsc{length}(T)\right\}.\]
A measure $\eta$ on $\L$ is a joint distribution of 
these three components:
\[\eta(T\in\{\tau,d\bar x,d\ell\}) = 
\mu(\tau)\times \chi_{\tau}(d\bar x)\times  F_{\tau,\bar x}(d\ell),\]
where the tree planar shape is specified by
\[\mu(\tau)={\sf Law }\left(\textsc{p-shape}(T)=\tau\right),\quad \tau\in\cT_{\rm plane},\]
the relative edge lengths is specified by 
\[\chi_\tau(\bar x) 
={\sf Law }\left(x_T=\bar{x} \,|\,\textsc{p-shape}(T)=\tau \right),\quad \bar x\in \Delta^{\#T},\]
and the total tree length is specified by
\[ F_{\tau, \bar x} (\ell) ={\sf Law }\left(\textsc{length}(T)=\ell \,|\,x_T=\bar{x}, ~ \textsc{p-shape}(T)=\tau \right),\quad \ell\ge0.\]
Let us fix $t\ge 0$ and a function $\varphi:\L \rightarrow \mathbb{R}$ that is monotone 
non decreasing with respect to the partial order $\preceq$.
We denote by $\S^{-1}(\varphi,T)$ the preimage of a tree $T\in\L$ under the generalized dynamical
pruning:
\[\S^{-1}(\varphi,T)=\{\tau\in\L:\S(\varphi,\tau)=T\}.\]

\noindent Consider the distribution of edge lengths induced by the pruning:
$$\Xi_\tau(\bar x) 
={\sf Law }\left(x_{\tilde T}=\bar{x}  \,|\,\textsc{p-shape}\big(\tilde T\big)=\tau \right)$$
and
$$\Phi_{\tau, \bar x} (\ell)
={\sf Law }\left(\textsc{length}\big(\tilde T\big)=\ell \,|\,x_{\tilde T}=\bar{x}, ~\textsc{p-shape}\big(\tilde T\big)=\tau \right),$$
where the notation $\tilde T:=\S(\varphi,T)$ is used for brevity. 

\begin{Def}[{{\bf Prune invariance}}]\label{def:pi}
Let fix $t\ge 0$ and a function $\varphi:\L \rightarrow \mathbb{R}^+$ that is 
monotone non decreasing with respect to the partial order $\preceq$.
We call a measure $\eta$ on $\L$ {\it invariant with respect to pruning $\S(\varphi,\cdot)$} 
(or simply prune invariant) if the following conditions hold 
\begin{itemize}
\item[(i)]
The measure is prune invariant in planar shapes.
This means that for $\nu=\mu\circ\S^{-1}(\varphi,\cdot)$ we have
\[\mu(\tau)=\nu(\tau|\tau\ne\phi).\]
\item[(ii)]
The measure is prune invariant in edge lengths.
This means that for any 
combinatorial planar tree $\tau\in\cT_{\rm plane}$
\[\Xi_\tau(\bar x)=\chi_\tau(\bar x)\]
and there exists a {\it scaling exponent} $\zeta\equiv \zeta(\varphi,t)>0$ such that for any 
 relative edge length vector
$\bar x\in\Delta^{\#\tau}$ we have 
\[\Phi_{\tau, \bar x} (\ell)=\zeta^{-1}F_{\tau, \bar x} \left(\frac{\ell}{\zeta}\right).\]
\end{itemize}
\end{Def}

This definition unifies multiple invariance properties examined in the literature.
For example, the classical work by Neveu~\cite{Neveu86} establishes prune invariance 
of critical Galton-Watson trees with i.i.d. exponential edge lengths with respect to
tree erasure from leaves at a unit rate, which is equivalent to the
generalized dynamical pruning with function $\varphi(T)=\textsc{height}(T)$ 
(see Example~\ref{ex:height}).
Prune invariance with respect to the Horton pruning (see Example~\ref{ex:H}
and Fig.~\ref{fig:HST}) has been established by Burd et al.~\cite{BWW00}
for the combinatorial binary critical Galton-Watson trees with no edge lengths.
Duquesne and Winkel \cite{Winkel2012} established prune-invariance of 
critical Galton-Watson trees with i.i.d. exponential edge lengths with
respect to so-called {\it hereditary property}, which includes the
tree erasure of Example~\ref{ex:height} and Horton pruning of Example~\ref{ex:H}.
The authors recently introduced a one-parameter family of 
{\it critical Tokunaga trees} \cite{KZ18} that are prune invariant with 
respect to the Horton pruning; this model includes the critical binary Galton-Watson tree
with i.i.d. exponential lengths as a special case.
Section~\ref{sec:PI} below establishes prune invariance of critical binary Galton-Watson trees
with i.i.d. exponential edge lengths with respect to arbitrary generalized pruning.

\section{Dynamical pruning for exponential critical Galton-Watson trees}
\label{sec:GWT}
We review here the results on tree representation of continuous functions, 
following \cite{LeGall93,NP,Pitman,ZK12,KZ18}.

\subsection{Harris path}
\label{sec:Harris}
The {\it Harris path} of a tree $T\in\L$ 
is defined as a piece-wise linear function \cite{Harris,Pitman} 
\[H_T(t)\,:\,[0,2\cdot\textsc{length}(T)]\to\mathbb{R}\]
that equals the distance from the root traveled along the tree $T$ 
in the depth-first search, as illustrated in Fig.~\ref{fig:Harris}.
For a tree $T$ with $n$ leaves, the Harris path 
$H_T(t)$ is a piece-wise linear positive excursion
that consists of $2n$ linear segments with alternating slopes $\pm 1$, see \cite{Pitman}.

\begin{figure}[h] 
\centering\includegraphics[width=0.8\textwidth]{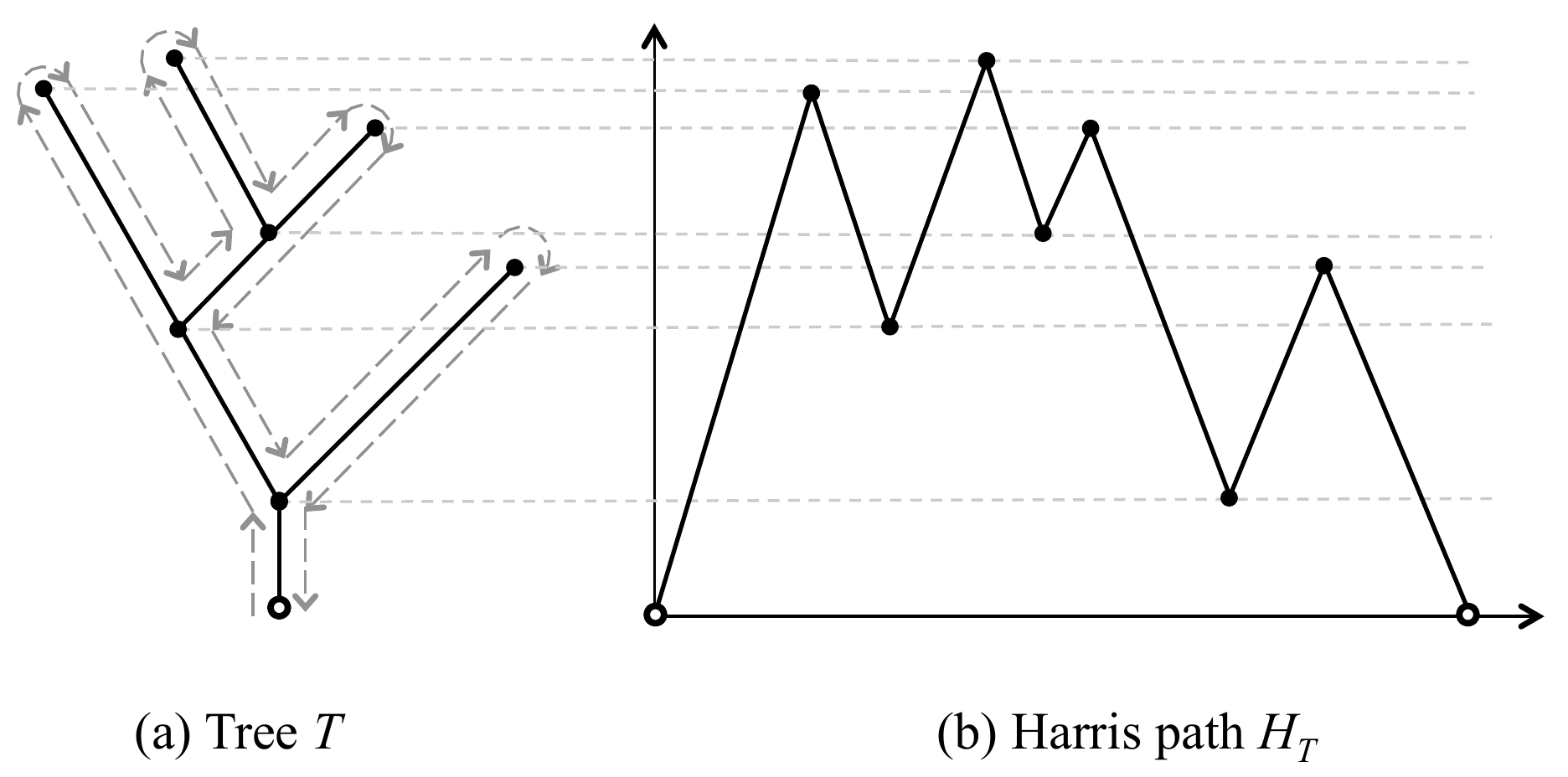}
\caption[Harris path]
{(a) Tree $T$ and its depth-first search illustrated by dashed arrows.
(b) Harris path $H_T(t)$ for the tree $T$ of panel (a). }
\label{fig:Harris}
\end{figure} 

\subsection{Level set tree}
\label{sec:level}
Consider a continuous function $X_t$, $t\in[a,b]$ with a finite number of
local minima.
The level set $\mathcal{L}_{\alpha}\left(X_t\right)$ is defined 
as the pre-image of the function values above $\alpha$: 
\[\mathcal{L}_{\alpha}\left(X_t\right) = \{t\,:\,X_t\ge\alpha\}.\]
The level set $\mathcal{L}_{\alpha}$ for each $\alpha$ is
a union of non-overlapping intervals; we write 
$|\mathcal{L}_{\alpha}|$ for their number.
Notice that 
$|\mathcal{L}_{\alpha}| = |\mathcal{L}_{\beta}|$ 
as soon as the interval $[\alpha,\,\beta]$ does not contain a value of
local extrema of $X_t$; and  
$0\le |\mathcal{L}_{\alpha}| \le n$, where $n$ is the number 
of the local maxima of $X_t$.
As the threshold $\alpha$ decreases, the new intervals appear
and the existing intervals merge. 
The {\it level set tree} $\textsc{level}(X_t)\in\cL_{\rm plane}$ is a 
tree that describes the topology of the level sets $\mathcal{L}_{\alpha}$ 
as a function of threshold $\alpha$, as illustrated in Fig.~\ref{fig:LST}.

\begin{figure}[h] 
\centering\includegraphics[width=0.8\textwidth]{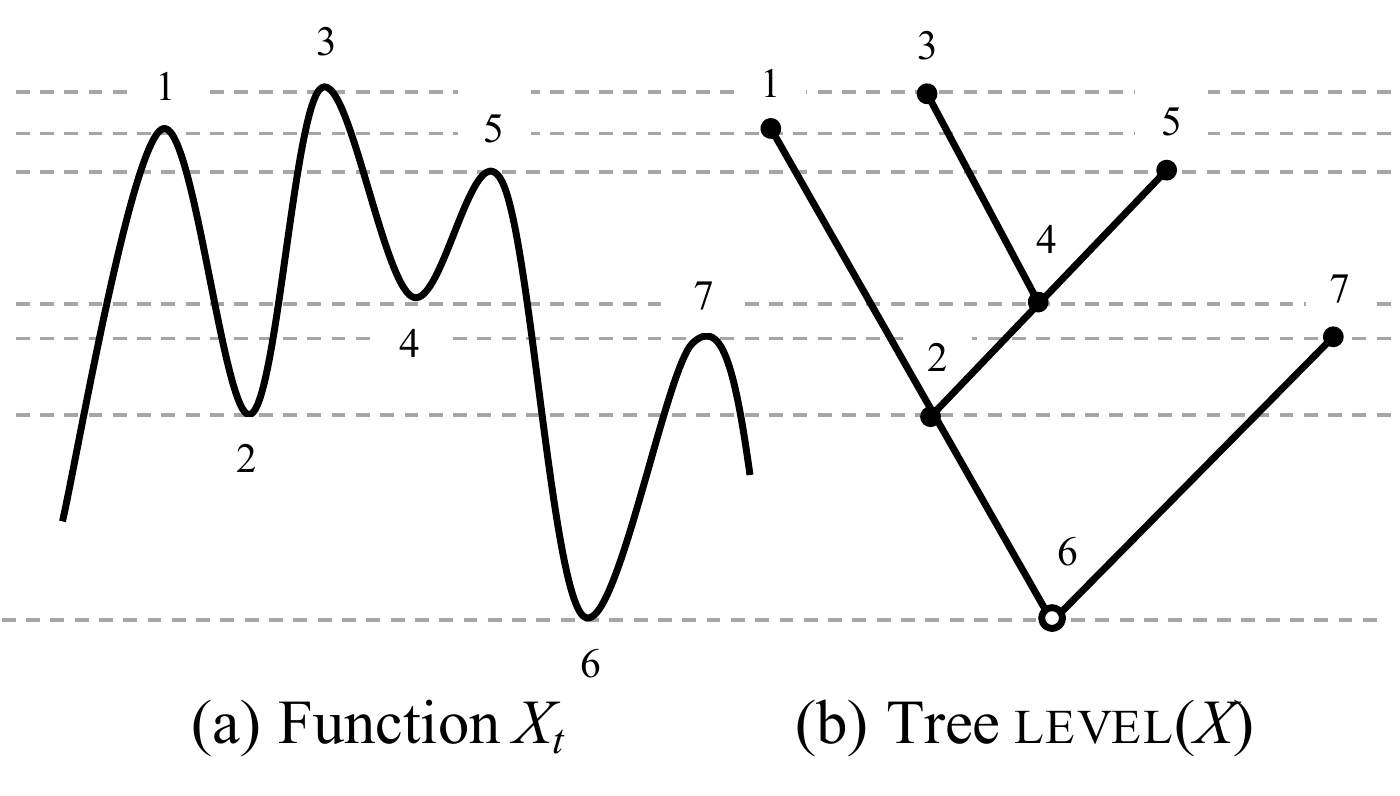}
\caption{Function $X_t$ (panel a) with a finite number of local
extrema and its level-set tree $\textsc{level}(X)$ (panel b).
The local extrema of $X_t$ and the respective vertices of $\textsc{level}(X_t)$
are labelled by numbers 1 to 7.
Local maxima (odd numbers) correspond to leaves,
and local minima (even numbers) correspond to internal vertices of $\textsc{level}(X_t)$.
The global minima of $X_t$ coincides with one of the internal local minima (label 6),
so the level set tree is unplanted (root has degree 2). 
}
\label{fig:LST}
\end{figure}

Next we formally specify the combinatorial and metric structure of 
the level set tree, as well as its planar embedding.

{\bf Combinatorial structure of} $\textsc{level}(X_t)$.
Assume that there exist $n$ local maxima of $X_t$,
including possible local maxima at the boundaries of the interval $[a,b]$.
Then there exist $(n-1)$ internal minima, excluding possible local minima
at the boundaries.
Denote by $t_1<t_2<\dots<t_{2n-1}$ the ordered times of these local extrema. 
The level set tree $\textsc{level}(X_t)$ has $2n-1$ vertices: 
$n$ leaves that correspond (one-to-one) to the local maxima of $X_t$,
and 
$n^{\rm int}\le n-1$ non-root internal vertices that correspond (one-to-one) 
to the {\it chains} of equally-valued local minima of $X_t$.
Formally, we say that two consecutive local minima at $t_i$ and $t_{i+2}$ belong to the 
same chain if $X_{t_i}=X_{t_{i+2}}$.

The tree root corresponds to the global minimum of $X_t$ on $[a,b]$.
If the global minimum is reached at the boundary ($t=a$ or $t=b$), then the 
root has degree 1 and the level set tree is planted;
in this case the vertex that corresponds to the lowest
local minimum within $(a,b)$ is connected to the root.
Otherwise, when the global minimum coincides with one of the local minima, 
the root corresponds to that local minimum, it has degree 2 and the 
level set tree is unplanted. 
This situation is illustrated in Fig.~\ref{fig:LST}.

For every vertex $i$, except the one that corresponds to the lowest internal minimum of $X_t$
and was discussed above,
the parental vertex ${\sf parent}(i)$ 
corresponds to the maximal of the two local minima adjacent to and below $X_{t_i}$.
Formally, let $s_j$, $1\le j\le n-1$, denote the times of local minima within $(a,b)$.
For every local extrema $i$, except the lowest internal minimum, we define 
its right and left lower adjacent local minima: 
\[r_i = \min_{1\le j\le n-1}\{s_j: s_j>t_i \text{ and }X_{s_j}<X_{t_i}\},\]
\[l_i = \max_{1\le j\le n-1}\{s_j: s_j<t_i \text{ and } X_{s_j}<X_{t_i}\}.\]
It is understood that one of $r_i, l_i$ can be empty. In particular,
$r_i=\emptyset$ for the rightmost internal minimum,
$l_i=\emptyset$ for the leftmost internal minimum of $X_t$.
We write $j\in {\sf chain}(k)$ to denote that local minimum at $s_j$ belongs 
to chain $k$ for some $1\le k\le n^{\rm int}$.
We now define
\[{\sf parent}(i) = \{k: s_j = \arg\max(X_{r_i},X_{l_i}) \text{ and } j\in {\sf chain}(k) \},\]
with a convention that $X_{\emptyset}=\emptyset$ and $\max(a,\emptyset)=a$ for any $a\in\mathbb{R}$. 

The chains of equally-valued local minima correspond to non-binary
trees: the degree of a vertex that corresponds to a chain that includes $k$ local 
minima is $k+1$.
The functions with distinct values of local minima correspond to binary trees.

{\bf Metric structure of} $\textsc{level}(X_t)$.
We specify the metric structure by assigning the edge lengths
\[|e_i|\equiv l_i = |X_{t_i} - X_{t_{{\sf parent}(i)}}|\]
to all edges $e_i$.

{\bf Planar embedding of} $\textsc{level}(X_t)$.
The planar embedding (ordering) of the offspring of the same parent vertex 
coincides with that of the time instants $t_i$ of the respective local 
extrema of $X_t$.

\bigskip

By construction, the level-set tree $\textsc{level}(X_t)$ is completely
determined by the sequence of the values of local extrema of $X_t$.
In particular, if $g(t)$ is a continuous and monotone increasing function on 
$[a,b]$, then the time transformation of $X_t$ by function $g(t)$ does not
disturb the level set tree:
\[\textsc{level}(X_t)=\textsc{level}\left(X_{g(t)}\right).\]
Hence, without loss of generality we can focus on the level set trees
of continuous functions with alternating slopes $\pm 1$.
To ensure that the level set tree of a function is binary, we need 
to eliminate the chains of equally-valued minima.
This, for instance, is achieved if the distribution 
of lengths of linear segments has no atoms.
We denote by $\cE^{\rm ex}$ the space of finite piece-wise linear positive 
continuous excursions with alternating slopes $\pm 1$, atomless 
segment length distribution, and a finite number of segments.

By construction, the level set tree and Harris path are reciprocal
to each other as described in the following statement.

\begin{lem}[{{\bf Reciprocity of Harris path and level set tree}}]
\label{lem:recip}
The Harris path
$H: \L^{|}\to \cE^{\rm ex}$ and the level set tree
${\textsc{level}}: \cE^{\rm ex}\to\L^{|}$ are reciprocal
to each other.
This means that for any $T\in\L^{|}$ we have
$\textsc{level}(H_T(t))\equiv T,$
and for any $Y_t\in\cE^{\rm ex}$ we have
$H_{\textsc{level}(Y_t)}(t)\equiv Y_t.$
\end{lem}

\subsection{Exponential critical binary Galton-Watson tree ${\sf GW}(\lambda)$}
\label{erw}

Recall that a (combinatorial) critical binary Galton-Watson tree $T\in\cT$ describes a 
trajectory of the Galton-Watson branching process.
The process starts with a single progenitor (tree root) at time $t=0$. 
At each discrete time step every existing population member terminates and 
produces, equiprobably, either no or two offspring, independently of the other members. 
We denote the resulting tree distribution on $\cT$ by $\mathcal{GW}^{\rm crit}$.

\begin{Def} [{{\bf Exponential critical binary Galton-Watson tree}}]
\label{def:binary}
We say that a random tree 
$T\in\L^{|}$ is
an exponential critical binary Galton-Watson tree with parameter $\lambda>0$, 
and write $T\stackrel{d}{=}{\sf GW}(\lambda)$,
if  
\begin{itemize}
\item[(i)] \textsc{shape}($T$) is a critical binary Galton-Watson tree $\mathcal{GW}^{\rm crit}$,
\item[(ii)] the orientation for every pair of siblings in $T$ is random and symmetric (e.g., in each pair of siblings, 
a randomly and uniformly selected sibling is assigned a right orientation, and the other is assigned a left orientation) 
\item[(iii)] given \textsc{shape}($T$), the edges of $T$ are sampled as independent 
exponential random variables with parameter $\lambda$, i.e., with density
\be
\label{exp}
\phi_\lambda (x)=
\lambda e^{-\lambda x} {\bf 1}_{\{x\ge0\}}.\\ 
\ee
\end{itemize}
\end{Def}

The following result is well-known.
\begin{thm}{\rm \cite[Lemma 7.3]{Pitman},\cite{LeGall93,NP}}
\label{Pit7_3}
Consider a random excursion $X_t\in\cE^{\rm ex}$.
The level set tree $\textsc{level}(X_t)$ 
is an exponential critical binary Galton-Watson tree ${\sf GW}(\lambda)$ 
if and only if the rises and falls of $X_t$, excluding the last fall, are 
distributed as independent exponential random variables with parameter $\lambda/2$.
\end{thm}

Consider a random walk $\{X_k\}_{k\in\mathbb{Z}}$ with a 
homogeneous transition kernel $p(x,y)\equiv p(x-y)$, for any $x,y\in\mathbb{R}$,
given by a mixture of exponential jumps (Laplace distribution): 
\be
\label{laplace}
p(x)={\phi_{\lambda}(x)+\phi_{\lambda}(-x)\over 2} = 
{\lambda \over 2}e^{-\lambda|x|},\quad-\infty<x<\infty.
\ee
This process is called a {\it symmetric exponential random walk} with parameter $\lambda$.
Each symmetric exponential random walk with parameter $\lambda$ corresponds to a 
piece-wise linear continuous function $\{X_t\}_{t\in\mathbb{R}}$ with slopes $\pm1$
whose alternating rises and falls, taken from $\{X_k\}_{k\in\mathbb{Z}}$, 
have independent exponential lengths with parameter $\lambda/2$.
Specifically, consider a piece-wise linear function that interpolates the local
extrema of $X_k$; then transform the time in such a way that the slopes of 
the linear interpolation are $\pm 1$. 
There is one-to-one correspondence between the infinite sequences of the values 
of local extrema of $\{X_t\}_{t\in\mathbb{R}}$ and $\{X_k\}_{k\in\mathbb{Z}}$.
We refer to such a function as a symmetric exponential random walk with
parameter $\lambda/2$ on $\mathbb{R}$.

\begin{cor}
The Harris path $H_{{\sf GW}(\lambda)}$ of an exponential critical binary 
Galton-Watson tree with parameter $\lambda$ is an excursion of a symmetric 
exponential random walk  $\{X_t\}_{t\in\mathbb{R}}$ with parameter $\lambda/2$.
\end{cor}

\subsection{Length of a random tree ${\sf GW}(\lambda)$} 
Recall the modified Bessel functions of the first kind
$$I_\nu(z)=\sum\limits_{n=0}^\infty {\left({z \over 2}\right)^{2n+\nu} \over \Gamma(n+1+\nu)\, n!}.$$

\begin{lem}\label{lem:ell}
Suppose $T\stackrel{d}{=}{\sf GW}(\lambda)$ is an exponential critical binary Galton-Watson tree
with parameter $\lambda$.
The {\it total length} of the tree $T$ has the probability density function 
\begin{equation} \label{eq:pdfs}
\ell(x) = {1 \over x}  e^{-\lambda x}  I_1 \big(\lambda x\big), \quad x>0. 
\end{equation}
\end{lem}
\begin{proof}
The number of different combinatorial shapes of a planar
binary tree with $n+1$ leaves, and therefore $2n+1$ edges, is given by 
the {\it Catalan number} $C_n={1 \over n+1}\binom{2n}{n}={(2n)! \over (n+1)! n! }$.  
The total length of $2n+1$ edges is a gamma random variable with parameters 
$\lambda$ and $2n+1$ and density function
\[\gamma_{\lambda,2n+1}(x)={\lambda^{2n+1} x^{2n} e^{-\lambda x} \over \Gamma(2n+1)}, \quad x>0.\]
Hence, the total length of the tree $T$ has the probability density function
\begin{align}\label{eq:ell}
\ell(x) &= \sum\limits_{n=0}^\infty {C_n \over 2^{2n+1}} \cdot {\lambda^{2n+1} x^{2n} e^{-\lambda x} \over (2n)!} 
= \sum\limits_{n=0}^\infty {\lambda^{2n+1} x^{2n} e^{-\lambda x} \over 2^{2n+1} (n+1)! n!}  \nonumber \\
&=  {1 \over x}  e^{-\lambda x} \sum\limits_{n=0}^\infty {\left({\lambda x \over 2}\right)^{2n+1} \over \Gamma(n+2)\, n!} = {1 \over x}  e^{-\lambda x}  I_1 \big(\lambda x\big).
\end{align}
\end{proof}

\medskip
\noindent
Next, we compute the Laplace transform of $\ell(x)$. By the summation formula in (\ref{eq:ell}),
\begin{align*}
\mathcal{L}\ell(s) &= \int\limits_0^\infty  \sum\limits_{n=0}^\infty {C_n \over 2^{2n+1}} \cdot {\lambda^{2n+1} x^{2n} e^{-(\lambda +s) x} \over (2n)!}  \,dx\\
&=\sum\limits_{n=0}^\infty {C_n \over 2^{2n+1}} \cdot \left({\lambda \over \lambda +s}\right)^{2n+1}\int\limits_0^\infty {(\lambda +s)^{2n+1} x^{2n} e^{-(\lambda +s) x} \over (2n)!}  \,dx\\
&=\sum\limits_{n=0}^\infty {C_n \over 2^{2n+1}}\cdot \left({\lambda \over \lambda +s}\right)^{2n+1} =Z \cdot c(Z^2),
\end{align*}
where we let $Z={\lambda \over 2(\lambda +s)}$, and the characteristic function of Catalan numbers $$c(z)=\sum\limits_{n=0}^\infty C_n z^n={2 \over 1+\sqrt{1-4z}}$$ is well known.
Therefore
\begin{equation}\label{eq:LaplaceL}
\mathcal{L}\ell(s)=Z \cdot c(Z^2)={\lambda \over \lambda + s+\sqrt{(\lambda +s)^2 -\lambda^2}}.
\end{equation}

Note that the Laplace transform $\mathcal{L}\ell(s)$ could be derived from the total probability formula
\begin{equation}\label{recursionEll}
\ell(x)={1 \over 2}\phi_\lambda (x)+{1 \over 2}\phi_\lambda \ast \ell \ast \ell(x),
\end{equation}
where $\phi_\lambda (x)$ is the exponential p.d.f. (\ref{exp}). Thus, $\mathcal{L}\ell(s)$ solves
\begin{equation}\label{recursionLaplace}
\mathcal{L}\ell(s)={1 \over 2} {\lambda \over \lambda +s}\Big(1+\big(\mathcal{L}\ell(s)\big)^2\Big).
\end{equation}

\begin{cor}
The probability density function $f(x)$ of the length of an excursion
in an exponential symmetric random walk with parameter $\lambda$
is given by
\begin{equation}
f(x)={1 \over 2} \ell(x/2).
\end{equation}
\end{cor}
\begin{proof}
Observe that the excursion has twice the length of a tree ${\sf GW}(\lambda)$.
\end{proof}
\bigskip

\subsection{Prune invariance of ${\sf GW}(\lambda)$}
\label{sec:PI}
This section establishes prune invariance of exponential Galton-Watson trees 
with respect to arbitrary generalized pruning.

\begin{thm} \label{main}
Let $T\stackrel{d}{=}{\sf GW}(\lambda)$ be an exponential critical binary Galton-Watson tree 
with parameter $\lambda>0$.
Then, for any monotone non-decreasing function $\varphi:\L\rightarrow\mathbb{R}^+$ and
any $\Delta>0$ we have
\[T^\Delta:=\{\cS_\Delta(\varphi,T)|\cS_\Delta(\varphi,T) \not= \phi\} \stackrel{d}{=} 
{\sf GW}(\lambda p_{\Delta}(\lambda,\varphi)),\]
where $p_{\Delta}(\lambda,\varphi)={\sf P}(\cS_\Delta(\varphi,T) \not= \phi)$. 
That is, the pruned tree $T^\Delta$ conditioned on surviving is an exponential 
critical binary Galton-Watson tree with parameter
$$\mathcal{E}_\Delta(\lambda,\varphi)=\lambda p_{\Delta}(\lambda,\varphi).$$
\end{thm}
\begin{proof}
Let $X$ denote the length of the edge of $T$ adjacent to the root of $T$, and let $Y$ denote the length of the edge of $T^\Delta$ adjacent to the root of $T^\Delta$. 
Let $x$ be the  descendent vertex (a junction or a leaf) to the root in $T$. 
Then $X$, which is an exponential random variable with parameter $\lambda$, represents the distance from the root of $T$ to $x$. 
Let ${\sf deg}_T(x)$ denote the degree of $x$ in tree $T$ and ${\sf deg}_{T^\Delta}(x)$ denote the degree of $x$ in tree $T^\Delta$.
If $T^\Delta=\phi$, then $Y=0$. 
Let $$F(h)={\sf P}(Y \leq h~|~ \cS_\Delta(\varphi,T) \not= \phi).$$  
The event $\{Y \leq h\}$ is partitioned into the following non-overlapping sub-events S$_1,\dots$ S$_4$ illustrated in Fig.~\ref{fig:thm2}:

\begin{figure}[hbt]
	\centering
	\includegraphics[width=0.9\textwidth]{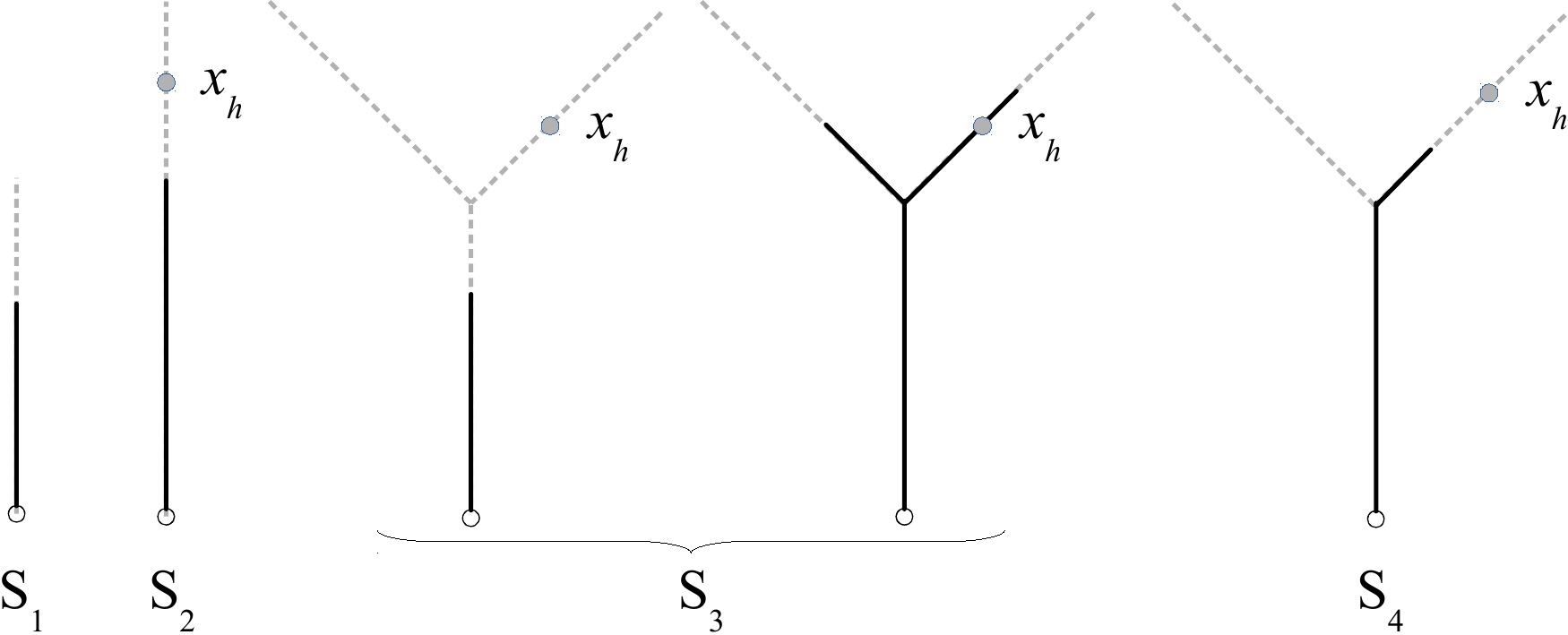}
	\caption{Sub-events used in the proof of Theorem~\ref{main}.
	Gray dashed line shows (a part of) initial tree $T$.
	Solid black line shows (a part of) pruned tree $T^{\Delta}$. 
	We denote by $x_h$ a point in $T$ located at distance $h$ from the root, if it exists.}
	\label{fig:thm2}
\end{figure}
\begin{itemize}
  \item[(S$_1$)] The event 
  $\{{\sf deg}_T(x)=1\text{ and }X \leq h\}$ has probability
  $${1 \over 2}(1-e^{-\lambda h}).$$
  
  \item[(S$_2$)] The event 
  $\{X>h$ and all points of $T$ descendant to $x_h$ do not belong to $T^\Delta\}$ 
  has probability
  $$e^{-\lambda h}(1-p_\Delta).$$
  
  \item[(S$_3$)] The event $\{X \leq h$ and ${\sf deg}_T(x)=3$ and either both subtrees of $T$ descending from $x$ are pruned away
  completely (not intersecting $T^\Delta$) or $\{x \in T^\Delta,~{\sf deg}_{T^\Delta}(x)=3\}\}$
   has probability
  $${1 \over 2}(1-e^{-\lambda h})\big((1-p_\Delta)^2+p_\Delta^2 \big).$$
  
  \item[(S$_4$)] The event $\{\{X \leq h$, ${\sf deg}_T(x)=3$\} and 
  $\{x \in T^\Delta, ~{\sf deg}_{T^\Delta}(x)=2\}$ and $Y \leq h\}$  has probability\footnote{Here, ${\sf deg}_{T^\Delta}(x)=2$ means $x$ is neither a junction nor a leaf in $T^\Delta$.}
  $${1 \over 2}\int\limits_0^h \lambda e^{-\lambda t} \cdot 2p_\Delta(1-p_\Delta)\cdot F(h-t)\, dt=p_\Delta(1-p_\Delta) \int\limits_0^\infty \lambda e^{-\lambda t} F(h-t)\, dt.$$
\end{itemize}
Using this we have two representations for the probability ${\sf P}(Y\le h)$:
\begin{align*}
{\sf P}(Y\le h) =& (1-p_\Delta)+p_\Delta F(h)  \\
=& {1 \over 2}(1-e^{-\lambda h})+e^{-\lambda h}(1-p_\Delta)\\
&+{1 \over 2}(1-e^{-\lambda h})\big((1-p_\Delta)^2+p_\Delta^2 \big)\\
&+p_\Delta(1-p_\Delta) \int\limits_0^\infty \lambda e^{-\lambda t} F(h-t)\, dt,
\end{align*}
which simplifies to
$$(1-p_\Delta)+p_\Delta F(h) =(1-p_\Delta+p^2_\Delta )-e^{-\lambda h} p_\Delta +p_\Delta(1-p_\Delta) \int\limits_0^\infty \lambda e^{-\lambda t} F(h-t)\, dt.$$
Differentiating the above equality we obtain the following equation for the probability density function $f(y)={d\over dy}F(y)$ of $Y$:  
$$f(h)=p_\Delta \phi_\lambda (h) +(1-p_\Delta) \phi_\lambda \ast f(h),$$
where as before $\phi_\lambda$ denotes the exponential density with parameter $\lambda$.
Applying integral transformation on both sides of the equation, we obtain the characteristic function $\hat{f}(s)=E[e^{isY}]$ of $Y$,
$$\hat{f}(s)={\lambda p_\Delta \over \lambda p_\Delta -is}=\hat{\phi}_{\lambda p_\Delta}(s).$$
Thus, we conclude that $Y$ is an exponential random variable with parameter $\lambda p_\Delta$.

\bigskip
Next, let $y$ be the  descendent vertex (a junction or a leaf) to the root in $T^\Delta$. 
If $T^\Delta = \phi$, let $y$ denote the root.
Let $$q={\sf P}({\sf deg}_{T^\Delta}(y)=3 ~|~ S_\Delta(T) \not= \phi).$$
Then,
\begin{align*}
p_\Delta q=& {\sf P}({\sf deg}_{T^\Delta}(y)=3)\\
=& {\sf P}({\sf deg}_T(x)=3) \cdot\Big\{{\sf P}\big({\sf deg}_{T^\Delta}(x)=3~|~{\sf deg}_T(x)=3\big)\\
&+{\sf P}\big({\sf deg}_{T^\Delta}(x)=2~|~{\sf deg}_T(x)=3\big)\cdot q \Big\}\\
=& {1 \over 2}\Big\{p^2_\Delta + 2p_\Delta(1-p_\Delta) q  \Big\}
\end{align*}
implying
$$q={1 \over 2} p_\Delta+(1-p_\Delta) q,$$
which in turn yields $q={1 \over 2}$.

\bigskip
We saw that conditioning on $\cS_\Delta(\varphi,T) \not= \phi$, the pruned tree $T^\Delta$ has the edge connecting the root to its descendent  vertex $y$ distributed exponentially 
with parameter $\lambda p_\Delta$. 
Then, with probability $q={1 \over 2}$, the pruned tree $T^\Delta$ branches at $y$ into two independent subtrees, each distributed as
$\{T^\Delta ~|~T^\Delta\ne \phi\}$. 
Thus, we recursively obtain that $T^\Delta$ is a critical binary Galton-Watson tree with i.i.d. exponential edge length with parameter $\lambda p_\Delta$.
\end{proof}

Next, we find an exact form of the survival probability $p_\Delta(\lambda,\varphi)$ for three particular choices of $\varphi$, thus obtaining
$\mathcal{E}_\Delta(\lambda,\varphi)$.

\begin{thm} \label{pdelta}
In the settings of Theorem \ref{main}, we have 
\begin{description}
  \item[(a)] If $\varphi(T)$ equals the total length of $T$ $(\varphi = \textsc{length}(T))$, 
  then
  $$\mathcal{E}_\Delta(\lambda,\varphi)=\lambda e^{-\lambda \Delta}\Big[ I_0(\lambda \Delta)+ I_1(\lambda \Delta) \Big].$$
  
  \item[(b)] If $\varphi(T)$ equals the height of $T$ $(\varphi = \textsc{height}(T))$, then
  $$\mathcal{E}_\Delta(\lambda,\varphi)={2\lambda  \over \lambda \Delta +2}.$$
  
  \item[(c)] If $\varphi(T)+1$ equals the Horton-Strahler order of the tree $T$, then 
  $$\mathcal{E}_\Delta(\lambda,\varphi)=\lambda 2^{-\lfloor \Delta \rfloor},$$
  where $\lfloor \Delta \rfloor$ denotes the maximal integer  $\le \Delta$.
\end{description}
\end{thm}

\begin{proof}
{\bf Part (a).} By Lemma \ref{lem:ell}, 
\begin{align}\label{eq:pDelta1}
p_\Delta =& 1-\int\limits_0^{\Delta} \ell(x) \,dx =1-\int\limits_0^{\lambda \Delta}{1 \over x}  e^{-x}  I_1 \big(x\big) \,dx \nonumber \\
&= e^{-\lambda \Delta}\Big[ I_0(\lambda \Delta)+ I_1(\lambda \Delta) \Big],
\end{align}
where for the last equality we used formula 11.3.14 in \cite{AS1964}.

\bigskip
\noindent
{\bf Part (b).} 
Using the representation of an exponential critical binary Galton-Watson tree 
${\sf GW}(\lambda)$ via its Harris path as in Theorem~\ref{Pit7_3}, we consider 
a symmetric exponential random walk $Y_k$ with jump distribution given by 
\eqref{laplace}, and notice that $Y_k$ is a martingale.

The probability $p_\Delta$ now can be interpreted for the Harris path as the probability that a given excursion of $Y_k$ is above $\Delta$ for some $k$.  
This would guarantee that the depth of a tree exceeds $\Delta$. 
Specifically, let $T_0=\min\{k>1 ~:~ Y_k \leq 0\}$. 
We condition on $Y_1>0$, and consider an excursion  
$Y_0,Y_1,\hdots, Y_{T_0}$. 
Here
$$p_\Delta={\sf P}\left(\max\limits_{j:~ 0< j <T_0} Y_j >\Delta ~\Big|~Y_1>0 \right).$$

\medskip
\noindent
The problem of finding $p_\Delta$ is solved using the Optional Stopping Theorem. Let 
$$T_\Delta=\min\{k >0~:~Y_k \geq \Delta\} \qquad \text{ and } \qquad T:=T_\Delta \wedge T_0.$$
Next, observe that
$$p_\Delta={\sf P}(T=T_\Delta ~|~Y_1>0).$$
For a fixed $y \in (0,\Delta)$, by the Optional Stopping Theorem, we have
\begin{eqnarray*}
y & = & E[Y_T ~|~Y_1=y]\\
& = & E[Y_T ~|~T=T_0, Y_1=y] {\sf P}(T=T_0 ~|~ Y_1=y)\\
& & \qquad +E[Y_T ~|~T=T_\Delta, Y_1=y] {\sf P}(T=T_\Delta ~|~ Y_1=y)\\
& = & E[Y_T ~|~Y_T \leq 0, Y_1=y] {\sf P}(T=T_0 ~|~ Y_1=y)\\
& & \qquad +E[Y_T ~|~Y_T \geq \Delta, Y_1=y] {\sf P}(T=T_\Delta ~|~ Y_1=y)\\
& = & -{1 \over \lambda}{\sf P}(T=T_0 ~|~ Y_1=y)+
\left(\Delta+{1 \over \lambda}\right){\sf P}(T=T_\Delta ~|~ Y_1=y)\\
& = & \left(\Delta+{2 \over \lambda}\right){\sf P}(T=T_\Delta ~|~ Y_1=y)-{1 \over \lambda}.\\
\end{eqnarray*}

\noindent
Hence, $${\sf P}(T=T_\Delta ~|~ Y_1=y)={y+{1 \over \lambda} \over \Delta +{2 \over \lambda}}.$$
Thus,
\begin{eqnarray*}
{\sf P}\Big(T=T_\Delta,  0<Y_1<\Delta ~|~Y_1>0\Big) & = &\int\limits_0^\Delta {\sf P}(T=T_\Delta ~|~ Y_1=y)~\lambda e^{-\lambda y} dy\\
& = & \int\limits_0^\Delta {y+{1 \over \lambda} \over \Delta +{2 \over \lambda}}~\lambda e^{-\lambda y} dy\\
& = & {2 \over \lambda \Delta +2}-e^{-\lambda \Delta},
\end{eqnarray*}

\noindent
and therefore,
\begin{eqnarray*}
p_\Delta & = & {\sf P}\left(\max\limits_{j:~ 0< j <K } Y_j >\Delta ~|~Y_1>0 \right)\\
& = & {\sf P}\Big(T=T_\Delta,  0<Y_1<\Delta ~|~Y_1>0\Big)+{\sf P}\Big(T=T_\Delta,  Y_1 \geq \Delta ~|~Y_1>0\Big)\\
& = & {2 \over \lambda \Delta +2}-e^{-\lambda \Delta}+{\sf P}\Big(Y_1 \geq \Delta ~|~Y_1>0\Big)\\
& = & {2 \over \lambda \Delta +2}.
\end{eqnarray*}

\bigskip
\noindent
{\bf Part (c).} 
Follows from \cite[Corollary 1]{KZ18}. 
\end{proof}

\begin{rem}
Let ${\mathcal E}_\Delta (\lambda,\varphi)={2\lambda \over \lambda \Delta +2}$ as in Theorem \ref{pdelta}(b).  
Here $~{\mathcal E}_0\lambda=\lambda$
and ${\mathcal E}_\Delta (\lambda,\varphi)$ is a linear-fractional transformation 
associated with matrix
$${\mathcal A}_\Delta=\begin{pmatrix}
     1 & 0   \\
     {\Delta \over 2}  &  1
\end{pmatrix}.$$
Since ${\mathcal A}_\Delta$ form a subgroup in $SL_2(\mathbb{R})$, the transformations $\left\{{\mathcal E}_\Delta \right\}_{\Delta \geq 0}$ satisfy the semigroup property
$${\mathcal E}_{\Delta_1} {\mathcal E}_{\Delta_2} ={\mathcal E}_{\Delta_1+\Delta_2}$$
for any pair $\Delta_1, \Delta_2 \geq 0$.

We notice also that the operator ${\mathcal E}_\Delta (\lambda,\varphi)$ in part (c) of Theorem~\ref{pdelta} satisfies only the discrete semigroup property for nonnegative integer times. 
Finally, one can check that ${\mathcal E}_\Delta (\lambda,\varphi)$ in part (a) does not satisfy the semigroup property.
\end{rem}

\section{Continuum ballistic annihilation model, $A+A \rightarrow \zeroslash$}
\label{sec:annihilation}
Here we turn to the dynamics of particles governed by 1-D ballistic annihilation model.
Section~\ref{sec:shock} introduces the model and describes the natural emergence of sinks (shocks).
The model initial conditions are given by a particle velocity distribution and particle density 
on $\mathbb{R}$.
We consider here a constant density and initial velocity distribution 
with alternating values $\pm1$, or, equivalently,
initial piece-wise linear potential $\psi(x,0)$ with alternating slopes $\pm1$ 
(Fig.~\ref{fig:vel}).
Section~\ref{sec:basics} summarizes the basic constraints on the model dynamics.
Section~\ref{sec:solution} provides a construction of the
shock tree (Lemma~\ref{lem:SWT}) and its graphical 
embedding into the phase space $(x,\psi(x,t))$ and space-time domain $(x,t)$.
Theorem~\ref{thm:SWT} states that the shock tree is the level set tree
for the initial potential $\psi(x,0)$.
Theorem~\ref{thm:pruning} in Section~\ref{sec:Bdyn} establishes equivalence of 
the ballistic annihilation dynamics and a generalized dynamical pruning of 
a mass-equipped shock tree. 

\subsection{Continuum model, sinks, and shock trees}
\label{sec:shock}
The ballistic annihilation model describes the dynamics of a set of 
particles that move with given initial velocities and annihilate at collision, 
hence the model notation $A+A \rightarrow \zeroslash$  
\cite{EF85, BNRL93,Belitsky1995,Piasecki95,Droz95,BNRK96,Ermakov1998,Blythe2000,KRBN2010,Sidoravicius2017}.
The model dynamics is illustrated in Fig.~\ref{fig:bam}.
We introduce here a continuum mechanics formulation of ballistic annihilation, analyzing the dynamics of a  mass distributed on the real line. 
This formulation makes sense when a large number of particles densely populates the system.

Formally, we are given a Lebesgue measurable initial density $g(x) \geq 0$ of 
particles on a line $\mathbb{R}$. 
The initial particle velocities are given by $v(x,0)=v(x)$.
Prior to collision and subsequent annihilation, a particle located at $x_0$ at time $t=0$ 
moves according to its initial velocity, so its coordinate $x(t)$ changes as
\begin{equation}
\label{eq:tangency_point}
x(t)=x_0+tv(x_0).
\end{equation}
When the particle collides with another particle, it annihilates.
Accordingly, two particles with initial coordinates and velocities $(x_-,v_-)$ and $(x_+,v_+)$
collide and annihilate at time $t$ when they meet at the same new position, 
\[x_-+tv_-=x_++tv_+,\]
given that neither of the particles annihilated prior to $t$.
In this case the annihilation time is given by
\be
\label{t_sink}
t = -\frac{x_+-x_-}{v_+-v_-}.
\ee

Let $v(x,t)$ be the Eulerian specification of the velocity field at coordinate $x$
and time instant $t$; we define the corresponding 
{\it potential function} 
\[\psi(x,t)=-\int_a^x v(y,t)dy,\] 
so that $v(x,t)=-\partial_x \psi(x,t)$. 
Let $\psi(x,0)=\Psi_0(x)$ be the initial potential. 

We call a point $\sigma(t)$ {\it sink} (or {\it shock}), if there exist two particles that
annihilate at coordinate $s$ at time $t$.
Suppose $v(x) \in \mathcal{C}^1(\mathbb{R})$. 
The equation \eqref{t_sink} implies that appearance of a sink is associated with
a negative local minima of $v'(x^*)$; 
we call such points {\it sink sources}.
Specifically, if $x^*$ is a sink source, then a sink will appear at {\it breaking time} 
$t^*=-1/v'(x^*)$ with the location given by 
\[\sigma(t^*)=x^*+t^*v(x^*)=x^*-{v(x^*) \over v'(x^*)}\]
provided there exists a  punctured neighborhood 
$$N_\delta(x^*)=\{x:~0<|x-x^*|<\delta\} \subseteq \mathrm{supp}\{g(x)\}$$ 
such that none of the particles with the initial coordinates in $N_\delta(x^*)$ is annihilated before time $t^*$. 

Sinks, which originate at sink sources, can move and coalesce (see Fig.~\ref{fig:bam}). 
We refer to a sink trajectory as a {\it shock wave}. 
We impose the conservation of mass condition by defining the {\it mass of a sink} 
at time $t$ to be the total mass of particles annihilated in the sink between 
time zero and time $t$. 
When sinks coalesce, their masses add up.
It will be convenient to assume that sinks do not disappear
when they stop accumulating mass.
Informally, we assume that the sinks are being pushed by the system particles.
Formally, there exists several cases depending on the occupancy of a
neighborhood of $\sigma(t)$.
If there exists an empty neighborhood around the sink coordinate $\sigma(t)$, the 
sink is considered 
at rest (its coordinate does not change). 
If only the left neighborhood of $\sigma(t)$ is empty, and the right adjacent velocity
is negative:
\[v(\sigma_+,t):=\lim_{x\downarrow \sigma(t)} v(x,t)<0,\]
the sink at $\sigma(t)$ moves with velocity $v(\sigma_+,t)$.
A similar rule is applied to the case of right empty neighborhood.
The appearance, motion, and subsequent coalescence of sinks can be described 
by a time oriented {\it shock tree}.
In particular, the coalescence of sinks under initial conditions with a finite 
number of sink sources is described by a finite tree.

The dynamics of ballistic annihilation, either in discrete or 
continuum versions, can be quite intricate and is lacking
a general description.
The existing analyses
focus on the evolution of selected statistics under particular initial conditions. 
In the next section, we outline some basic principles that can facilitate
the analysis of the continuum ballistic annihilation.
We then give a complete description of the dynamics in case
of two-valued initial velocity and constant particle density.
Specifically, we describe the evolution of a constant-slope 
potential $\psi(t,x)$ within a finite spatial domain for $t>0$.

\subsection{Basic constraints on ballistic annihilation dynamics}
\label{sec:basics}
Suppose $x_-<x_+$ are such that  $v(x_-)>v(x_+)$. 
Assume that density $g(x)$ is positive, and suppose there is only 
one sink source $x^* \in (x_-, x_+)$. 
In order for $x_-$ and $x_+$ to annihilate each other in the sink originated 
at $x^*$ at time $t>t^*$ we need the following:
\begin{itemize}
  \item[(i)] Collision at time $t$: $$x_- +tv(x_-)=x_+ +tv(x_+).$$
  \item[(ii)] The mass between $x_-$ and $x^*$  annihilates the mass between $x^*$ and $x_+$:
  $$\int\limits_{x_-}^{x^*} g(x) \, dx=\int\limits_{x^*}^{x_+} g(x) \, dx.$$
  \item[(iii)]  Neither $x_-$ nor $x_+$  is annihilated before time $t$.
\end{itemize} 
From conditions (i) and (ii), we obtain the location of the sink at time $t$:
\begin{equation}\label{eq:sinkspeed}
\sigma(t)={v(x_-){g(x_-) \over 1+tv'(x_-)}+v(x_+){g(x_+) \over 1+tv'(x_+)} 
\over {g(x_-) \over 1+tv'(x_-)}+{g(x_+) \over 1+tv'(x_+)}}.
\end{equation}

\medskip
\noindent
This sink dynamics description is not restricted to $v(x) \in \mathcal{C}^1(\mathbb{R})$, and can be extended to the case of piecewise smooth $v(x)$.
\begin{ex}[{\bf Two velocities, single sink}]
Suppose $g(x)$ is positive.
Let $v(x)=\begin{cases}
      v_- & \text{ if }x\le x^* \\
      v_+ & \text{ if }x>x^*
\end{cases}$, where the constants $v_- >v_+$ and $x^*$ are given. 
Naturally, $x^*$ is the only sink source and the only sink appears 
at the sink source at time $t=0$. 
Moreover, analogously to (\ref{eq:sinkspeed}), one can derive the dynamics of the sink at time $t$:
\begin{equation}\label{eq:sinkspeedEX}
\sigma(t)={v_- g(x_-) +v_+ g(x_+) \over g(x_-) +g(x_+)},
\end{equation}
where $x_- <x^*$ is the only root of 
\[G_t(y)=\int\limits_{x^*}^{y+t(v_- -v_+)} g(x) \, dx -\int\limits_{y}^{x^*} g(x) \, dx,\]
and $x_+=x_- +t(v_- -v_+)$.
Note that $G_t(x)$ is continuous and strictly increasing, and that $G_t(x^*)>0>G_t\big(x^*-t(v_- -v_+)\big)$.
\end{ex}

\subsection{Piece-wise linear potential with unit slopes}
\label{sec:solution}
The discrete 1-D ballistic annihilation model with two possible velocities $\pm v$ was considered in
\cite{EF85,Belitsky1995,BNRK96,Ermakov1998,Blythe2000}; 
the three velocity case ($-1$, $0$, and $+1$) appeared in \cite{Droz95,Sidoravicius2017}. 
We study a continuum version of the 1-D ballistic annihilation with two possible initial 
velocities and constant initial density, i.e.  $v(x)=\pm v$ and $g(x)\equiv g_0$.
Since we can scale both space and time, without loss of generality we let $v(x)=\pm 1$ and $g(x)\equiv 1$.
More formally, we consider an initial potential 
$\psi(x,0)=\Psi_0(x)$ that 
is a negative piece-wise linear excursion  on a 
finite interval with segment slopes $\pm 1$; see Fig.~\ref{fig:vel}.
We denote the space of such functions $\widetilde{\E}$, and
write $\widetilde{\E}([a,b])$ if the domain $[a,b]\subset\mathbb{R}$ is explicitly given.
This space bears a lot of symmetries that facilitate our analysis.

Observe that the domain $[a,b]$ is partitioned into non-overlapping subintervals with
boundaries $x_j$ such
that the initial particle velocity assumes alternating values of $\pm 1$ within
each interval, with boundary values $v(a,0)=1$ and $v(b,0)=-1$. 
For a finite interval $[a,b]$, there exists a finite time $t_{\rm max} = (b-a)/2$
at which all particles aggregate into a single sink of mass
$m = (b-a)=2\,t_{\rm max}$; see discussion below.
We only consider the solution on the time interval $[0,t_{\rm max}]$, and
assume that the density of particles vanishes outside of $[a,b]$.

Since we assumed a unit mass density, the model dynamics is completely determined
by the potential $\Psi_0(x)$. 
The trajectories of the sinks in this model have a tree structure;
we denote the respective shock tree by $S(\Psi_0)$. 

Our first goal is to construct the system trajectory, including the 
trajectories of the massive 
sinks, in the phase space 
\[(x,\psi(x,t))\in\mathbb{R}^2, x\in[a,b], t\in[0,t_{\rm max}]\]
and in the respective space-time domain $(x,t)$.
We do this below for consecutively more complicated potentials.
We focus on the shock wave tree $S(\Psi_0)$, and demonstrate
that this tree can be used to reconstruct the entire system dynamics.

{\bf V-shaped potential.}
Let $x_1=c=(a+b)/2$ be the center of the segment $[a,b]$.
Consider a simplest V-shaped potential that consists of a negative
segment on $[a,c]$ and a positive segment on $[c,b]$, see Fig.~\ref{fig:V}.
In this case, there exists a single sink that originates
at $t=0$ at the point $(c,\Psi_0(c))$.
In $x$-space, it remains at rest and accumulates mass at rate 2 during the time
interval of duration $t_{\rm max}$, which reflects accumulation of particles 
that merge into the sink from left and right.
After this, the mass of the sink is $(b-a)$, which reflects complete accumulation
of mass from the interval $[a,b]$.

In the phase space (Fig.~\ref{fig:V}, bottom panel), the trajectory of the sink 
corresponds to a vertical segment of length $\v_1 = t_{\rm max} = (b-a)/2$ between 
points $(x_1,\Psi_0(x_1))$ and $(x_1,\Psi_0(x_1)+\v_1)$.
The trajectory of each particle is a horizontal line from the particle's
initial position $(x,\psi(x,0))$ to the point of merging with the sink at 
$(x_1,\psi(x,0))$. 

In the space-time domain (Fig.~\ref{fig:V}, top panel) the trajectory of the 
sink is a vertical segment between points $(x_1,0)$ and $(x_1,\v_1)$.
The trajectories of particles, each of which moves with its initial velocity 
until merging with the sink, is shown by thin diagonal lines.

  \begin{figure}[hbt]
  	\centering
  	\includegraphics[width=0.8\textwidth]{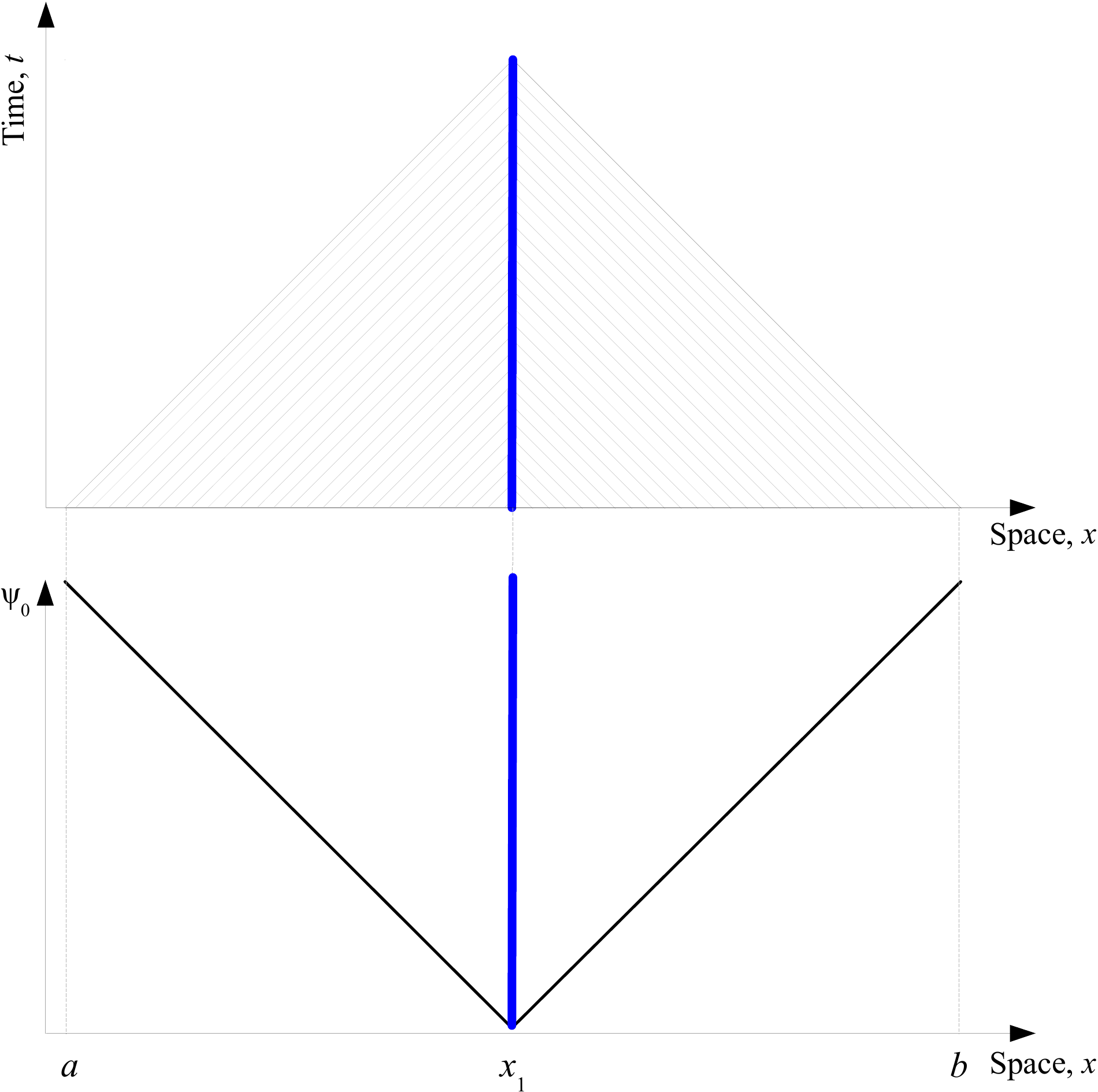}
  \caption{A V-shaped potential. (Bottom): The potential $\Psi_0$ (black line)
  and the shock wave in the phase space $(x,\psi)$ (blue segment).
  (Top): The space-time portrait.
  The system occupies a triangular (shaded) region in the $(x,t)$ space. 
  Thin hatching illustrates the trajectories of particles.
  Blue vertical segment show the trajectory of the sink.}
  \label{fig:V}
  \end{figure}

{\bf W-shaped potential.}
Consider now a negative excursion on $[a,b]$ with exactly two local minima
at $x_1, x_3$ and the only local maxima at $x_2$, with $a<x_1<x_2<x_3<b$, see
Fig.~\ref{fig:W}.
There exist two sinks that originate at $t=0$ at points $x_1$ and
$x_3$. 
The sink at $x_1$ remains at rest and accumulates mass at rate 2
during time interval of duration $\v_1 = \Psi_0(x_2) - \Psi_0(x_1)$.
At instant $t = \v_1$ the right neighborhood of the sink 
at $x_1$ becomes empty, and it starts moving at unit speed to the right.
Similarly, the sink at $x_3$ remains at rest and accumulates mass at rate 2
during time interval of duration $\v_3 = \Psi_0(x_2) - \Psi_0(x_3)$.
At instant $t = \v_3$ the left neighborhood of the sink 
at $x_3$ becomes empty, and it starts moving at unit speed to the left.
The two sinks move toward each other until they merge to
form a new sink of mass $2(\v_1+\v_3)$.
We denote by $\h_i$, $i=1,3$ the durations of these respective movements.
Since both right and left neighborhoods of the new sink are 
occupied by regular particles, the particle remains at rest for some time.

The following lemma describes the shape of the respective graphical shock 
trees 
$\cG^{(x,\psi)}(\Psi_0)$ and $\cG^{(x,t)}(\Psi_0)$ in the phase space and space-time domain,
respectively.

\begin{lem}[{{\bf Shock tree of a W-shaped potential}}]
\label{lem:W}
For a W-shaped potential described above (and illustrated in Fig.~\ref{fig:W}) 
we have:
\begin{itemize}
\item[(a)] The graphical shock tree $\cG^{(x,\psi)}(\Psi_0)$ in the phase 
space  
has the bracket shape that consists of two leaves and a root edge
(Fig.~\ref{fig:W}, bottom panel).
Each leaf corresponds to the dynamics of one of the two initial sinks;
the root edge corresponds to the dynamics of the final sink.
Each leaf consists of a vertical segment between
points $(x_i,\Psi_0(x_i))$ and $(x_i,\Psi_0(x_2))$,
and a horizontal segment between points
$(x_i,\Psi_0(x_2))$ and $(c,\Psi_0(x_2))$, for $i=1,3$.
The stem consists of a vertical segment between points
$(c,\Psi_0(x_2))$ and $(c,\Psi_0(a))$.
Here $c=(a+b)/2$ is the center of the interval $[a,b]$.

\item[(b)] In the space-time domain $(x,t)$, the system occupies 
a {\it cone} $\cC$ that has the shape of a right triangule with the hypothenuse 
on the interval $x\times t = [a,b]\times\{0\}$
and the catheti merging at the point $(c,c)$. 
The trajectories of the sinks form an inverted Y-shaped tree
shown in Fig.~\ref{fig:W} (top panel) that consists of two leaves and a stem.
Each leaf corresponds to the dynamics of one of the two initial sinks;
the stem corresponds to the dynamics of the final sink.
Each leaf consists of a vertical segment between
points $(x_i,0)$ and $(x_i,\v_i)$,
and a slanted segment between points
$(x_i,\v_i)$ and $(c,\v_1+\v_3)$, for $i=1,3$.
The stem consists of a vertical segment between points
$(c,\v_1+\v_3)$ and $(c,c)$.
There exists a rectangular empty region (no particles) with vertices at the points
(clock-wise from the bottom point): $(x_2,0)$, $(x_1,\v_1)$,
$(c,\v_1+\v_3)$, and 
$(x_3,\v_3)$.

\item[(c)] The particles move with their initial velocities
until they merge with a sink. 
A particle $x$ in the interval 
$[x_1-\v_1, x_1+\v_1 = x_2)$
merges with the sink at point $x_1$
at time instant $t = |x_1-x|$.
A particle $x$ in the interval 
$(x_3-\v_3=x_2, x_3+\v_3]$
merges with the sink at point $x_3$ at time instant $t=|x_3-x|$.
A particle $x$ in the intervals 
$[a,x_1-\v_1)$ and
$(x_3+\v_3,b]$
merges with the sink at point $(a+b)/2$ at time
instant $t = |(a+b)/2-x|$.
The particle at $x_2$ merges the sink
at $x_1$ ($x_3$) if the potential is left (right) continuous
at time instant $t=x_2 - x_1$ ($t=x_3-x_2$).

\end{itemize}
\end{lem} 
\begin{proof}
By construction of the solution of the continuum annihilation dynamics for a potential 
$\Psi_0(x)\in\widetilde{\E}([a,b])$, see Fig.~\ref{fig:W}.
\end{proof}

\begin{figure}[hbt]
\centering
\includegraphics[width=0.8\textwidth]{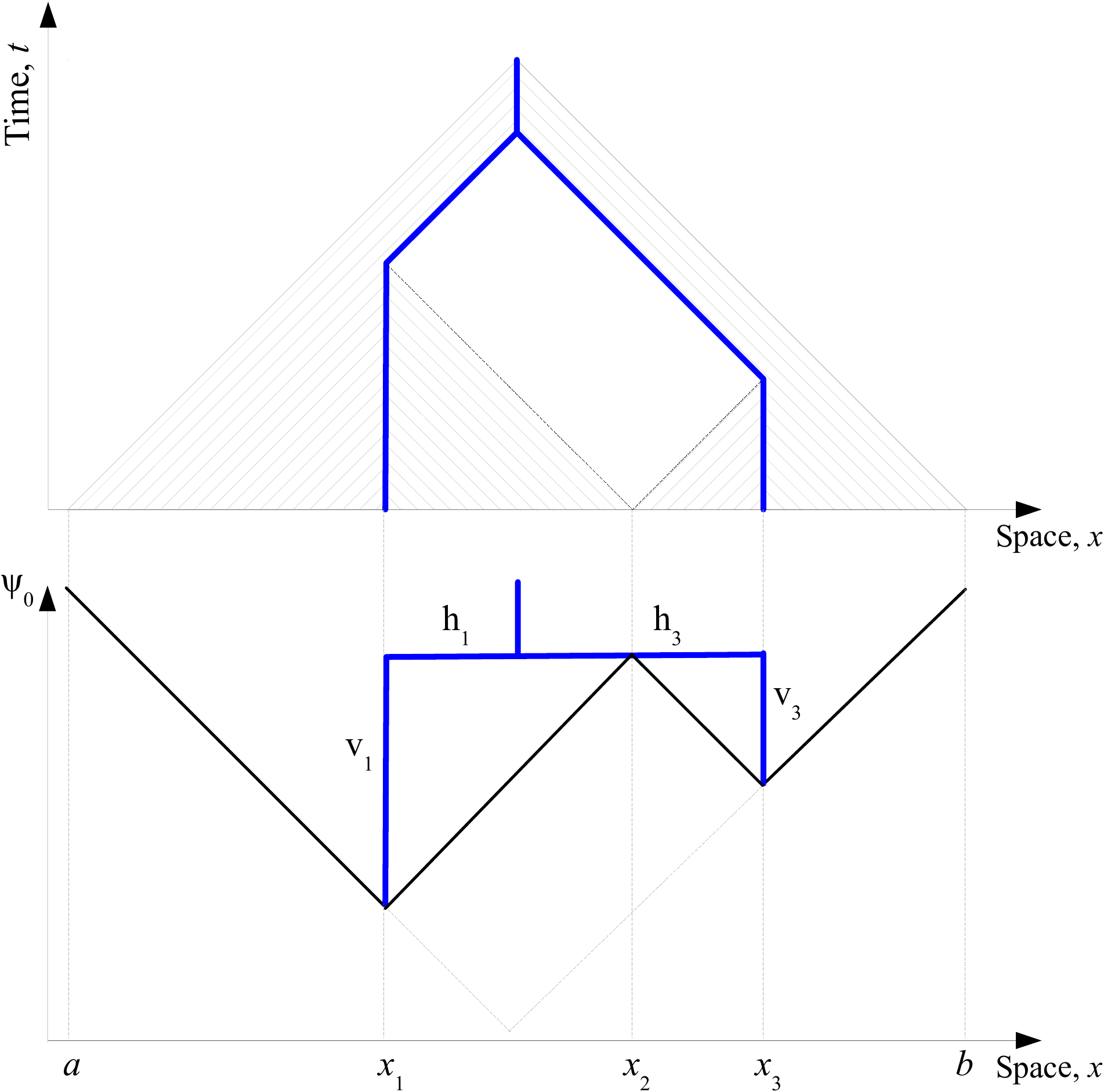}	
\caption{
  A W-shaped potential. (Bottom): The potential $\Psi_0$ (solid black)
  and the shock in the phase space $(x,\psi)$ (blue).
  (Top): The space-time portrait.
  The system occupies a shaded region in the $(x,t)$ space,
  bounded by a triangle that corresponds to the V-shaped potential
  on the interval $[a,b]$, as in Fig.~\ref{fig:V}.
  Notice the appearance of an empty rectangular region
  in the space-time portrait that corresponds to $(x,t)$ locations
  with no particles.
  Thin hatching illustrates the trajectories of particles.
  Blue lines show the trajectories of sinks.}
\label{fig:W}
\end{figure}  
  
We make now an important symmetry observation, which helps to extend
our geometric construction of the shock tree to an arbitrary potential
from $\widetilde{\E}$.
We define the {\it basin} $\cB_{j}$ for a local maximum 
at $x_j$ as the shortest interval that contains $x_j$ and supports 
a non-positive excursion in $\Psi_0(x)$.
Formally, $\cB_{j}=[x^{\rm left}_j, x^{\rm right}_j]$, where
\[x^{\rm right}_j = \inf(x:x>x_j\text{ and }\Psi_0(x)>\Psi(x_j)),\]
\[x^{\rm left}_j = \sup(x:x<x_j\text{ and }\Psi_0(x)>\Psi(x_j)).\]
The basin length is denoted by $|\cB_j| = x^{\rm right}_j-x^{\rm left}_j$.

\begin{lem}[{{\bf Symmetry lemma}}]
\label{lem:symmetry}
Let $\v_i, \h_i$ for $i=1,3$ be the lengths of the vertical and
horizontal segments, respectively, of the leaves of the shock tree for 
a W-shaped potential (see Fig.~\ref{fig:W}):
\[\v_i = \Psi_0(x_2)-\Psi_0(x_i),\quad
\h_i=|(a+b)/2 - x_i|.\]
Then
\[\h_1=\v_3,\quad \h_3=\v_1\] 
and
\[\v_i+\h_i = |\cB_2|/2 = (b-a)/2 - (\Psi_0(a) - \Psi_0(x_2)).\]
\end{lem} 
\begin{proof}
By elementary geometric properties of the tree for a W-shaped
potential illustrated in Fig.~\ref{fig:W}. 
\end{proof}

Lemma~\ref{lem:symmetry} implies that after instant $t=|\cB_2|/2$
when the sinks that originate at $x_1$ and $x_3$ merge,
the process dynamics is indistinguishable from that of
the V-shaped potential on $[a,b]$.
This allows us to construct the shock tree in a general case,
using a sequential construction by basin lengths.

{\bf General potential.}
Consider a potential $\Psi_0(x)\equiv\Psi_1(x)\in\widetilde{\E}([a,b])$.
The shock tree for the V-shaped potential was constructed above.
If the potential is not V-shaped, it has $n\ge 1$ local maxima
(see illustrations in Figs.~\ref{fig:unfold},\ref{fig:C}).
Consider the basins that correspond to the local
maxima of $\Psi_1(x)$ and index them according to their lengths, from shortest to longest:
\[|\cB_1| < |\cB_2| <\dots < |\cB_n|.\]
Let $t_i=|\cB_i|/2$.
For each basin $\cB_i$ we define the corresponding {\it space-time cone}
$\cC_i$ that has the shape of a right triangle with hypothenuse
$x\times t = \cB_i\times \{0\}$ and the catheti merging at 
the point $(x_i^{\rm left}+t_i,t_i)$.

It is readily seen that the shortest basin necessarily corresponds to a W-shaped potential.
We construct the shock tree for the regular particles within the space 
interval $\cB_1$ during the time interval $[0, t_1]$, 
using the W-shaped potential construction of Lemma~\ref{lem:W}.
Hence, we describe the system dynamics in the space-time cone
$\cC_1$.

Consider now an {\it unfolded} potential $\Psi_2(x)$ that coincides with
$\Psi_1(x)$ outside of $\cB_1$ and has a V-shaped form on $\cB_1$.
By construction, and using Lemma~\ref{lem:W}, the trees that correspond to the potentials
$\Psi_1$ and $\Psi_2$ coincide outside of the cone $\cC_1$
in the space-time domain.
The potential $\Psi_2(x)$ has $n-1$ local minima.
Its shortest basin is $\cB_2$; it necessarily corresponds 
to a W-shaped potential within $\Psi_2(x)$.
We use $\Psi_2(x)$ to construct the space-time tree on 
$\cC_2$, using the W-shaped potential algorithm of Lemma~\ref{lem:W}. 
The resulting tree is only considered within the space-time
subregion
$ \cC_2\setminus\cC_1.$
The union of this tree and the tree constructed in the initial step
within $\cC_1$ results in the tree within 
$\cC_1\cup\cC_2$.

    \begin{figure}[hbt]
  	\centering
	\includegraphics[width=0.7\textwidth]{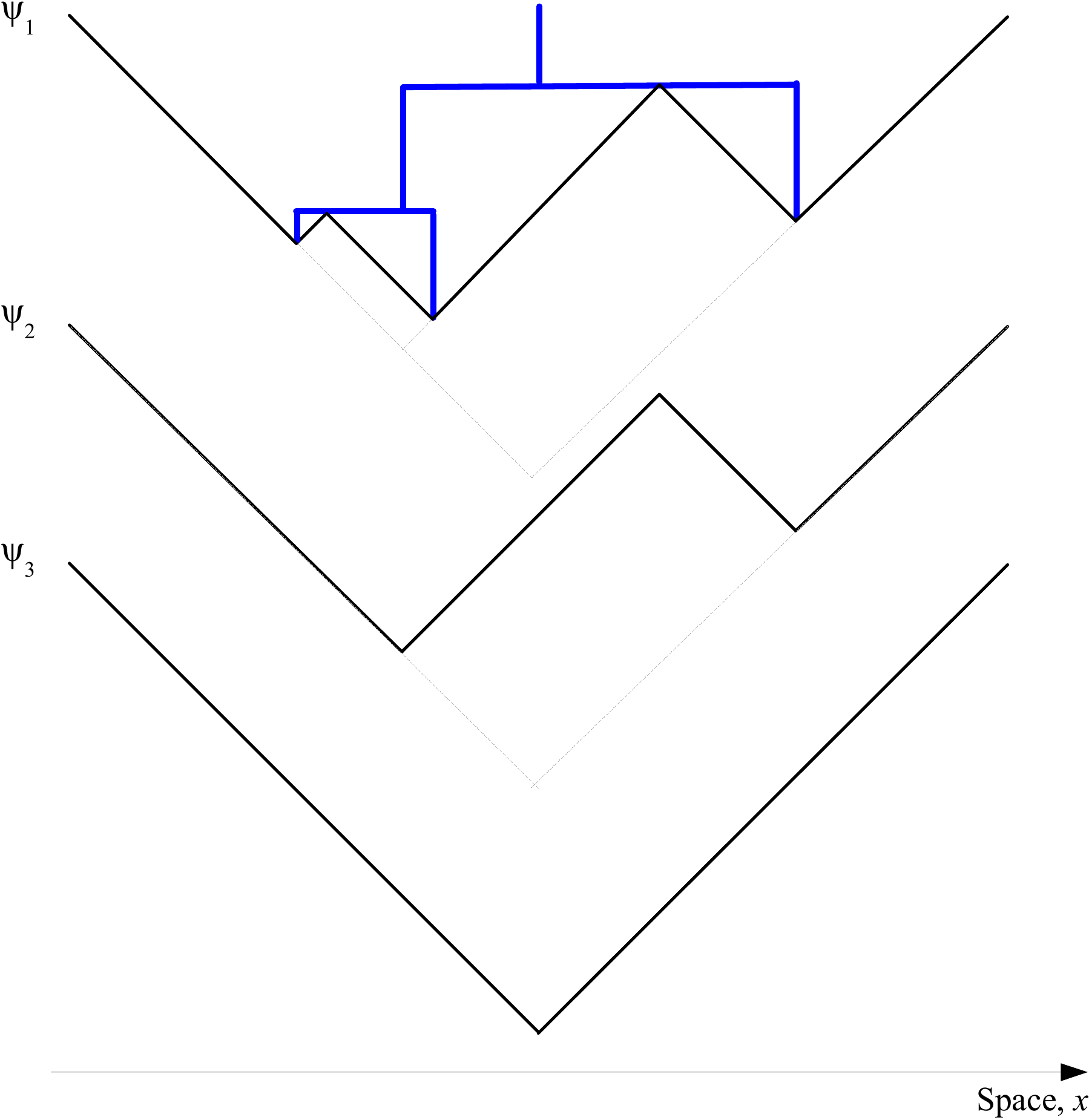}
  \caption{Potential unfolding: an illustration. The potential $\Psi_2(x)$
  is an unfolding of $\Psi_1(x)$, and $\Psi_3(x)$ is an unfolding of $\Psi_2(x)$.
  The shock tree in the phase space is shown (blue segments) next to the 
  initial potential
  $\Psi_1(x)$.}
  \label{fig:unfold}
  \end{figure}

Consider now a set of unfolded potentials $\Psi_i(x)$,
$i=3,\dots,n$, 
such that $\Psi_i(x)$ coincides with $\Psi_{i-1}(x)$ on
$[a,b]\setminus \cB_i$ and has a V-shaped negative excursion
on $\cB_{i-1}$, see Fig.~\ref{fig:unfold}.
By construction, the shortest basin within every $\Psi_i(x)$ is $\cB_i$; 
and this basin supports a W-shaped potential.
We apply the W-shaped potential algorithm to each
potential $\Psi_i(x)$ within the basin
$\cB_i$, hence consecutively
extending the shock tree construction to the space-time subsets 
\[\bigcup_{j=1}^{i}\cC_j,\quad i=3,\dots,n.\]
\noindent At instant $t_n$ there exists a single sink within
a V-shaped potential $\Psi_n(x)$ on $[a,b]$,
which is treated according to the V-shaped potential construction.
This completes the space-time tree construction.

Figure~\ref{fig:C} illustrates the above process for a 
potential with 4 local maxima.
The space-time cones $\cC_i$, $i=1,\dots,5$ are labeled in the figure.
Here, the largest cone $\cC_5$ corresponds to the entire space-time
system's domain.

Observe that the graphical shock trees $\cG^{(x,\psi)}$ and $\cG^{(x,t)}$ in 
the phase space 
and in the space-time domain have
the same combinatorial structure and planar embedding, coinciding with that of 
$S(\Psi_0)$ (recall that embedding
only involves ordering between the offspring of the same parent, and is 
different from a particular graphical representation of a tree);
see Figs.~\ref{fig:V},\ref{fig:W},\ref{fig:tree}.

  \begin{figure}[hbt]
  	\centering
	\includegraphics[width=0.7\textwidth]{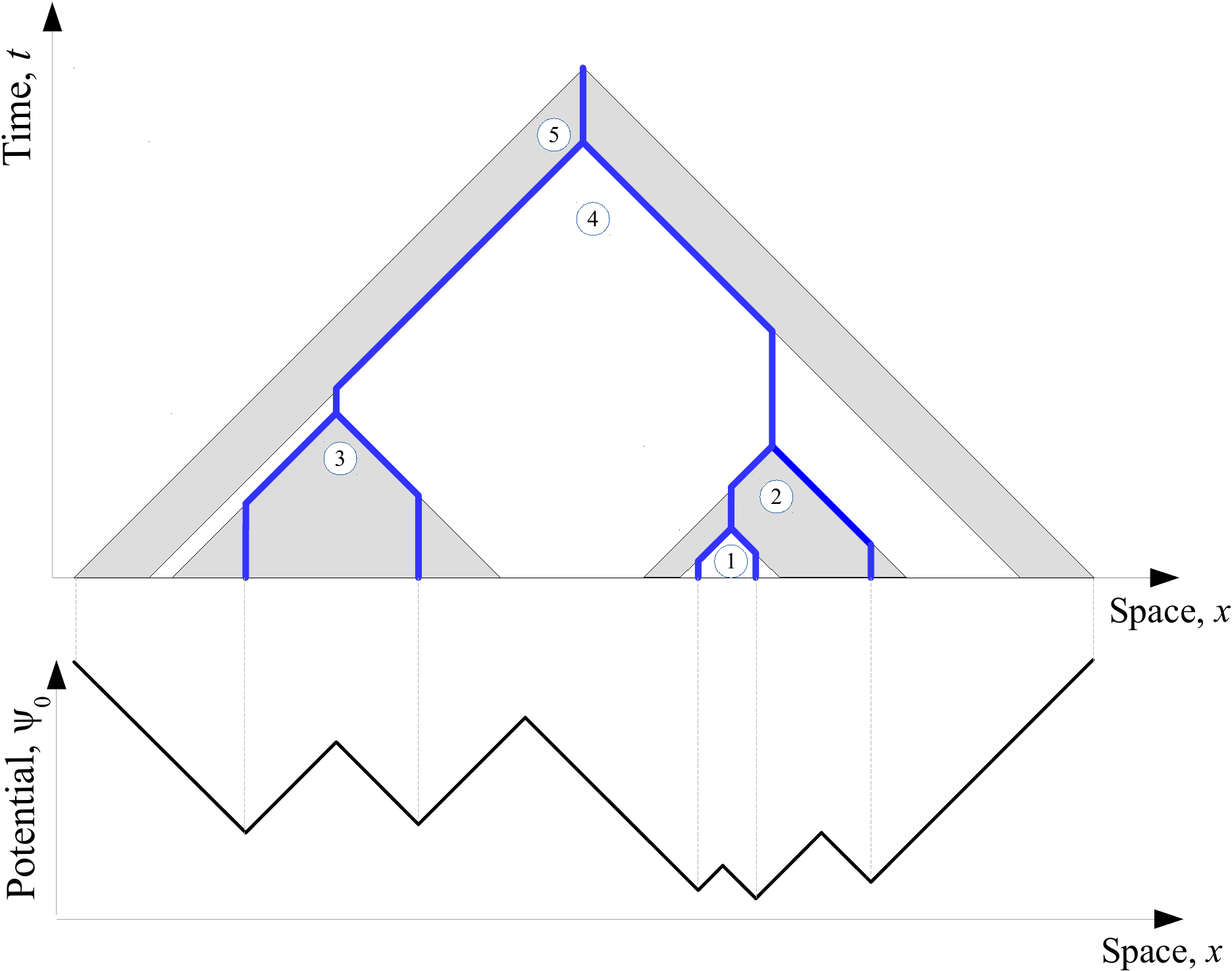}
  \caption{Iterative solution construction: 
  an illustration for a potential $\Psi_0(x)$ (bottom panel) with four local maxima. 
  (Top): Space-time cones $\cC_1,\dots,\cC_4$ that correspond to the basins $\cB_1,\dots,\cB_4$.
  Blue segments show the shock tree. 
  The cone $\cC_5$ corresponds to the V-shaped potential on the whole space interval.}
  \label{fig:C}
  \end{figure}

{\bf Combinatorial tree structure and planar embedding.}
By construction, there is one-to-one correspondence between internal local 
maxima of $\Psi_0(x)$ and internal non-root vertices of $S(\Psi_0)$. 
There is also a one-to-one correspondence between local minima and the leaves.
We label the tree vertices by indices $j$ that correspond to the 
local extrema $x_j$ of $\Psi_0(x)$.

Each internal non-root vertex $j$ has degree 2.
To find its offspring, consider the basin $\cB_j$ of a local maximum $x_j$.
It is sub-divided into two sub-basins by point $x_j$.
The right (left) offspring of $j$ corresponds to the internal 
local extrema with the maximal value within the right (left) sub-basin.
In case of a V-shaped potential, this is the only local minimum.
Otherwise, it is the internal local maximum with the maximal value.

We write ${\sf parent}(i)$ for the index of the single parent of the vertex $i$;
${\sf right}(i)$ and ${\sf left}(i)$ for the indices of the right and left
offspring of an internal non-root vertex $i$;
and ${\sf sibling}(i)$ for the index of the sibling of vertex $i$. 

{\bf Metric tree structure in the phase space.}
We let $c_i$ be the center of the basin $\cB_i$ and
\[\v_i = \Psi_0(x_{{\sf parent}(i)})-\Psi_0(x_i),\quad
\h_i = |\cB_{{\sf sibling}(i)}|/2.\]
The length $l_i$ of the parental edge of a non-root vertex $i$ 
within $S(\Psi_0)$ is given by 
$l_i = \v_i + \h_i.$
The shock tree $S(\Psi_0)$ can be considered a length metric space 
with the Euclidean length measured along the tree branches.
Figure~\ref{fig:tree} shows graphical shock trees $\cG^{(x,\psi)}$ 
and $\cG^{(x,t)}$ for an initial potential with two local maxima and
three local minima and illustrates the labeling of
vertical ($\v_i$) and horizontal ($\h_i$) segments of the tree.

{\bf Graphical shock tree in the phase space.}
Here we give a concise description of the graphical tree $\cG^{(x,\psi)}(\Psi_0)$.

The shortest description is the following:
\[\cG^{(x,\psi)}(\Psi_0) = \{(\sigma,\psi(\sigma,t))
\text{ such that there exists a sink } \sigma(t)\},\]
with the length of a segment between points $(x,\psi(x,t))$ and $(y,\psi(y,s))$
given by $|t-s|$.
Since all segments of the tree consist of horizontal and vertical lines,
the segment lengths can be measured in Euclidean metric on the plane.

A longer, albeit more constructive, description states that the tree
$\cG^{(x,\psi)}(\Psi_0)$
is the union of the following vertical and horizontal segments:
\begin{itemize}
\item[$(\v_{\rm min})$] For every local minimum $x_j$ of $\Psi_0(x)$ there exists a vertical 
segment from $(x_j,\Psi_0(x_j))$ to $(x_j,\Psi_0(x_j)+\v_j)$.
\item[$(\v_{\rm max})$] For every local maximum $x_j$ of $\Psi_0(x)$ there exists a vertical
segment from $(c_j,\Psi_0(x_j))$ to $(c_j,\Psi_0(x_j)+\v_j)$.
\item[(h)] For every local maximum $x_j$ of $\Psi_0(x)$ there exists a horizontal
segment of length $\h_j$ 
from $(c_{{\sf left}(j)},\Psi_0(x_j))$ to $(c_{{\sf right}(j)},\Psi_0(x_j))$.
\end{itemize}

The following statement summarizes the correspondence between
the dynamics of the sinks and the graphical tree $\cG^{(x,\psi)}(\Psi_0)$.
An analogous statement holds for the graphical shock tree 
$\cG^{(x,t)}(\Psi_0)$ 
in the space-time domain. 

\begin{lem}[{{\bf Shock tree}}]
\label{lem:SWT}
Let $\cG\equiv\cG^{(x,\psi)}(\Psi_0)$ be the graphical shock tree  
of a continuum annihilation dynamics with potential $\Psi_0(x)\in\widetilde{\E}([a,b])$. 
The following statements hold:
\begin{itemize}

\item[a)] There exists a one-to-one correspondence between points 
$z\in \cG$
and space-time locations $(x_z,t_z)$ of sinks.
In particular, there exists a one-to-one correspondence between the 
sinks at instant $t=0$ and the leaves of $\cG$, and
a one-to-one correspondence between the instants when two
sink merge (and hence a new sink creates) and the 
internal vertices of $\cG$.

\item[b)] Every sink at any time can either be at rest and 
accumulate mass at rate 2, or move with a unit speed with no mass 
accumulation. 
A point on any vertical segment of $\cG$
corresponds to a sink at rest. 
A point on any horizontal segment of $\cG$
corresponds to a sink in motion.

\item[c)] Suppose a point $z\in\cG$ corresponds
to a sink with mass $m_z$ at location $(x_z,t_z)$.
Then $t_z$ equals the length from $z$ to any descendant leaf within $\cG$.
The mass $m_z\le 2t_z$ equals double the total length of the 
vertical part of the tree $\cG$ descendant to $z$.
In particular, $m_z=2\,t_z$ if and only if $z$ is located on a vertical segment of $\cG$.

\item[d)] The horizontal segment of every vertex within $\cG$ has the length 
equal to the total length of the vertical segments within its sibling tree.
In other words, the time spent by a sink in motion equals
half the mass of the sink with which it collides.

\end{itemize}
\end{lem}
\begin{proof}
Follows from the construction of the tree
$\cG^{(x,\psi)}(\Psi_0)$ and Lemmas \ref{lem:W},\ref{lem:symmetry}.
\end{proof}

Consider a tree $\cV(\Psi_0)\in\L^{|}$ that has the same combinatorial 
structure as $S(\Psi_0)$, and with edge lengths given by $l_i=\v_i$.
Informally, this is a tree that consists of the vertical segments of 
the graphical tree $\cG^{(x,\psi)}(\Psi_0)$.

\begin{thm}[{\bf Shock tree is a level set tree}]
\label{thm:SWT}
Consider a potential $\Psi_0(x)\in\widetilde{\E}$
and the corresponding tree $\cV(\Psi_0)\in\L^{|}$. Then
the tree $\cV(\Psi_0)$ is isometric to the level set tree
of the negative potential $-\Psi_0(x)$:
\[\textsc{level}\left(-\Psi_0\right)=
\cV\left(\Psi_0\right).\]
\end{thm}
\begin{proof}
Follows from construction of the shock trees in this section
and construction of the level set tree in Sect.~\ref{sec:level}.
Considering a negative potential reflects the fact that the level set tree
is constructed top to bottom (leaves correspond to local maxima), and the 
shock tree is constructed bottom to top (leave correspond to local minima).
\end{proof}

\begin{figure}[h]
	\centering
	\includegraphics[width=0.7 \textwidth]{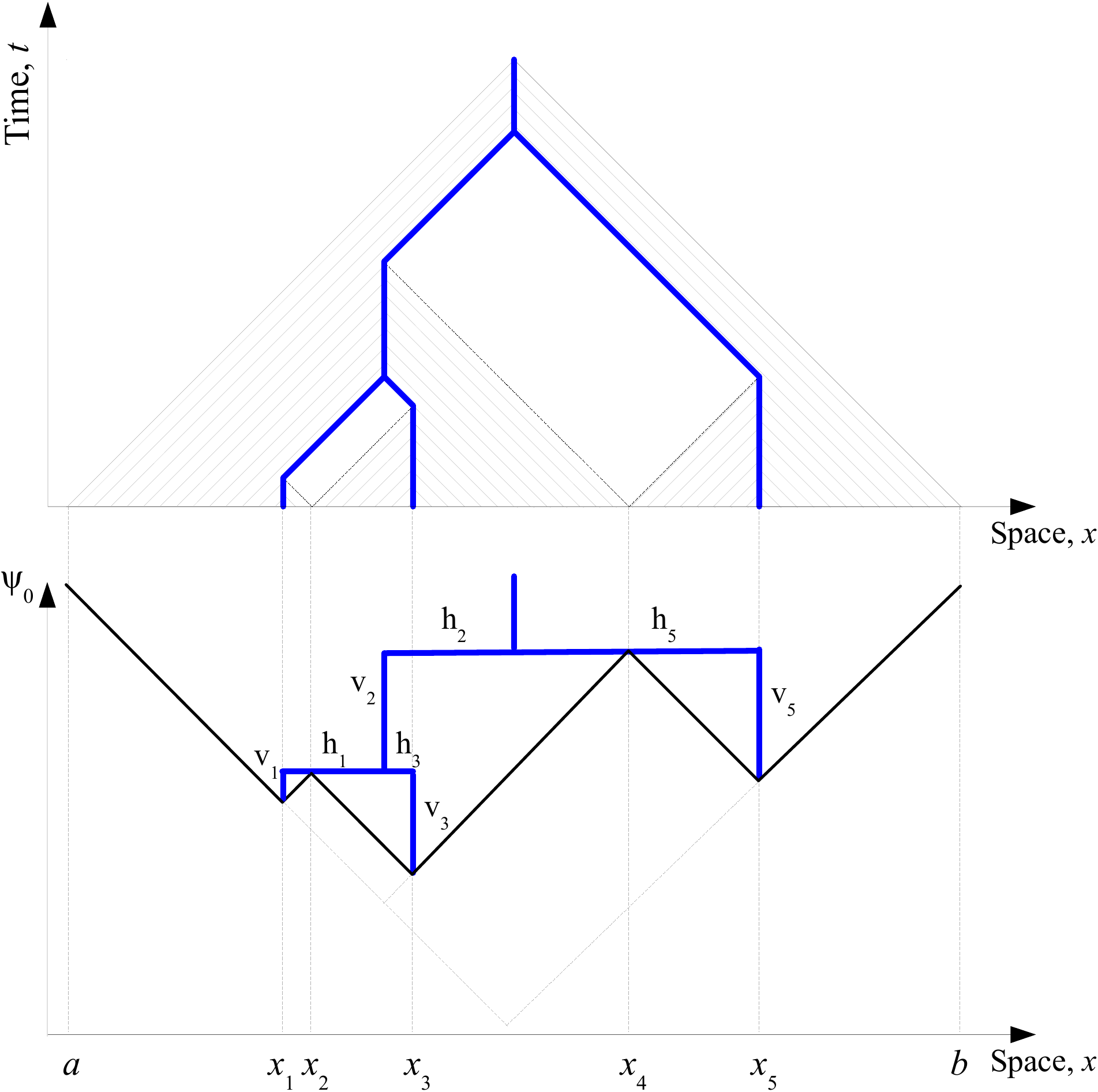}
	\caption{Shock tree for a piece-wise linear potential with two local 
	maxima.
	(Top): The shock tree in space-time domain (blue). Hatching illustrates 
	motion of 
	regular particles. There exist two empty rectangular areas, each corresponding to one
	of the local maxima.
	(Bottom): Potential $\Psi_0(x)$ (black) and the shock tree in the phase 
	space (blue).
	The figure illustrates the labeling of vertical ($\v_i$) and horizontal ($\h_i$) 
	segments of the tree.}
	\label{fig:tree}
\end{figure}

\subsection{Ballistic annihilation of a unit-slope potential}
\label{sec:Bdyn}
Suppose that a tree $T\in\L$ has a particular
graphical representation $\cG_T\in\mathbb{R}^2$ and let
$f:T\rightarrow\cG_T$ be a bijective root-preserving isometry.
We define the generalized dynamical pruning $\S(\varphi,\cG_T)$
for the graphical tree $\cG_T$ as the $f$-image of $\S(\varphi,T)$:
\[\S(\varphi,\cG_T) = f\left[\S(\varphi,T)\right].\]

Consider a natural isometry between the graphical shock tree
$\cG^{(x,t)}(\Psi_0)$ in the space-time domain and
tree $\cG^{(x,\psi)}(\Psi_0)$ in the phase space or 
the shock tree $S(\Psi_0)$. 
By construction (and using the natural isometry), the ballistic annihilation dynamics of 
the sinks corresponds to the continuous pruning (erasure at unit speed) of 
the shock  tree $S(\Psi_0)$, or its graphical representations, 
from the leaves.
Recall that this corresponds to selecting
$\varphi(T)=\textsc{height}(T)$ in the generalized dynamical pruning 
framework.

It is intuitive that the potential $\psi(x,t)$ at any given $t>0$ can be 
uniquely reconstructed from either of the pruned graphical trees 
$\cG^{(x,t)}(\Psi_0)$ and $\cG^{(x,\psi)}(\Psi_0)$.
Because of the strong symmetries (see Lemmas~\ref{lem:symmetry},
\ref{lem:SWT}d), 
the graphical trees possess significant redundant information.
We show below that the reduced tree $\cV(\Psi_0)$ equipped
with information about the sinks provides a minimal
description sufficient for reconstructing the entire continuum annihilation dynamics.

First, we expand the space $\L$ by equipping the trees with oriented 
massive sinks. 
Specifically, consider a tree $T\in\L$. 
A non-vertex point $z\in T$ can be equipped with an oriented mass
$m_{\rm L}>0$ (left) or $m_{\rm R}>0$ (right).
Every leaf vertex is necessarily equipped with either
a single (un-oriented) mass $m\ge0$, or
a pair of positive masses $(m_{\rm L},m_{\rm R})$.
Finally, we assume that that number of sinks within a tree
is finite.
The new space of mass-equipped trees is denoted by $\mL$.
Observe that any tree $T\in\L$ can be considered as an element 
of $\mL$ with no internal mass (no sinks attached to non-vertex points)
and zero mass at the leaves.

\medskip

\begin{con}[{{\bf Potential $\rightarrow$ tree}}]
\label{con1}
Suppose that $\psi(x,0)\in\widetilde{\E}$ and fix $t\in(0,t_{\rm max}]$. 
A mass-equipped tree $T_{\psi}(t)\in\mL$ for the potential $\psi(x,t)$ is 
constructed as follows:
\begin{itemize}
\item[(a)] The planar shape of the tree, as an element of $\L$, corresponds
to the level set tree of the potential restricted to the 
positive density domain: $\psi(x,t)|_{g(x,t)>0}$.
(This corresponds, for any given $t>0$, to cutting zero-density space intervals 
and glueing the potential segments from positive-density intervals to form a 
continuous function from $\widetilde{\E}$.)
\item[(b)] Every leaf that corresponds to a local minimum point of $\psi(x,t)$
is equipped with mass $m=2t$.
\item[(c)] Every leaf that corresponds to a local minimum plateau of length $\varepsilon$ in 
$\psi(x,t)$ is equipped with an arbitrary double mass $(m_{\rm L},m_{\rm R})$ that satisfies
\[m_{\rm L}+m_{\rm R}=\varepsilon+2t.\]
\item[(d)] Every internal point that corresponds to a plateau of length $\varepsilon$ that 
is not a local maximum is equipped with mass $m=v$. 
\end{itemize}
\end{con} 
\begin{rem}
The non-unique choices in item (c) reflect the memory loss of the ballistic annihilation dynamics.
\end{rem}

\begin{Def}
A tree $T\in\mL$ is called $t${\it-admissible} for a given $t\ge 0$ if and only if
its masses satisfy the following conditions:
\begin{itemize}
\item[(a1)] Any internal mass $m$ satisfies $m<2t$.
\item[(a2)] Any single mass $m$ at a leaf satisfies $m=2t$.
\item[(a3)] Any mass pair $(m_{\rm L},m_{\rm R})$ at a leaf satisfies 
$m_{\rm L}+m_{\rm R}>2t$.
\end{itemize}
\end{Def}

For any $t$-admissible tree $T\in\mL$ one can construct a unique corresponding  
potential $\psi_{T,t}(x)$.

\begin{con}[{{\bf Tree $\rightarrow$ potential}}]
\label{con2}
Suppose that $T\in\mL$ is a $t$-admissible tree for some $t>0$.
The corresponding potential $\psi_{T,t}(x)\in\widetilde{\E}$ is constructed in the following steps: 
\begin{itemize}
\item[(1)] Construct the Harris path $H_T(x)$ of $T$ as an element of $\L^{|}$.
\item[(2)] At every local minimum of $H_T(x)$ that corresponds to a double mass
$(m_{\rm L}, m_{\rm R})$, insert a horizontal plateau of length
\[\varepsilon = m_{\rm L}+m_{\rm R}-2t.\]
\item[(3)] At every monotone point of $H_T(x)$ that corresponds to an internal mass $m$, insert
a horizontal plateau of length $m$.
\item[(4)] At every local maxima, insert
a horizontal plateau of length $2t$.
\end{itemize}
\end{con}

We now extend the definition of the generalized dynamical pruning to
the mass-equipped trees.
\begin{Def}[{{\bf Cuts}}]
\label{def:cuts}
The set $\D(\varphi,T)$ of cuts in a pruned tree $\S(\varphi,T)$ is defined as
the boundary of the pruned part of the tree
\[\D(\varphi,T)=\partial\{x\in T~:~\varphi(\Delta_{x,T})<t\}.\]
\end{Def}
\noindent Accordingly, the set of tree cuts is a union of the leaves of the 
pruned tree and the vertices of the initial tree that became edge points in 
the pruned tree. 

Every cut point $z\in\S(\varphi,T)$ is equipped with either single or double mass.
If the root of the subtree $\Delta_{z,T}$ has degree 1, the point $z$
is equipped with mass $m_z$ equal to the value of the pruning function $\varphi$ 
at the subtree $\Delta_{z,T}$.
If the root of the subtree $\Delta_{z,T}$ has degree 2, 
the subtree $\Delta_{z,T}$ is comprised of two subtrees, $T_1$ and $T_2$,
that merge (and share the root) at $z$.
In this case, the point $z$
is equipped with mass(es) equal to the value(s) of $\varphi$
at every tree $T_i$ such that $\varphi(T_i)<t$.
We denote the mass-equipped generalized dynamical pruning by $\mS(\varphi,T)$. 

\medskip

The following theorem establishes the equivalence of the continuum annihilation dynamics and mass-equipped generalized dynamical pruning with respect 
to the tree length.

\begin{figure}[!h]
	\centering
	\includegraphics[width=0.7 \textwidth]{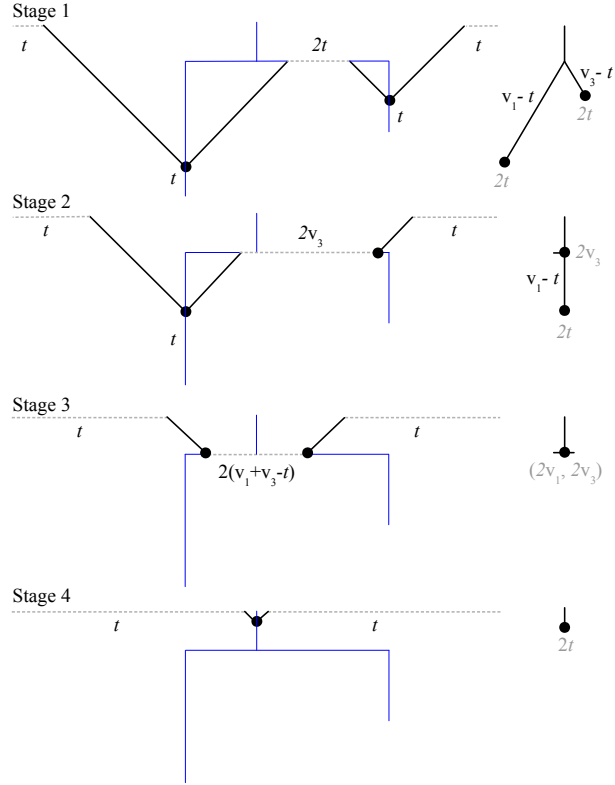}
	\caption{Four generic stages in the ballistic annihilation dynamics of a W-shaped potential (left), 
	and respective mass-equipped trees (right).
	(Left): Potential (solid black) $\psi(x,t)$.
	Each plateau (dashed gray) corresponds to an empty interval.
	The graphical shock tree $\cG^{(x,\phi)}(\Psi_0)$ (blue)
	and sinks (black circles) are shown for visual convenience.
	(Right): Mass-equipped trees. 
	Segment lengths are marked in black, point masses are indicated in gray.
	A tick next to the mass point in Stage 2 indicates its orientation (left).
	A double-tick next to the mass point in Stage 3 indicates two masses
	at the leaf.}
	\label{fig:Vtree}
\end{figure}

\begin{thm}[{{\bf Annihilation pruning}}]
\label{thm:pruning}
Suppose that $\psi(x,0)\in\widetilde{\E}$. 
For any fixed $t\ge 0$,
let $T_{\psi}(t)$ be the mass-equipped tree for the
potential $\psi(x,t)$.
Then the following statements hold: 
\begin{itemize}
\item[(a)] The tree $T_{\psi}(t)$ is $t$-admissible, and it uniquely 
reproduces the potential:
\[\psi(x,t) = \psi_{T_{\psi}(t),t}(x).\]
\item[(b)] The ballistic annihilation dynamics of the potential $\psi(x,t)$ corresponds
to a generalized mass-equipped dynamical pruning of the tree 
\[T_{\psi}(0)=\textsc{level}(\psi(x,0)),\]
considered as an element of $\mL$, with $\varphi(T)=\textsc{length}(T)$:
\[T_{\psi}(t) = \mS(\textsc{length},T_{\psi}(0)).\]
\end{itemize}
\end{thm}

\begin{proof}
We start by proving the statement for a W-shaped potential.
Consider a W-shaped potential $\Psi_0(x)$ with two local minima at $x_i$, $i=1,3$
and the local maximum at $x_2$ (see Fig.~\ref{fig:W}). 
We define (as before) $\v_i = \Psi_0(x_2)-\Psi_0(x_i)$ and
set $v_{\rm min} = \v_1\wedge\v_3$, $v_{\rm max} = \v_1\vee\v_3$. 
There exist four stages in the ballistic annihilation dynamics of a W-shaped potential,
illustrated in Fig.~\ref{fig:Vtree}.
\begin{itemize}

\item[{\bf 1:}] During the time interval $t\in (0,\v_{\rm min})$, there exist 
two sinks (at rest) located in the local minima of $\psi(x,t)$.
This corresponds to a Y-shaped tree $T_{\psi}(t)$ with no internal sinks
and mass $m=2t$ at each of the leaves.

\item[{\bf 2:}] During the time interval $t\in [\v_{\rm min},\v_{\rm max})$, 
there exist two sinks. 
One is at rest and is located in the (only) local minima of $\psi(x,t)$;
and the other moves and is located next to a plateau. 
This corresponds to an I-shaped tree $T_{\psi}(t)$ with an internal 
mass $m_{\rm R}=2\v_{\rm min}$
placed at distance $\v_{\rm max}-t$ from the leaf and mass $m=2t$ at the leaf.

\item[{\bf 3:}] During the time interval $t\in [\v_{\rm max},\v_1+\v_3)$, 
there exist two sinks in motion.
They both are located at the boundaries of the plateau.
This corresponds to an I-shaped tree $T_{\psi}(t)$ with no internal mass
and a double mass $(2\v_1,2\v_3)$ at the leaf.

\item[{\bf 4:}] During the time interval $t\in [\v_1+\v_3,t_{\rm max})$, 
there exists a single sink at rest.
It is located at the local minimum of the potential.
This corresponds to an I-shaped tree $T_{\psi}(t)$ with no internal mass
and a single mass $m=2t$ at the leaf.
\end{itemize}

The statements (a) and (b) for a W-shaped potential are verified by direct 
observation in each of the four stages, and using Constructions~\ref{con1},\ref{con2}. 

The general case is now treated by considering the nested W-shaped potentials via the unfolding procedure,
as in the sequential cone-based construction of the solution to the continuum annihilation dynamics in Section~\ref{sec:solution}.
\end{proof}

\section{Real tree description of ballistic annihilation}
\label{sec:Rtree}
The main object of this study is a finite tree of sinks  (shock tree) produced by the
ballistic annihilation model, and its dynamical pruning. 
This section suggests a natural description of the model in terms of {\it real trees}
introduced in Sect.~\ref{sec:Rsetup}. 
A {\it real tree} that completely describes the model dynamics and is tightly connected 
to the shock tree is constructed in Sect.~\ref{sec:Rfull}.
Section Sect.~\ref{sec:Rdomain} introduces two complementary metric space representations of the system's domain $[a,b]$
that summarize the essential features of ballistic coalescence.
Section Sect.~\ref{sec:Rprune} discusses a natural approach to introducing 
prunings on $\mathbb{R}$-trees.

\begin{figure}[h]
	\centering
	\includegraphics[width=0.7 \textwidth]{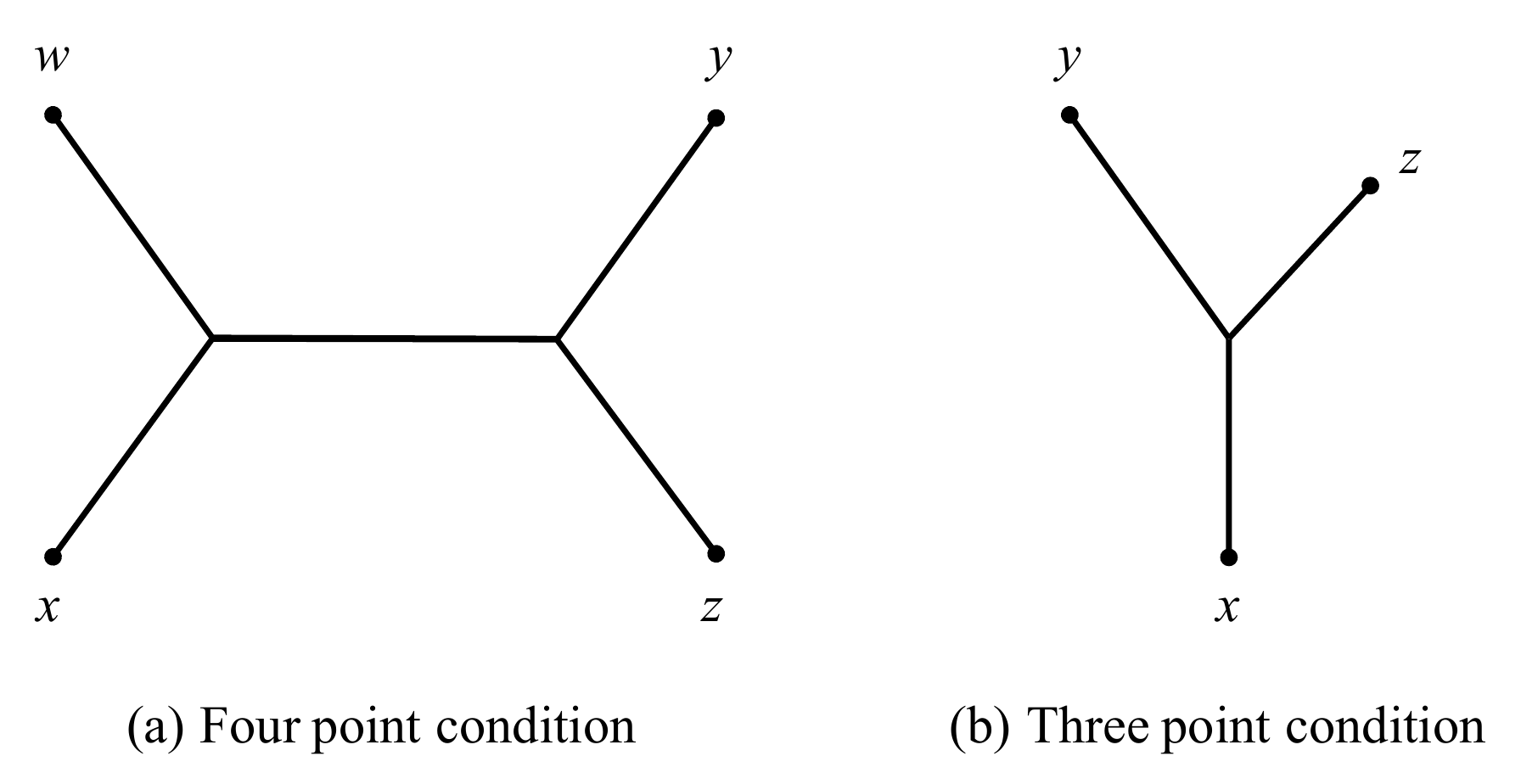}
	\caption{Equivalent conditions for $0$-hyperbolicity of a metric space $(X,d)$.
	(a) Four point condition: any quadruple $w,x,y,z\in X$ is geodesically connected as shown in the figure. 
	This configuration is algebraically expressed in Eq.~\eqref{four}.
	(b) Three point condition: any triplet $x,y,z\in X$ is geodesically connected as shown in the figure. }
	\label{fig:hyp}
\end{figure}

\subsection{Real trees}
\label{sec:Rsetup}
Recall that a metric space $(X,d)$ is called $0$-hyperbolic, if any quadruple
$w,x,y,z\in X$ satisfies the following {\it four point condition} \cite[Lemma 3.12]{Evans2005}:
\be
\label{four}
d(w,x)+d(y,z)\le\max\{d(w,y)+d(x,z),d(x,y)+d(w,z)\}.
\ee
The four point condition is an algebraic description of an intuitive geometric constraint
on geodesic connectivity of quadruples that is shown in Fig.~\ref{fig:hyp}(a).
An equivalent way to define $0$-hyperbolicity is the three point condition
illustrated in Fig.~\ref{fig:hyp}(b).
It is readily seen that the four point condition is satisfied by any finite tree with 
edge lengths (considered as a metric space with segment lengths induced by the edge lengths).
In general, a connected and $0$-hyperbolic metric space is called a {\it real tree}, 
or $\mathbb{R}$-tree \cite[Theorem 3.40]{Evans2005}.
We denote a real tree by $(T,d)$, referring to the underlying space $T$ and metric $d$, 
respectively.
A real tree $(T,d)$ is {\it geodesically linear}, which means that for any two points $x,y\in T$ 
there exists a unique segment (an isometry image) within $T$ with endpoints $\{x,y\}$
\cite[Definition 3.2]{Evans2005}.
We denote this segment by $[x,y]\subset T$. 
A real tree is called {\it rooted} if one of its points, denoted here by $\rho_T$, is selected as the
tree root.
Similarly to the case of finite trees, we say that a point $p\in T$ is an {\it ancestor} of
point $q\in T$ if the segment with endpoints $q$ and $\rho$ includes $p$: 
$p\in [p,\rho]\subset T$.
In this case, the point $q$ is called a {\it descendant} of point $p$. 
We denote by $\Delta_{p,T}$ the descendant tree at point $p$, that is the set 
of all descendants of point $p\in T$, including $p$ as the tree root.
The set of all descendant leaves of point $p$ is denoted by $\Delta^{\circ}_{p,T}$.

\subsection{$\mathbb{R}$-tree representation of ballistic annihilation}
\label{sec:Rfull}
We construct here a real tree 
representation of the continuum ballistic annihilation model.
As before, we assume a unit particle density 
$g(x)\equiv 1$ and initial potential $\Psi_0(x)\equiv \psi(x,0)\in\widetilde{\cE^{\rm ex}}$,
i.e. $\Psi_0(x)$ is a unit slope negative excursion on a finite interval $[a,b]$, as illustrated
in the bottom panel of Fig.~\ref{fig:shock_tree}. 
This model was discussed in Sect.~\ref{sec:solution}.
Recall that the interval $[a,b]$ completely annihilates by time $t_{\rm max}=(b-a)/2$,
producing a single sink at space-time location $((b+a)/2,t_{\rm max})$.

Consider the model's entire space-time domain
$\mathbb{T}=\mathbb{T}(\Psi_0)$
that consists of all points of the form $(x,t)$, $x\in[a,b]$, $0\le t\le t_{\rm max}$, 
such that there exists either a particle or a sink at location $x$ at time instant $t$. 
The shaded (hatched) regions in the top panels of Figs.~\ref{fig:shock_tree},\ref{fig:V},
\ref{fig:W} are examples of such sets of points.
For any pair of points $(x,t)$ and $(y,s)$ in $\mathbb{T}$, we define their unique
{\it earliest common ancestor} as a point 
\[{\sf A}_{\mathbb{T}}((x,t),(y,s)) = (z,w)\in\mathbb{T}\]
such that
$w$ is the infimum over all $w'$ such that 
\[\exists\,z'\,:\, 
\{(x,t),(y,s)\}\in\Delta_{(z',w'),\mathbb{T}}.\]
The length of the unique segment between the points $(x,t)$ and $(y,s)$ is defined as 
\be
\label{T_len}
d\left((x,t),(y,s)\right)=\frac{1}{2}\left((w-t)+(w-s)\right)=\frac{1}{2}(2w-s-t),
\ee
where $w$ is the time component of
$(z,w)={\sf A}_{\mathbb{T}}((x,t),(y,s))$.

The tree $(\mathbb{T},d)$ for a simple initial potential is illustrated in 
the top panel of Fig.~\ref{fig:shock_tree} by gray lines.
The tree has a relatively simple structure.
There is a one-to-one
correspondence between the initial particles $(x,0)$, $x\in[a,b]$, and the leaf vertices
of $\mathbb{T}$.
There is a one-to-one correspondence between the ballistic runs of the initial particles
(runs before collision and annihilation) and the leaf edges of $\mathbb{T}$.
Four of such runs are shown by green arrows in Fig.~\ref{fig:Rtree}.
There is one-to-one correspondence between the sink points $(\sigma(t),t)$ and the non-leaf part
of $\mathbb{T}$.
In particular, the tree root corresponds to the final sink $((a+b)/2,t_{\rm max})$.
The sink points are shown by blue line in Fig.~\ref{fig:shock_tree}. 
It is now straightforward to check that the tree $(\mathbb{T},d)$ satisfies the
four point condition.

Consider again the sink subspace of $\mathbb{T}$, which consists of the points 
$\{\sigma(t),t)\}$ such that there exists a sink at location $\sigma(t)$ at
time instant $t$, equipped with the distance \eqref{T_len}.
This metric subspace is also a tree, as a connected subspace of 
an $\mathbb{R}$-tree \cite{Evans2005}.
This tree is isometric to 
the shock wave tree $S(\Psi_0)$ and hence to either of its graphical 
representations $\cG^{(x,t)}(\Psi_0)$ or 
$\cG^{(x,\psi)}(\Psi_0)$ that are
illustrated in Fig.~\ref{fig:shock_tree} (top and bottom panels, respectively).

From the above construction, it follows that all leaves $(x,0)$ are located at the
same depth (distance from the root) $t_{\rm max}$.
To see this, consider the segment that connect a leaf and the root and apply \eqref{T_len}.
Moreover, each {\it time section} at a fixed instant $t_0$, 
${\sf sec}(\mathbb{T},t_0)=\{(x,t_0)\in\mathbb{T}\}$,
is located at the same depth $(t_{\max}-t_0)$. 
This implies, in particular, that for any fixed $t_0\ge 0$, the metric induced 
by $\mathbb{T}$ on ${\sf sec}(\mathbb{T},t_0)$ is an {\it ultrametric}, which means 
that $d_1(p,q)\le d_1(p,r)\vee d_1(r,q)$ for 
any triplet of points $p,q,r\in{\sf sec}(\mathbb{T},t_0)$.
Accordingly, each triangle $p,q,r\in{\sf sec}(\mathbb{T},t_0)$
is an {\it isosceles}, meaning that at least two of the three pairwise distances 
between $p,q$ and $r$ are equal and not greater than the third \cite[Definition 3.31]{Evans2005}.
The length definition \eqref{T_len} implies that the distance between
any pair of points from any fixed section ${\sf sec}(\mathbb{T},t_0)$ equals
the time until the two points (each of which can be either a particle or a sink) 
collide.

We notice that the collection of leaf vertices $\Delta^{\circ}_{p,\mathbb{T}}$ descendant 
to a point $p\in\mathbb{T}$ can be either a single point $(x_p,0)$, if $p$ is within
a leaf edge and represents the ballistic run of a particle, or 
an interval $\{(x,0): x_{\rm left}(p) \le x \le x_{\rm right}(p)\}$,
if $p$ is a non-leaf point that represents a sink.
We define the {\it mass} $m(p)$ of a point $p\in\mathbb{T}$ as
\[m(p)=\int\limits_{x:(x,0)\in\Delta^{\circ}_{p,\mathbb{T}}}g(z)dz 
= x_{\rm right}(p)-x_{\rm left}(p),\]
where the last equality reflects the assumption $g(z)\equiv 1$.
The mass $m(p)$ generalizes the quantity ``number of descendant leaves'' 
(see Example~\ref{ex:numL}) to the $\mathbb{R}$-tree situation
with an uncountable set of leaves.
We observe that
(i) a point $p\in\mathbb{T}$ represents a ballistic run if and only if $m(p)=0$;
(ii) a point $p\in\mathbb{T}$ represents a sink if and only if $m(p)>0$. 
This means that the shock wave tree, which is isometric to
the sink part of the tree $(\mathbb{T},d)$, can be extracted from  
$(\mathbb{T},d)$ by the condition $\{p:m(p)>0\}$.


\subsection{Metric spaces on the set of initial particles}
\label{sec:Rdomain}
In this section we discuss two metrics on the system's domain $[a,b]$,
which is isometric to the set $\{(x,0): x\in[a,b]\}$, of initial particles.
These spaces contain the key information about 
the system dynamics and, unlike the complete tree $(\mathbb{T},d)$,  
can be readily constructed from the potential $\Psi_0(x)$.
 
Metric $h_1(x,y)$ reproduces the ultrametric induced by 
$(\mathbb{T},d)$ on $[a,b]$.
Below we explicitly connect this metric to $\Psi_0(x)$. 
For any pair of points $x,y\in[a,b]$ 
we define a basin ${\sf B}_{\Psi_0}(x,y)$ as the interval that supports the 
minimal negative excursion within $\Psi_0(x)$ 
that contains the points $x,y$.
Formally, assuming (without loss of generality) that $x<y$ we find the maximum 
of $\Psi_0$ on $[x,y]$:
\[{\sf m}_{\Psi_0}(x,y) = \sup\limits_{z\in[x,y]}\Psi_0(z)\]
and use it to define the basin
\[{\sf B}_{\Psi_0}(x,y) = \left[\sup\limits_{z\le x, \Psi_0(z)\ge {\sf m}_{\Psi_0}(x,y)} z,
\quad\inf\limits_{z\ge y, \Psi_0(z)\ge {\sf m}_{\Psi_0}(x,y)} z\right].\]
The metric is now defined as
\[h_1(x,y)=\frac{1}{2}|{\sf B}_{\Psi_0(x,y)}|.\]
It is straightforward to check that 
\[h_1(x,y) = \text{the time until collision of the particles } (x,0) \text{ and } (y,0),\]
where the collision is understood as either collision of particles, collision of
sinks that annihilated the particles, or collision between a sink that annihilated
one of the particles 
and the other particle.
For instance, the claim is readily verified, by examining the bottom panel of Fig.~\ref{fig:Rtree},
for any pair of points from the set $\{x,x',y,y'\}$.
The metric space $([a,b],h_1)$ is not a tree.
Moreover, this space is {\it totally disconnected}, since there only exists a finite number
of points (local minima of $\Psi_0(x)$) that have a neighborhood of arbitrarily small size.
Any other point at the Euclidean distance $\epsilon$ from the nearest local minimum is separated
from other points by at least $\epsilon/2$.

Metric $h_2(x,y)$ describes the mass accumulation by sinks during the annihilation process.
Specifically, we introduce an equivalence relation among the annihilating
particles, by writing $x\sim_{\Psi_0}y$ 
if the particles with initial coordinates $x$ and $y$ collide and annihilate with each other.
For example, in Fig.~\ref{fig:Rtree} we have $x \sim_{\Psi_0} x'$ 
and $y \sim_{\Psi_0}y'$.
The following metric is now defined on the quotient space $[a,b]|_{\sim_{\Psi_0}}$:
\[h_2(x,y) = 2\sup\limits_{z\in[x,y]}\left[\Psi_0(z)\right]-\Psi_0(x)-\Psi_0(y).\]
In words, the distance $h_2(x,y)$ between particles $x$ and $y$ equals to the 
total mass accumulated by the sinks to which the particles belong during the
time intervals between the instants when the particles joined the respective
sinks and the instant of particle (or respective sink) collision.
Another interpretation is that $h_2(x,y)$ equals to the minimal Euclidean 
distance between points $x,y\in[a,b]|_{\sim_{\Psi_0}}$ in the quotient space; one can travel in this quotient space
as along a regular real interval, with a possibility to jump (with no distance accumulation)
between equivalent points.  
This $\mathbb{R}$-tree construction is know as the {\it tree in continuous path}
\cite[Definition 7.6]{Pitman},\cite[Example 3.14]{Evans2005}.

The metric space $([a,b]|_{\sim_{\Psi_0}},h_2)$ is a tree that is 
isometric to the level set tree of the potential $\Psi_0(x)$ on $[a,b]$
and hence to the (finite) shock wave tree $\cV(\Psi_0)$ (by Theorem~\ref{thm:SWT}), with 
the convention that the root is placed 
in $a\sim_{\Psi_0} b$.
This means, in particular, that prunings of these two trees,  with the same pruning function
and pruning time, coincide.

\subsection{Other prunings on $\mathbb{T}$}
\label{sec:Rprune}
One can introduce a large class of prunings on an $\mathbb{R}$-tree $(\mathbb{T},d)$ 
following the approach used above to define the point mass $m(p)$.
Specifically, consider a measure $\eta(\cdot)$ on $[a,b]$
and define $m_{\eta}(p) = \eta(\Delta^{\circ}_{p,\mathbb{T}})$.
The function $m_{\eta}(p)$ is non-decreasing along each segment
that connect a leaf and the root $\rho_{\mathbb{T}}$ of $\mathbb{T}$.
Hence, one can define a pruning with respect to $m_{\eta}$ on $\mathbb{T}$
by cutting all points $p$ with $m_{\eta}(p)<t$
for a given $t\ge 0$.
Clearly, the function $m_{\eta}(p)$ typically has discontinuities along a
path between a leaf and the root of $\mathbb{T}$.
This means that pruning with respect to $m_{\eta}$ typically does not
have the semigroup property.

\section{Ballistic annihilation of an exponential excursion}
\label{sec:BDEE}

This section analyses a special case of a piece-wise
linear potential with unit slopes: a negative exponential excursion.
Consider potential $\psi(x,0)=-H_{{\sf GW}(\lambda)}(x)$ 
that is the negative Harris path of an exponential critical binary 
Galton-Watson tree with parameter $\lambda$.
In words, the potential is a negative finite excursion with linear segments of
alternating slopes $\pm 1$, such that the lengths of all segments except
the last one are i.i.d. exponential random variables with parameter $\lambda/2$.
This means that the initial particle velocity $v(x,0)$ 
alternates between the values $\pm1$ at epochs of a stationary 
Poisson point process on $\mathbb{R}$ with rate $\lambda/2$, starting with $+1$ and until the 
respective potential crosses the zero level.

\begin{cor}[{{\bf Exponential excursion}}]
\label{cor:GW}
Let potential $\Psi_0(x)\in\widetilde{\E}$
be a negative exponential excursion,
$\Psi_0(x)=-H_{{\sf GW}(\lambda)}(x)$. 
Then the corresponding shock tree $\cV(\Psi_0)\in\L^{|}$
is an exponential binary critical Galton-Watson tree 
${\sf GW}(\lambda)$.
\end{cor}
\begin{proof}
By Theorem~\ref{thm:SWT}, the shock tree $\cV(\Psi_0)$ is the level 
set tree of the negative potential $-\Psi_0(x)$.
The statement now follows from Theorem~\ref{Pit7_3}.
\end{proof}

By Theorem~\ref{thm:pruning} the ballistic annihilation dynamics of an initial 
potential $\psi(x,0)$ is equivalent to a generalized dynamical pruning of the 
respective mass-equipped tree $T_{\psi}(0)$ with pruning function $\varphi(T)=\textsc{length}(T)$.
Here we give a complete description of the ballistic annihilation dynamics for a mass-equipped exponential 
critical binary Galton-Watson tree $T_{\psi}(t)\in\mL$. 

Recall that if $T={\sf GW}(\lambda)$ and $\varphi(T) = \textsc{length}(T)$, then by (\ref{eq:pDelta1}),
$$p_t:={\sf P}(\varphi(T)>t)=e^{-\lambda t}\Big[ I_0(\lambda t)+ I_1(\lambda t) \Big].$$

\begin{thm}[{{\bf Ballistic annihilation dynamics of an exponential excursion}}] 
\label{thm:annihilation}
Suppose the initial potential $\psi(x,0)$ is the negative Harris path 
of an exponential critical binary Galton-Watson tree 
with parameter $\lambda$, $T=T_\psi(0)\stackrel{d}{=}{\sf GW}(\lambda)$.
Then, at any instant $t>0$ the mass-equipped shock tree 
$T_{\psi}(t) =  \mS(\textsc{length},T_\psi(0))$ conditioned on surviving,
$T_{\psi}(t) \not= \phi$, is distributed according to the 
following rules.
\begin{enumerate}
  \item[(i)] $T_{\psi}(t)\stackrel{d}{=}{\sf GW}(\lambda_t)$ with
  $\lambda_t:=\lambda p_t$.
  
  \item[(ii)] A single or double mass points are placed independently in each leaf 
  with the probability of a single mass being
  \[{2 \over \lambda}{\ell(t) \over p_t^2}.\]

 \item[(iii)] Each single mass at a leaf has mass $m=2t$. 
 For a double mass point, the masses $(m_{\rm L},m_{\rm R})$ has the following joint p.d.f.
  $$f(a,b)={{1\over 4}\ell({a\over2})\ell({b\over2}) \over p_t^2 -{2  \over \lambda}\ell(t)}$$
  for $a,b>0$, $ a\vee b ~\leq 2t ~< a+b$.
  
  \item[(iv)] The number of mass points placed in the interior of any edge 
  is distributed geometrically with the probability of placing $k$ masses being
  $$p_t\big(1-p_t\big)^k, \qquad k=0,1,2,\hdots.$$
  The locations of $k$ mass points are independent uniform in the interior of the edge. 
  The orientation of each mass is left or right independently with probability $1/2$.
  
  \item[(v)]The edge masses are i.i.d. random variables with the following common p.d.f.
  $${\ell({a\over 2}) \over 2(1-p_t)}, \qquad a \in (0,2t).$$
\end{enumerate}
\end{thm}

\begin{proof}
Part (i) follows directly from Theorem \ref{pdelta}(a).
\medskip

To establish the other parts, we first introduce a particular
representation of the survival event $\S(\varphi,T) \not= \phi$. 
Let $X$ denote the length of the edge of $T$ adjacent to the root and let $x$ be the  descendent vertex (a junction or a leaf) to the root in $T$. 
If ${\sf deg}_T(x)=3$, let $\h_1$ and $\h_2$ represent the lengths of the two subtrees descendent from $x$. 
Then the event 
$\S(\varphi,T) \not= \phi$ can be written as the union of the following five non-overlapping events, illustrated in Fig.~\ref{fig:thm8},
\begin{align}\label{5events}
\Big\{\S(\varphi,T) \not= \phi\Big\} = & \{{\sf deg}_T(x)=3 ~\text{ and }~t\leq~\h_1 \wedge \h_2\} \nonumber \\
& \cup \{{\sf deg}_T(x)=3 ~\text{ and }~\h_1 \wedge \h_2 \leq ~t<~ \h_1 \vee \h_2\} \nonumber \\
& \cup \{{\sf deg}_T(x)=3 ~\text{ and }~ \h_1 \vee \h_2 \leq ~t<~ \h_1+\h_2\} \nonumber \\
& \cup \{{\sf deg}_T(x)=3 ~\text{ and }~ \h_1 +\h_2 \leq ~t<~ X+\h_1+\h_2\} \nonumber \\
& \cup \{{\sf deg}_T(x)=1 ~\text{ and }~t<~X\}.
\end{align}
The probabilities of the five events in (\ref{5events}) are computed below.

\begin{figure}[h] 
\centering\includegraphics[width=0.9\textwidth]{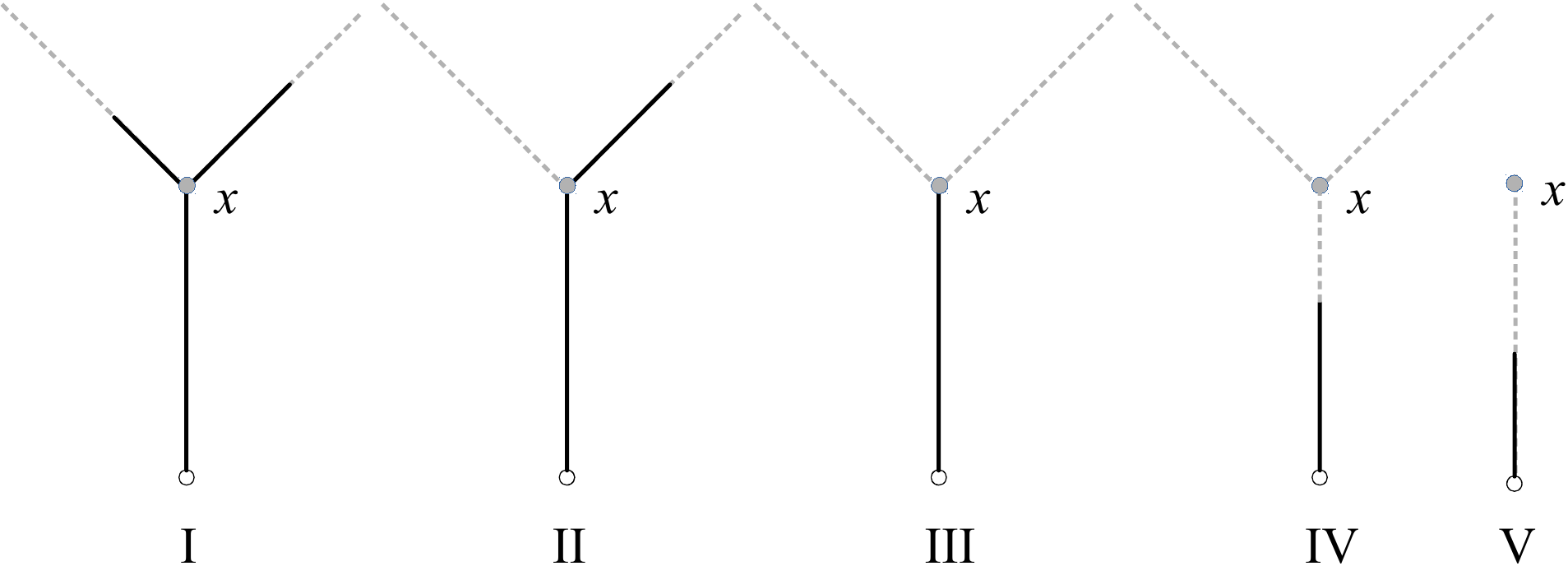}
\caption[Theorem 8: illustration]
{Subevents used in the proof of Theorem~\ref{thm:annihilation}.
Solid line depicts (a part of) pruned tree $\S(\varphi,T)$.
Dashed line depicts (a part of) initial tree $T$.}
\label{fig:thm8}
\end{figure} 

\begin{itemize}
  \item[Case I] 
  
  \begin{align}\label{eq:case1}
  {\sf P}({\sf deg}_T(x)=3 ~\text{ and }~t\leq~\h_1 \wedge \h_2) 
  & ={1 \over 2}{\sf P}(t\leq~\h_1 \wedge \h_2~|{\sf deg}_T(x)=3)\nonumber \\
  &={p_t^2 \over 2}.
  \end{align}
  
 \bigskip
 
  \item[Case II]  
  
  \begin{align}\label{eq:case2}
  {\sf P}({\sf deg}_T(x)=3 &~\text{ and }~\h_1 \wedge \h_2 \leq ~t<~ \h_1 \vee \h_2)\nonumber \\
  &={1 \over 2}{\sf P}(\h_1 \wedge \h_2 \leq ~t<~ \h_1 \vee \h_2~|{\sf deg}_T(x)=3)\nonumber \\
  &={\sf P}(\h_1 ~t<~ \h_2~|{\sf deg}_T(x)=3)\nonumber \\
  &=p_t\big(1 -p_t\big).
  \end{align}
  
 \bigskip
  
  \item[Case III] 
  
   \begin{align*}
  {\sf P}({\sf deg}_T(x)=3 &~\text{ and }~\h_1 \vee \h_2 \leq ~t<~ \h_1+\h_2) \\
  &={1 \over 2}{\sf P}(\h_1 \vee \h_2 \leq ~t<~ \h_1+\h_2~|{\sf deg}_T(x)=3) \\
  &={1 \over 2}{\sf P}(\h_1 \vee \h_2 \leq ~t~|{\sf deg}_T(x)=3) -{1 \over 2}{\sf P}(t<~ \h_1+\h_2~|{\sf deg}_T(x)=3)\\
  &={1 \over 2}\big(1 -p_t\big)^2 -{1 \over 2}F_{\h_1+\h_2}(t),
  \end{align*}
  
  where
  \[F_{\h_1+\h_2}(t):={\sf P}( \h_1+\h_2<~t~|{\sf deg}_T(x)=3)
  =\int\limits_0^t \ell \ast \ell(y) \,dy.\]
  
    The Laplace transform of $p_t$ is known to be
\begin{align*}
\mathcal{L}p(s) &= {1 \over \sqrt{(\lambda +s)^2 -\lambda^2}}+{\lambda \over \sqrt{(\lambda +s)^2 -\lambda^2}\big(\lambda +s +\sqrt{(\lambda +s)^2 -\lambda^2}\big)}\\
&={1 \over s}-{\lambda \over s\big(\lambda +s +\sqrt{(\lambda +s)^2 -\lambda^2}\big)} \quad ={1 \over s}-{1 \over s}\mathcal{L}\ell(s).
\end{align*}
Thus, by (\ref{recursionLaplace}),
\begin{equation}\label{eq:pL}
\mathcal{L}p(s)={1 \over 2s}+{1 \over \lambda}\mathcal{L}\ell(s)-{1 \over 2s}\big(\mathcal{L}\ell(s)\big)^2.
\end{equation}

Hence, the Laplace transform of $F_{\h_1+\h_2}(t)$ is
  \begin{align*}
  \mathcal{L}F_{\h_1+\h_2}(s)&={1 \over s}\int\limits_0^\infty e^{-st}\ell \ast \ell(t) \,dt={1 \over s}\big(\mathcal{L}\ell(s)\big)^2\\
  &={1 \over s}+{2 \over \lambda}\mathcal{L}\ell(s)-2\mathcal{L}p(s).
  \end{align*}

  Therefore,
  $$F_{\h_1+\h_2}(t)=1+{2 \over \lambda}\ell(t)-2p_t$$
  and
  \begin{align}\label{eq:case3}
   {\sf P}({\sf deg}_T(x)=3 ~\text{ and }~\h_1 \vee \h_2 \leq ~t<~ \h_1+\h_2) &={1 \over 2}\big(1 -p_t\big)^2 -{1 \over 2}F_{\h_1+\h_2}(t) \nonumber \\
   &= {p_t^2 \over 2}-{1  \over \lambda}\ell(t).
  \end{align}

\bigskip  
  
  \item[Case IV]  
  
  \begin{align}\label{eq:case4}
  {\sf P}({\sf deg}_T(x)=3 &~\text{ and }~\h_1 +\h_2 \leq ~t<~ X+\h_1+\h_2)\nonumber \\
  &={1 \over 2}{\sf P}(\h_1 +\h_2 \leq ~t<~ X+\h_1+\h_2~|{\sf deg}_T(x)=3)\nonumber \\
  &={1 \over 2}\int\limits_0^t e^{-\lambda(t-y)} \ell \ast \ell(y) \,dy \nonumber \\
  &={1 \over 2\lambda}\phi_\lambda \ast \ell \ast \ell(t) \quad ={1  \over \lambda}\ell(t) - {1 \over 2\lambda}\phi_\lambda(t)
  \end{align}
  by (\ref{recursionEll}).

\bigskip
  
  \item[Case V]  
  
  \begin{equation}\label{eq:case5}
  {\sf P}({\sf deg}_T(x)=1 ~\text{ and }~t<~X)={1 \over 2}{\sf P}(t<~X)={1 \over 2}e^{-\lambda t}={1 \over 2\lambda}\phi_\lambda(t).
  \end{equation}
  
\end{itemize}

Alternatively, the sum of the probabilities in cases IV and V can be computed as
$$\int\limits_0^\infty e^{-\lambda y}\ell(t)dy={1  \over \lambda}\ell(t),$$
which is consistent with (\ref{eq:case4}) and (\ref{eq:case5}).

Observe that the probabilities in (\ref{eq:case1}), (\ref{eq:case2}), (\ref{eq:case3}), (\ref{eq:case4}), and (\ref{eq:case5}) add up to
$$p_t={\sf P}\Big(\S(\varphi,T) \not= \phi\Big).$$
\bigskip
\noindent

To prove part (ii), observe that the probabilities in (\ref{eq:case4}) and (\ref{eq:case5}) add up to ${1  \over \lambda}\ell(t)$, while the probabilities in (\ref{eq:case3}), (\ref{eq:case4}), and (\ref{eq:case5}) add up to ${p_t^2 \over 2}$. 
Thus the fraction of leaves with single sink is 
${1  \over \lambda}\ell(t)\Big/ {p_t^2 \over 2}$.

By construction, each single mass at a leaf has mass $2t$. 
For a double mass, (\ref{eq:case3}) implies the following cumulative distribution function for positive $a$ and $b$ satisfying $a\vee b ~\leq 2t ~< a+b$.
\begin{align*}
F&(a,b) \\ 
&={{\sf P}\left({\sf deg}_T(x)=3,~\h_1 \leq a/2, ~\h_2 \leq b/2,~t< \h_1+\h_2~|\S(\varphi,T) \not= \phi\right) \over {p_t^2 \over 2}-{1  \over \lambda}\ell(t)}\\
&= {\int\limits_0^{a/2} {\sf P}\left(t-y < \h_2 \leq b/2~|~\S(\varphi,T) \not= \phi ~\text{ and }~{\sf deg}_T(x)=3 \right)\, \ell(y) \, dy \over p_t^2 -{2  \over \lambda}\ell(t)}\\
&= {\int\limits_0^{a/2} \Big(p_{t-y}-p_{b/2}\Big)\, \ell(y) \, dy \over p_t^2 -{2  \over \lambda}\ell(t)}\\
&= {\int\limits_0^{a/2} p_{t-y}\ell(y) \, dy  -p_{b/2}\big(1-p_{a/2}\big) \over p_t^2 -{2  \over \lambda}\ell(t)}.
\end{align*}
Differentiating, we obtain the statement of part (iii):
$$f(a,b)={\partial^2 \over \partial a \partial b}F(a,b)=
{{1 \over 4}\ell(a/2)\ell(b/2) \over p_t^2 -{2  \over \lambda}\ell(t)}.$$

\bigskip
\noindent
For part (iv) observe that by (\ref{eq:case2}),
$${\sf P}({\sf deg}_T(x)=3 ~\text{ and }~\h_1 \wedge \h_2 \leq ~t<~ \h_1 \vee \h_2 ~|\S(\varphi,T) \not= \phi)=1-p_t.$$
 Thus each edge in $\S(\varphi,T)$ is partitioned into subintervals whose lengths are independent exponential random variables with parameter $\lambda$. 
 The number of the subintervals is a geometric random variable with parameter $p_t$. 
 At every point that separates a pair of adjacent subintervals 
 there exists a mass, which can
 have either left or right orientation independently with probability $1/2$.
 
 \noindent
 Finally, for $a \in (0,2t)$, (\ref{eq:case2}) implies
 \begin{align*}
 &{{\sf P}({\sf deg}_T(x)=3 ~\text{ and }~\h_1 \wedge \h_2 \leq ~a/2 <~t<~ \h_1 \vee \h_2~|\S(\varphi,T) \not= \phi) \over p_t(1 -p_t)}\\
 &={{\sf P}(\h_1  \leq ~a/2 <~t<~ \h_2~|~\S(\varphi,T) \not= \phi ~\text{ and }~{\sf deg}_T(x)=3) \over p_t(1 -p_t)}\\
 &={p_t(1 -p_{a/2}) \over p_t(1 -p_t)} \quad ={1 -p_{a/2} \over 1-p_t}.
 \end{align*}
Next, we differentiate to obtain the p.d.f. for the mass of an interior sink, as in part (v),
$${d \over da}{1 -p_{a/2} \over 1-p_t}={\ell(a/2) \over 2(1 -p_t)}.$$

\end{proof}

\section{Random sink in an infinite exponential potential}
\label{sec:rand_mass}
Here we focus on the dynamics of a random sink
in the case of a negative exponential excursion potential.
To avoid subtle conditioning related to a finite potential from $\widetilde{\E}([a,b])$,
we consider space $\widetilde{\E}(\mathbb{R})$ of infinite potentials.
Specifically, we consider here an infinite exponential potential $\Psi^{\rm exp}_0(x)$,
$x\in\mathbb{R}$, constructed as follows.
Let $x_i$, $i\in\mathbb{Z}$ be the epochs of a Poisson point process on $\mathbb{R}$ with rate $\lambda/2$,
indexed so that $x_0$ is the epoch closest to the origin.
The initial velocity $v(x,0)$ is a piece-wise constant continuous function 
that alternates between values $\pm1$ within the intervals 
$(x_i-1,x_i]$ and with $v(x_0,0)=1$. 
Accordingly, the initial potential $\Psi^{\rm exp}_0(x)$ is a piece-wise 
linear continuous function with a local minimum at $x_0$ and alternating slopes 
$\pm1$ of independent exponential duration.
The results in this section refer to the sink $\cM_0$ with 
initial Lagrangian coordinate $x_0$.
We refer to $\cM_0$ as a {\it random sink}, using
translation invariance of Poisson point process.

\begin{figure}[h]
	\centering
	\includegraphics[width=0.7 \textwidth]{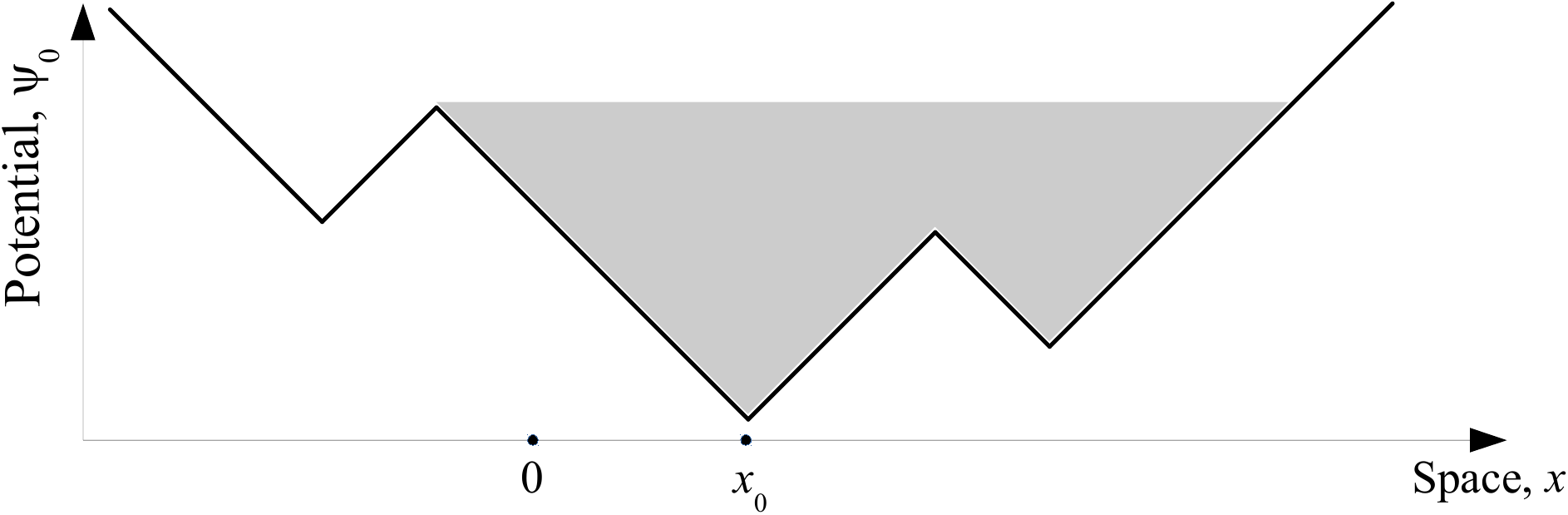}
	\caption{Random sink $\cM_0$. 
	The particle $\cM_0$ originates at point $x_0$ -- the local minimum 
	closest to the origin.
	Its dynamics during the time interval $[0,t]$ is completely specified 
	by the finite negative excursion
	$\cB_0^t$ highlighted in the figure.}
	\label{fig:rand_mass}
\end{figure}

Observe that for any fixed $t>0$, the dynamics of $\cM_0$ is completely
specified by a finite excursion within $\Psi^{\rm exp}_0(x)$.
For instance, one can consider the shortest negative excursion of $\Psi^{\rm exp}_0(x)$
within interval $\cB_0^t$ such that
$x_0\in\cB_0^t$, $|\cB_0^t|>2t$, and one end of $\cB_0^t$ is a local
maximum of $\Psi^{\rm exp}_0(x)$ (see Fig.~\ref{fig:rand_mass}).
The respective Harris path is an exponential Galton-Watson tree ${\sf GW}(\lambda)$.  
Lemma~\ref{lem:SWT} implies that
the dynamics of $\cM_0$ consists of alternating
intervals of mass accumulation (vertical segments of $\cG^{(x,\psi)}$) and
motion (horizontal segments of $\cG^{(x,\psi)}$), starting with a mass accumulation
interval.
Label the lengths $\v_i$ of the vertical segments and the lengths $\h_i$ of the
horizontal segments in the order of appearance in the 
examined trajectory. 
Corollary~\ref{cor:GW} implies that $\v_i,\h_i$ are independent;
the lengths of $\v_i$ are independent
identically distributed exponential random variables with parameter $\lambda$;
and the lengths of $\h_i$ equal the total lengths of independent
Galton-Watson trees ${\sf GW}\left(\lambda\right)$.
This description, illustrated in Fig.~\ref{fig:mass}, allows us to find
the mass dynamics of a random sink.

\begin{figure}[h]
	\centering
	\includegraphics[width=0.7 \textwidth]{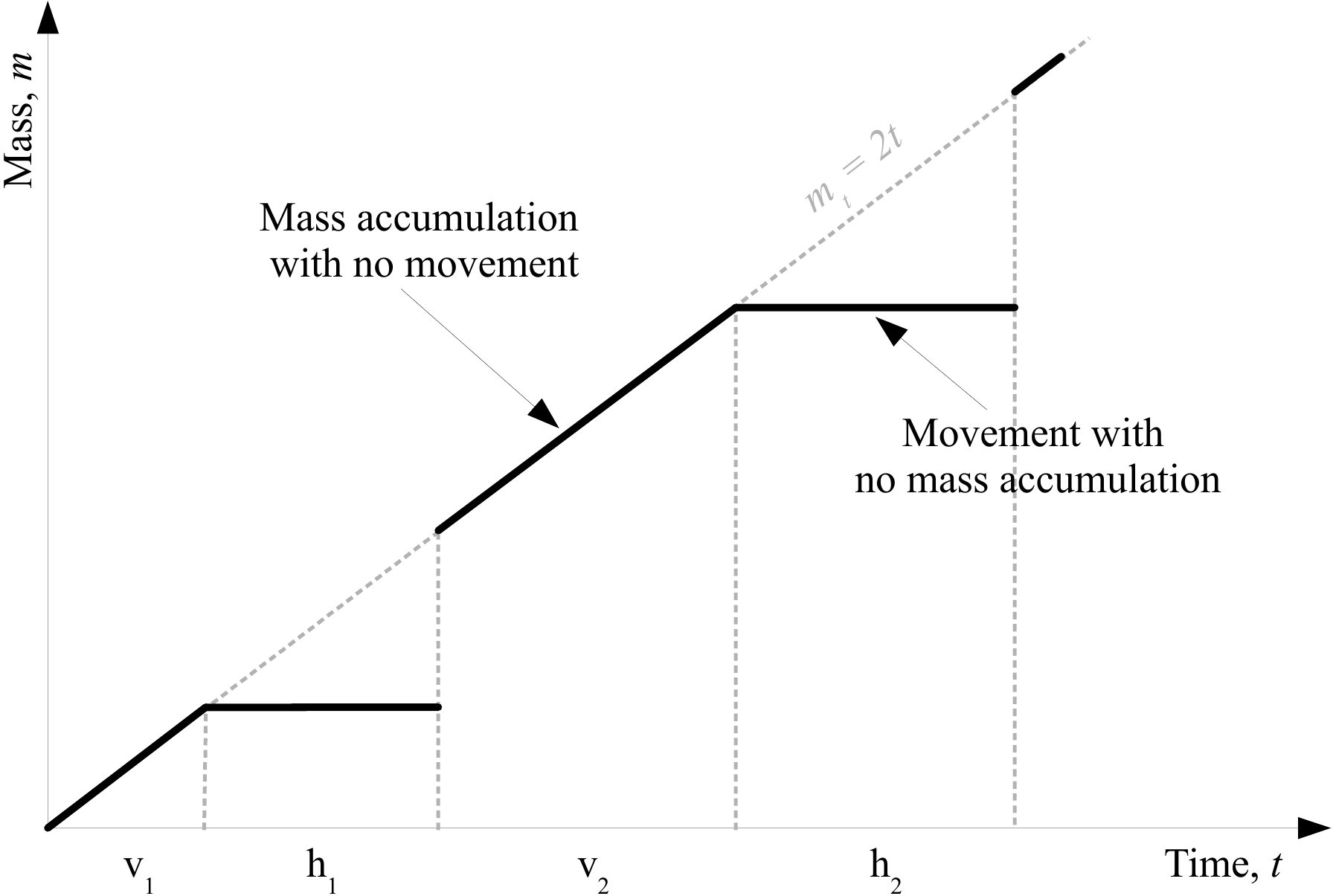}
	\caption{Dynamics of a random sink: an illustration. 
	The trajectory of a sink is partitioned into alternating 
	intervals of mass accumulation of duration $\v_i$ and intervals 
	of movement with no mass accumulation of duration $\h_i$. 
	Each $\v_i$ is an exponential random variable with parameter $\lambda$.
	Each $\h_i$ is distributed as the total length of a critical
	Galton-Watson tree with exponential edge lengths with parameter $\lambda$. }
	\label{fig:mass}
\end{figure}

\begin{thm}[{{\bf Growth probability of a random sink}}]
\label{thm:rand_mass}
The probability $\xi(t)$ that a random sink $\cM_0$ is growing at a given 
instant $t>0$ (that is, it is at rest and accumulates mass) is given by 
\begin{equation}\label{eq:growth2}
\xi(t)=e^{-\lambda t}I_0(\lambda t).
\end{equation}
\end{thm}
\begin{proof}
Let $\v_i$, $i\ge 1$ be independent exponential random variables with 
parameter $\lambda$, and
$\h_i$, $i\ge 1$ be the total lengths of independent ${\sf GW}\left(\lambda\right)$ trees.
The sum $\v_1+\dots+\v_k$ has the gamma density $\gamma_{\lambda,k}(x)={1 \over (k-1)!}\lambda (\lambda x)^{k-1}e^{-x}~$.
The probability $\xi(t)$ that a random sink is growing  
at a given instant $t> 0$ is
\begin{align} \label{eq:growth}
\xi(t) &:= {\sf P}(\text{ a random sink is growing at instant }t~) \nonumber \\
& = \sum\limits_{k=1}^\infty {\sf P}\left(\sum\limits_{i=1}^{k-1} [\v_i+\h_i] < t <\v_k+\sum\limits_{i=1}^{k-1} [\v_i+\h_i] \right) \nonumber \\
& = \sum\limits_{k=1}^\infty \int\limits_0^t \left(~\int\limits_{t-x}^\infty \lambda e^{-\lambda y} dy \right)  \gamma_{k-1}\ast \ell_{k-1}(x) dx \nonumber \\
& = \sum\limits_{k=1}^\infty \int\limits_0^t e^{-\lambda (t-x)}  \gamma_{k-1}\ast \ell_{k-1}(x) dx  ~~= {1 \over \lambda} \sum\limits_{k=1}^\infty \gamma_k\ast \ell_{k-1}(t),
\end{align}
where $$\ell_k(x)=\underbrace{\ell \ast  \hdots \ast \ell}_{k \text{ times}}(x).$$

\medskip
\noindent
We calculate the Laplace transform $\mathcal{L}\xi(s)$ of the probability $\xi(t)$ in (\ref{eq:growth}) as follows. We use the formula for the Laplace transform of $\ell(x)$ derived in (\ref{eq:LaplaceL}) and (\ref{eq:growth}) to obtain
\begin{align}\label{eq:LTvarphi}
\mathcal{L}\xi(s) &= {1 \over \lambda} \sum\limits_{k=1}^\infty \left({\lambda \over \lambda +s}\right)^k  \Big(\mathcal{L}\ell(s)\Big)^{k-1}={1 \over \lambda +s-\lambda \mathcal{L}\ell(s)} \nonumber\\
&={1 \over \lambda +s-{\lambda^2 \over \lambda + s+\sqrt{(\lambda +s)^2 -\lambda^2}}} ={1 \over \sqrt{(\lambda +s)^2 -\lambda^2}}.
\end{align}
Finally, we use formula 29.3.93 in \cite{AS1964} to invert the Laplace transform in (\ref{eq:LTvarphi}), and obtain
\[\xi(t)=e^{-\lambda t}I_0(\lambda t).\]
\end{proof}

\begin{thm}[{{\bf Mass distribution of a random sink}}]
\label{thm:mass}
The mass of a random sink $\cM_0$ at instant $t>0$ has
probability distribution
\begin{align}
\label{mass}
\mu_t(a)
= {\bf 1}_{(0,2t)}(a) &\cdot {\lambda \over 2} e^{-\lambda t}\Big[ I_0\big(\lambda (t-a/2)\big) + I_1\big(\lambda (t-a/2)\big) \Big] \cdot I_0(\lambda a/2)\nonumber\\
&+e^{-\lambda t}I_0(\lambda t) \delta_{2t}(a),
\end{align}
where $\delta_{2t}$ denotes Dirac delta function (point mass) at $2t$.
\end{thm}
\begin{proof}
Let $m_t$ denote the mass of a random sink at a fixed instant $t>0$. 
When the sink is not growing, its mass $m_t$ is strictly smaller than $2t$. 
Then for any positive $a < 2t$,
\begin{align*} 
{\sf P}(m_t \leq a ) & = \sum\limits_{k=1}^\infty {\sf P}\left( -\h_k+\sum\limits_{i=1}^{k} [\v_i+\h_i] ~\leq {a \over 2} < t < \sum\limits_{i=1}^{k} [\v_i+\h_i]\right) \nonumber \\
& = \sum\limits_{k=1}^\infty \int\limits_0^{a/2} \left(~\int\limits_{t-x}^\infty \ell(y)dy \right)  \gamma_k\ast \ell_{k-1}(x) dx,
\end{align*}
and the corresponding density will be
\begin{align*}
{d \over da} &{\sf P}(m_t \leq a ) ={1 \over 2}\int\limits_{t-a/2}^\infty \ell(y)dy \cdot \sum\limits_{k=1}^\infty \gamma_k\ast \ell_{k-1}(a/2)\\
&={\lambda \over 2}\int\limits_{t-a/2}^\infty \ell(y)dy \cdot \xi(a/2)\\
&={\lambda \over 2} e^{-\lambda (t-a/2)}\Big[ I_0\big(\lambda (t-a/2)\big)+ I_1\big(\lambda (t-a/2)\big) \Big] \cdot e^{-\lambda a/2}I_0(\lambda a/2)\\
&={\lambda \over 2} e^{-\lambda t}\Big[ I_0\big(\lambda (t-a/2)\big) + I_1\big(\lambda (t-a/2)\big) \Big] \cdot I_0(\lambda a/2)
\end{align*}
by (\ref{eq:pDelta1}) and (\ref{eq:growth2}).
Thus, the distribution of the mass of a random sink at instant $t$ is given by
\begin{align*}
\mu_t(a)={\bf 1}_{(0,2t)}(a) & \cdot {d \over da} {\sf P}(m_t \leq a )+\xi(t) \delta_{2t}(a)\\
= {\bf 1}_{(0,2t)}(a) &\cdot {\lambda \over 2} e^{-\lambda t}\Big[ I_0\big(\lambda (t-a/2)\big) + I_1\big(\lambda (t-a/2)\big) \Big] \cdot I_0(\lambda a/2)\\ 
&+e^{-\lambda t}I_0(\lambda t) \delta_{2t}(a),
\end{align*}
where $\delta_{2t}$ denotes Dirac delta function (point mass) at $2t$.
\end{proof}

\section{Discussion}
\label{sec:discussion}
This paper introduces a generalized dynamical pruning of binary rooted trees
(Section~\ref{sec:pruning}) that encompasses several pruning operations discussed
in the probability literature, notably including the tree erasure from leaves at 
a constant rate of Example~\ref{ex:height}
\cite{Neveu86,Evans2005,Winkel2012} 
and Horton pruning of Example~\ref{ex:H} \cite{Pec95,BWW00,KZ18}.
Curiously, these two examples seem to be the only cases when the pruning 
satisfies the semigroup property (either in discrete or continuous time).
The other natural pruning operations, like pruning by the total tree length  
(Example~\ref{ex:L}) or by the number of leaves (Example~\ref{ex:numL}),
do not have the semigroup property.
The absence of semigroup property is related to the existence of discontinuities of a respective 
pruning function $\varphi(T)$ along a tree $T\in\L$.
It would be interesting to find the necessary and sufficient conditions on $\varphi(T)$
for the existence or absence of the semigroup property.

The presented results naturally complement an existing modeling framework for 
finite binary self similar trees \cite{KZ17ahp,KZ18a,KZ18}; and are tailored for 
a particular application considered in this work (Sect.~\ref{sec:annihilation}).
However, the generalized dynamical pruning is readily applicable to more 
general $\mathbb{R}$-trees on uncountable spaces as discussed in Sect.~\ref{sec:Rtree}.
For instance, continuum ballistic annihilation with a general continuous potential is a natural 
object to be studies in a real tree framework.

Measures invariant with respect to the generalized pruning (Section~\ref{sec:pi},
Definition~\ref{def:pi}) seem to be abundant on $\L$.
A natural example is a critical binary Galton-Watson tree ${\sf GW}(\lambda)$ 
with i.i.d. exponential edge lengths,
a traditional subject of invariance studies, that is shown here to be prune invariant
under an arbitrary choice of the pruning function 
(Section~\ref{sec:PI}, Theorem~\ref{main}).
The work \cite{KZ18} shows how to construct a variety of measures invariant
with respect to the Horton pruning (on combinatorial trees); 
it is very likely that this approach, modified to include edge lengths, can be 
used for generating measures invariant with respect to
other versions of the generalized dynamical pruning.
An interesting problem is finding measures invariant with respect to
multiple versions of pruning.
At the moment the only known solution is the exponential 
critical binary Galton-Watson tree ${\sf GW}(\lambda)$, invariant with respect to all admissible prunings.
It seems that a family of critical Tokunaga trees, which is shown in \cite{KZ18}
to be invariant with respect to the Horton pruning, is a promising candidate 
to be invariant with respect to other prunings.

Pruning might play a role in the analysis of dynamical systems, including the
problem of finding self-similar or time-invariant solutions. 
As a motivational example, we show here (Section~\ref{sec:annihilation})
that the dynamics of a ballistic annihilation model $A+A\to\zeroslash$, well known
in the physics literature, is equivalent to the generalized dynamical pruning 
of a level set tree representation of the model potential (Section~\ref{sec:Bdyn}, Theorem~\ref{thm:pruning}).
This tree-based representation of the model dynamics opens a way for a complete probabilistic
description of model solutions (Section~\ref{sec:BDEE}, Theorem~\ref{thm:annihilation}),
and finding the time evolution of selected statistics (Section~\ref{sec:rand_mass},
Theorem~\ref{thm:rand_mass}).
It seems promising to expand the proposed analysis to other initial potentials,
as well as to other particle systems known to be critically dependent on the shock 
dynamics.
We observe, for instance, that the dynamics of particles in the ballistic annihilation 
model, prior to particle collision, coincides with that in the famous 1-D inviscid Burgers equation 
$$\partial_t v(x,t)+v \partial_x v(x,t)=0,\quad 
x\in\mathbb{R}, t\in\mathbb{R}_+,$$
which also describes the evolution of a velocity field $v(x,t)$.
The Burgers dynamics appears in a surprising variety of problems, ranging from cosmology 
to fluid dynamics and vehicle traffic models; see \cite{BK07,FrischBec,Gurbatov}
for comprehensive review.
The Burgers equation produces 
shock waves that correspond to the appearance of massive particles and can
be represented as trees (analogously to the shock waves of sinks considered in this work).
It would be interesting to explore if there exists a manageable relation 
between the initial Burgers velocity field $v(x,0)$ and the respective shock wave tree,
at least for special initial velocity fields, and if such trees
are prune invariant for a suitable pruning function $\varphi(T)$.

\section*{Acknowledgements}
We are grateful to Maxim Arnold for numerous discussions related 
to this work and to
Ed Waymire for his continuing support and encouragement. 
This research is supported by the NSF awards DMS-1412557 (Y.K.) and EAR-1723033 (I.Z.).


\end{document}